\documentclass[11pt, english, a4paper]{article}

\usepackage[full]{textcomp}
\usepackage[osf]{newtxtext}
\usepackage[left=3.5cm,right=3.5cm,top=3.4cm,bottom=3.4cm]{geometry}
\usepackage{babel, graphicx, textcomp, varioref, amsfonts}
\usepackage{amsthm,amssymb,amsmath,mathrsfs,mathtools,stmaryrd} 
\usepackage{enumitem}
\usepackage{dsfont}
\usepackage{upgreek} 
\usepackage{esint} 
\usepackage{breakurl}

\usepackage[usenames,dvipsnames]{xcolor} 
\usepackage{comment}

\usepackage[hidelinks]{hyperref} 

\usepackage{authblk} 
\usepackage{url}

\theoremstyle{plain}
\newtheorem{thm}{Theorem}[section]

\newtheorem{lem}[thm]{Lemma}
\newtheorem{prop}[thm]{Proposition}
\newtheorem{cor}[thm]{Corollary}
\newtheorem*{metatheorem}{Metatheorem}

\theoremstyle{definition}
\newtheorem{rem}[thm]{Remark}
\newtheorem{defn}[thm]{Definition}
\newtheorem{ass}[thm]{Assumption}
\newtheorem{exa}[thm]{Example}

\makeatletter

\@addtoreset{equation}{section}
\makeatother

\newcommand{\R}{\mathbf{R}} 
\newcommand{\C}{\mathbf{C}} 
\newcommand{\T}{\mathbf{T}} 
\newcommand{\Z}{\mathbf{Z}}
\newcommand{\N}{\mathbf{N}} 
\newcommand{\B}{\mathbb{B}} 
\newcommand{\X}{\mathbf{X}}
\newcommand{\XX}{\mathbb{X}}

\newcommand{\OO}{\mathcal{O}}

\newcommand{\A}{\mathcal{A}} 
\newcommand{\MM}{\mathcal{M}}
\newcommand{\LL}{\mathcal{L}}
\newcommand{\CC}{\mathcal{C}}
\newcommand{\cC}{\mathscr{C}}
\newcommand{\DD}{\mathcal{D}}
\newcommand{\PP}{\pi} 

\newcommand{\dd}{\delta}
\newcommand{\ddh}{\ensuremath{{\delta^{S}}}}
\newcommand{\<}{\langle}
\renewcommand{\>}{\rangle}

\newcommand{\HH}{\mathcal{H}}
\newcommand{\BB}{\mathcal{B}}
\newcommand{\id}{\mathrm{id}}

\newcommand{\E}{\ensuremath{{\mathcal E}}}
\newcommand{\ZZ}{\ensuremath{{\mathcal Z}}}
\newcommand{\I}{\ensuremath{{\mathscr I}}}

\newcommand{\Lip}{\mathrm{Lip}}

\DeclareMathOperator{\Div}{div}

\newlist{todolist}{itemize}{2}
\setlist[todolist]{label=$\square$}
\usepackage{pifont}

\begin{document}
	
	\date{\today}
	\title{Non-autonomous rough semilinear PDEs and the multiplicative Sewing Lemma}
	\author[1]{Andris Gerasimovi\v cs}
	\author[2]{Antoine Hocquet}
	\author[3]{Torstein Nilssen}
	
	\affil[1]{\small University of Bath, Bath, United Kingdom; 
	
	Email: \texttt{ag2616@bath.ac.uk} }
	\affil[2]{\small Technische Universit\"at Berlin, Berlin, Germany;
	
	Email: \texttt{antoine.hocquet86@gmail.com}}
	
	\affil[3]{\small University of Agder, Kristiansand, Norway;
	
	Email: \texttt{torstein.nilssen@uia.no}}

	\maketitle

\begin{abstract}
We investigate existence, uniqueness and regularity for local solutions of rough parabolic equations with subcritical noise of the form
$du_t- L_tu_tdt= N(u_t)dt + \sum_{i = 1}^dF_i(u_t)d\X^i_t$
where $(L_t)_{t\in[0,T]}$ is a time-dependent family of unbounded operators acting on some scale of 
Banach spaces, while $\X\equiv(X,\XX)$ is a two-step (non-necessarily geometric) rough path of 
H\"older regularity $\gamma >1/3.$
Besides dealing with non-autonomous evolution equations, our results also allow for unbounded operations 
in the noise term (up to some critical loss of regularity depending on that of the rough path $\X$). As 
a technical tool, we introduce a version of the multiplicative sewing lemma, which allows to construct the so 
called product integrals in infinite dimensions. We later use it to construct a semigroup analogue for the 
non-autonomous linear PDEs as well as show how to deduce the semigroup version of the usual sewing 
lemma from it.\\[.5em]



\end{abstract}
	
	\tableofcontents
	
	\section{Introduction}
	\subsection{Motivations}
	
	We are interested in non-autonomous semilinear evolution equations of the form
	\begin{equation}\label{evolution_equation}
	\begin{cases}
	du_t= (L_tu_t + N(u_t))dt + \sum\limits_{i =1}^d F_i(u_t)d\X^i_t,\quad t\in[0,T],
\\
u_0  \, \; = x\in \BB,
	\end{cases}
	\end{equation} 
	where the unknown $u$ is continuous with values in some Banach space $\BB$, $(L_t)_{t\in[0,T]}$ is a continuous in time family of unbounded operators satisfying a suitable sector condition, $X\colon[0,T]\to \R^d$ is a path of H\"older regularity $\gamma >1/3,$ while $N$ and $F$ denote some non-linearities. 
	
	Parabolic equations like \eqref{evolution_equation} appear in a stochastic context (e.g.\ in filtering theory), 
	where $X$ denotes for instance a finite-dimensional Brownian motion. They were first investigated in 
	the late 70's through the work of Pardoux, Krylov and Rozovskii 
	\cite{krylov1981stochastic,pardoux1980stochastic}, using an appropriate functional setting in which It\^o 
	calculus can be used, together with monotonicity arguments. In the autonomous case, the so called `mild 
	approach' was largely developed by Da Prato and his school, which culminated in the monograph~\cite{DPZ}. 
	
	In a `rough path' context, the treatment of equations similar to \eqref{evolution_equation} originates in the mild approach by Gubinelli, Lejay and Tindel \cite{GLT2006} for Young paths, later extended to rough paths by 
Gubinelli and Tindel in \cite{gubinelli2010} (see also the works of Deya 
\cite{deya2012numerical,deya2011rough,deya2012rough}). Parallel to that, a viscosity formulation following 
ideas from Lions and Souganidis \cite{lions1998fully,lions1998fully4} was proposed by Caruana, Friz, 
Oberhauser \cite{caruana2011rough,friz2014rough}.
Recently, a variational approach to evolutionary rough PDEs with transport noise was introduced by 
Gubinelli, Deya, Hofmanov\'a and Tindel \cite{DGHT} (see also Gubinelli and Bailleul \cite{BG}). Modelled on 
Sobolev spaces, their notion of solution is `intrinsic', in the sense that they work directly at the level of a 
rough evolution equation, avoiding the use of flow transforms.
The mild approach gained some renewed interest in \cite{gerasimovics2018}, where it was used to prove a 
H\"ormander type theorem for degenerate SPDEs.

	When $X=W$ is a Brownian Motion in $\R^d$ (the case of a Wiener Process defined on some abstract Wiener space could be carried out by letting $d=\infty$ in the sum below), and $L_t$ is a deterministic family of operators, one defines a mild solution by the Duhamel Formula
	\begin{equation}
	\label{mild_rep}
	u_t-S_{t,0}x = \int _{0}^tS_{t,r}N(u_r)dr+ \sum_{i=1}^d \int _0^t S_{t,r}F_i(u_r)dW_r^i,
	\quad t\in[0,T],
	\end{equation}
where the stochastic integral is understood in the It\^o sense, $S_{t,s} \in \LL(\BB,\BB)$ for $0\leq s\leq t\leq T$ is the propagator associated to $(L_t)$, that is, $S_{t,s}x$ denotes the solution $v$ evaluated at time $t$ (if it exists and is unique), of the linear equation
\begin{equation}
\label{linear}
\partial_t v = L_t v_t\quad \text{on}\enskip (s,T],
\quad
v_s:=x\in \BB\,.
\end{equation}
As natural as it may look, the formula \eqref{mild_rep} is meaningful only when $L_t$ is a deterministic family.
Even assuming adaptedness of $L_t$ does not guarantee that formula \eqref{mild_rep} is meaningful, since
in that case the stochastic integrand $S_{t,r}F_i(u_r)$ is only $\mathcal F_t$-measurable (and not $\mathcal F_r$-measurable, as required in standard It\^o calculus).
This anticipative nature causes major technical difficulties in stochastic contexts where a 
representation like \eqref{mild_rep} is needed. The situation of random $L_t$ is even more problematic since a solution constructed using Skorohod integral (which would be a natural candidate) is not, generally speaking, a weak solution of \eqref{evolution_equation} (see \cite{leon1990equivalence}). A solution was nonetheless provided by Leon and Nualart in \cite{leon1998stochastic}, relying on the formalism of `forward integral' introduced earlier by Russo and Vallois in \cite{russo1993forward}. Despite being quite successful in the non-autonomous, semilinear scenario, 
their approach does not seem to easily extend to broader contexts such as quasilinear equations.

The situation when the operator $L_t$ is random arises naturally, for instance in the context of non-degenerate quasilinear SPDEs of the form
\begin{equation}\label{quasilinear_ansatz}
	du_t= A(u_t)u_t dt + \sigma(u_t)dW_t,\quad t\in(0,T],
	\quad u_0\in\BB\,,
\end{equation}
where to prove existence and uniqueness one is led to investigate a fixed point for the map 
\begin{equation*}
	v=(v_t(\omega,x))\, \longmapsto \left (t\mapsto \,S^{v}_{t,0}u_0 + \int_0^t S^{v}_{t,s}\sigma(u_r)dW_s \right )\,.
\end{equation*}
Here $S^{v}_{t,s}$ denotes the random propagator generated by the elliptic part $L_t=A(v_t)$, and is therefore only $\mathcal F_t$-measurable.
In a deterministic setting, this functional analytic approach is due to Amann 
\cite{amann1986quasilinear}.
Based on an alternative (non-anticipative) formulation for \eqref{quasilinear_ansatz}, we note that similar quasilinear stochastic equations were investigated in the recent work \cite{kuehn2018pathwise}, though it remains unclear how to translate Amann's method to a stochastic context.

Another motivation for considering random generators lies in the study of infinite-dimensional Kolmogorov equations.
A concrete example is given by the weak error analysis for the following SPDE
\begin{equation}\label{general_spde}
dX_t^x= AX_t^x dt + \sigma(X_t^x),\quad X_0^x=x\in H=L^2(0,1)\,,
\end{equation}
seen as an evolution equation in the Hilbert space $H$, where the superscript in $X^x$ emphasizes the dependency of the solution w.r.t.\ the initial datum. Here $A$ is a uniformly elliptic operator (for instance the Dirichlet Laplacian) and $\sigma$ denotes a Lipschitz continuous Nemytskii operator.
In \cite{brehier2018kolmogorov} the authors use that the directional derivative $\langle Du_t^x ,h\rangle$ of the functional $u(t,x)=\mathbb E[\varphi(X_t^x )]$ for any bounded and continuously differentiable $\varphi \colon H \to \R$, is given formally for each $h\in H$ by the anticipative integral
\begin{equation}\label{formula_brehier}
\langle Du_t^x ,h\rangle =\int _0^t\langle S_{t,s}\sigma'(X_t^x ),e^{sA}h dW_s\rangle
\end{equation}
where this time $S_{t,s}$ is the propagator associated with the formal generator `$L_t:=A+\sigma(X_t)\frac{dW_t}{dt}$'.
Otherwise said, $S_{t,s}\colon H\to H$ is the solution map $z\mapsto Z^s_t$
for the linear equation
\[
dZ^s_{t}=AZ^s_{t}dt +\sigma'(X_t^x )Z^s_tdW_t,\quad \text{on}\enskip (s,t]\,,
\quad Z^s_s=z\,.
\]
The representation \eqref{formula_brehier} is a crucial tool in obtaining estimates on the solutions of Kolmogorov equations (see \cite[Sec.~5]{brehier2018kolmogorov}).
These estimates are then needed to address the weak order analysis of an explicit Euler scheme for \eqref{general_spde}.
We remark that the authors of the previous paper cannot use \eqref{formula_brehier} and have to hinge their analysis on a discretized version of \eqref{general_spde}. Though it is not our intention here to discuss the meaning of the ill-defined product in the corresponding formal generator, we note that the tools developed in Section \ref{sec:controlled} allow us to make perfect sense of the integral \eqref{formula_brehier}, and to obtain estimates for it.

Our main objective in this paper is to build quite a broad framework for a pathwise mild solution theory 
to semilinear SPDEs of the form \eqref{evolution_equation}. We do this through the theory of rough paths 
introduced by Lyons \cite{lions1998fully}, in a spirit that is similar to Gubinelli's controlled rough path 
approach \cite{gubinelli2004}. In keeping with this view, we will look for solutions $u$ 
of~\eqref{evolution_equation} that `locally look like $X$'. Loosely speaking, $(u, u')$ will be said to be 
controlled by $X$ if 
\begin{align*}
	u_t - u_s = u'_s \cdot (X_t-X_s) + R_{t,s},
\end{align*}
where $R_{t,s} = O(|t-s|^{2\gamma})$ and $\gamma > 1/3$ is the H\"older regularity of $X$. The main 
difference with~\cite{gubinelli2004} is that here the above control happens not in the space $\BB$ where the 
initial datum lies, but in some larger space. Indeed, as seen from \eqref{mild_rep} one cannot expect $u_t$ to 
be H\"older continuous in that space, since even the solution to the corresponding linear equation is not. 
One could potentially work in spaces with a weight at time $0$ (see \cite{hesse2019local} for the attempt in 
that direction, though authors can only consider $F_i$ that improves spatial regularity), but going to 
the larger space allows sticking very closely to the classical theory of controlled rough paths. We nevertheless 
point out that, using the smoothing properties of the propagator $S_{t,s},$ we will show that the solution 
$u$ is indeed continuous in time with values in the original space $\BB$.

One of the main advantages of the pathwise approach as it was numerously shown in finite dimensions, is 
that it allows to show that the solution depends continuously on the noise and the initial condition. This is in 
strong contrast to the approach using It\^o calculus, where in general only measurability of the solution map 
is available. Another advantage is that the notion of the rough integral that we will use is deterministic in its 
nature and therefore does not rely on the adaptedness of the integrands. This, in fact, allows to treat 
equation~\eqref{evolution_equation} with random $L_t$ in a less technical way than the one of Leon and 
Nualart. We will also show that in the case of adapted $L_t,$ our mild solutions are also weak 
It\^o solutions when $X$ is the Wiener process.
Finally, the rough path approach allows to consider equation~\eqref{evolution_equation} with a Gaussian 
noise which is not a martingale, in particular allowing to treat the case of a fractional Brownian motion with 
Hurst parameter less than $1/2$.

Our main result in this part is to state sufficient conditions on $N$ and $F$ under which existence, 
uniqueness and continuity of the solution map holds for \eqref{evolution_equation}. It can be loosely 
formulated as follows
(precise statements will be given in Theorems \ref{thm:main}, \ref{thm:weak}, \ref{RPDE}, \ref{stab_sol} and \ref{weak}).
\begin{metatheorem}
Let $(t \mapsto L_t)$ be a H\"older continuous, sectorial family of linear operators acting on a scale of Banach 
spaces $(\BB_\alpha )_{\alpha \in \R},$ in such a way that $L_t:\BB_\alpha \to \BB_{\alpha -1}$ for any $t$ and $\alpha.$
Fix some initial datum $x\in \BB_0.$ Consider nonlinearities $N$ and $F$ such that $N$ is polynomial, 
while $F$ is three times continuously differentiable and subcritical.
There exists a unique, local in time, mild solution to \eqref{evolution_equation} such that $u$ is controlled by $X$ with 
Gubinelli derivative $F(u).$ The solution map depends continuously on the initial condition, as well as on the 
rough path $\X.$ It is also a weak solution in the usual sense.
\end{metatheorem}

Roughly speaking, the above `subcriticality' assumption means that $F,$ seen as a non-linear operator acting 
on the scale $(\BB_\alpha )$ does not cause any loss of regularity greater than the H\"older 
time-regularity of the rough path $\X$. This condition illustrates the subtle interplay between space and time 
regularity that occurs in problems of the form \eqref{evolution_equation}.
An elementary but illustrative example that fits within this solvability theory is the linear equation
\begin{equation}\label{subcritical_example}
du_t=\Delta u_t dt + (-\Delta)^{\sigma }u_t d\X_t,\quad x\in \BB_0:= L^p(\R^n),\quad p\in(1,\infty)\,,
\end{equation}
under the subcritical assumption $\sigma <\gamma,$ where $(\BB_\alpha)$ is the Bessel potential scale (see 
below). In the case of $X$ being a Brownian motion, we see that $\sigma $ can be arbitrary close to $1/2,$ 
which in strength agrees to the well-known similar results obtained using mild formulation and It\^o calculus 
(see for instance \cite{brzezniak2012stochastic} and the references therein).
Note that, in the `super-critical' case $\gamma<1/2=\sigma ,$ there is at least one type of equations that still 
possess a unique solution, which consists in transport noise of the form $F(u_t)=\sigma\cdot \nabla u_t.$ As 
observed by Deya, Gubinelli, and Tindel \cite{DGHT} (see also \cite{HH}), it is possible in this case to prove a 
priori estimates, which in turn allow inferring existence and uniqueness of solutions. As pointed out in 
\cite[Remark 2.4]{HH}, in this case, the assumption that the rough path $X$ be \emph{geometric} is essential, 
in contrast with the subcritical case dealt with in the present work. 
This limitation constitutes another motivation for us to introduce an alternative formulation.
\\

A secondary objective of our manuscript is to introduce a new version of the multiplicative Sewing Lemma, 
which allows constructing product integrals in a rather general class of metric spaces $(\MM, d)$
equipped with a product operation. Product integrals are going to be the limits of the form 
\begin{equation}\label{product integral}
\varphi_{t,s} = \lim_{|\pi| \to 0} \prod_{[v,u] \in \pi} \mu_{v,u}\;,
\end{equation}
where the limit is taken over arbitrary partitions $\pi$ of $[s,t]$ with their mesh size going to zero. The 
novelty of our version is that it will apply to various cases where the limit lives in an infinite dimensional 
space. The main difference from the usual additive sewing lemma is that the product here is 
non-commutative in general. Note that a version of multiplicative sewing lemma for the non-commutative 
products has already been introduced by Feyel, De La Pradelle and Mokobodzki \cite{feyel2008non} (see also \cite{coutin2014perturbed,brault2019nonlinear,bailleul2019rough} for related works). We 
point out, however, that Theorem \ref{thm:MSL} below is, in essence, independent of the latter, since the 
assumptions and conclusions are different. In \cite{feyel2008non}, the authors have to assume that the 
function $|\cdot|:\MM\to\R_+$ which controls the `size' of $\mu_{t,s}$ in~\eqref{product integral}, is 
Lipschitz continuous with respect to the distance $d$. This is mostly not going to be the case for the 
examples treated below, since in infinite dimensions $|\cdot|$ is often going to be a norm which is stronger 
than $d$. For a further discussion, see Remark \ref{rem:feyel}.

We will show how to use this version of the Sewing lemma to give a precise meaning to rough equations of 
the form \eqref{mild_rep}. Though this particular step was already achieved in the previous works 
\cite{gubinelli2010,deya2011rough,gerasimovics2018}, the merit of our approach is that our integrands belong 
to a class of paths which are independent of the propagator $(S_{t,s})_{(s,t)\in \Delta_2}$ generated by 
$(L_t)_{t\in[0,T]}$. This paves the way to a possible treatment of quasilinear equations where the propagator 
will depend on the solution itself (we will address this problem in future work). It also has a potential 
attractive application to unify Lyons' theory of multiplicative functionals with that of Gubinelli and Tindel, for 
the construction of the rough convolutions 
\begin{equation}
\label{rough_convolution}
z_t:=\int _0^tS_{t,r}F(y_r)\cdot d\X_r\,.
\end{equation} 
A basic role in this construction is played by the \emph{affine group} $\MM,$ defined as the semi-direct product 
\begin{equation}
\label{semi-direct}
\MM:=\mathcal G\ltimes \BB
\end{equation} 
where $\mathcal G\subset \LL(\BB)$ is a group of linear bijections operating on $\BB$ via the natural action $(T,b)\mapsto Tb$.
More precisely, $\MM$ is the set of pairs $(T,b)\in\mathcal G\times \BB$ endowed with the group multiplication
\begin{equation}
\label{semi-direct-product}
(T_1,b_1)\circ (T_2,b_2):= \big(T_1  T_2, b_1 + T_1b_2\big)\,,\quad (T_j,b_j)\in \MM\,,\enskip j=1,2\,.
\end{equation}
In practice, because of the parabolic nature of \eqref{evolution_equation}, it will be necessary to replace $\mathcal G$ by an appropriate set of (non-necessarily invertible) linear operators containing $S_{t,s}$ for $0 \leq s \leq t \leq T$. In this case $\MM$ is only a monoid, but this is sufficient for our purposes.

Surprisingly enough, our version of the multiplicative Sewing Lemma (i.e.\ Theorem \ref{thm:MSL}), is seen to be useful for purposes that are orthogonal to rough paths theory and the construction of $z_t$. In particular, it will allow us to construct the propagators $S_{t,s}$ `by hand' (hence reproving classical results from Kato, Tanabe, and Sobolevskii), or to show a version of the Lie-Trotter product formula, for propagators which are generated by the sum of two dissipative operators. Though these results are well-known in principle, the observation that they could be deduced from a Sewing Lemma perspective is new and could be seen as one of our main contributions. Note that, apart from the new construction of the rough convolution, this first part of the paper is essentially independent of the second.\\

The paper is organized as follows.
In Section \ref{sec:results} we explain our functional analytic setting and present our main results. In particular, we will give an existence and uniqueness statement for \eqref{evolution_equation}, and provide a stochastic example with random coefficients.
In Section \ref{sec:MSL}, we prove a general multiplicative sewing lemma and give some applications in the context of (not necessarily rough) evolution equations.
Though we believe that these applications are interesting by themselves, they merely consist of establishing new proofs of already known results in functional analysis, and therefore their reading could be avoided at first. In Section \ref{sec:controlled}, we introduce the space of controlled paths $\DD_{X,\alpha}^{2\gamma}$ associated to a monotone family $(\BB_\alpha)_{\alpha\in\R}$ of interpolation spaces, and then apply the multiplicative Sewing Lemma in order to construct the `rough convolution' $\int_0^t S_{t,s}y_s \cdot d\X_s,$ where $S$ is the propagator associated to the family $(L_t)_{t\in[0,T]}$ (its existence will be guaranteed by Assumption \ref{ass:L_t}). Similar to \cite{gubinelli2004,gubinelli2010}, it will be seen that $\DD_{X,\alpha}^{2\gamma}$ is a natural space of integrands for which the rough convolution is well-defined. We point out that, though the construction of rough convolutions was already carried out in \cite{gubinelli2010} (see also \cite{gerasimovics2018}), it does require a proof because our definition of the controlled path space is different from the above-mentioned works.
As a natural continuation, Section \ref{sec:evolution} will be devoted to the proof of Theorem \ref{thm:main} where existence, uniqueness, and stability of the solution map are stated for \eqref{evolution_equation}, in the subcritical case. This will be done via a Picard fixed point argument in the controlled path space.
Finally, Section \ref{sec:examples} will be devoted to examples. For completeness, we will also address in that section a supercritical case corresponding to the transport noise. In particular, the equivalence between the mild representation \eqref{mild_rep} and the weak solution theory provided in \cite{HH} will be discussed.

\subsection{General notation}\label{sec:general notation}
Throughout the paper, $T\in(0,\infty)$ is considered as a fixed, positive time horizon.
We denote by $\N:=\{1,2,\dots\}$ and by $\N_0:=\N\cup\{0\}.$
Integers will be denoted by $\mathbf{Z},$ while real numbers (resp.\ complex numbers) will be 
denoted by $\R$ (resp.\ $\C$).
The set $[0,\infty)$ of non-negative real numbers will be denoted by $\R_+$. We also denote by $\T^n = \R^n/\Z^n$ the $n$-dimensional torus. We write $a \lesssim_{\gamma} b$ provided there exists $C = C(\gamma)$ such that $a \leq C b$. When $C$ can be chosen independently of $\gamma$ we write simply $a \lesssim b$. 

If $X,Y$ are Banach spaces, we denote by $\LL(X,Y)$ the space of continuous linear operators from $X$ to 
$Y,$ endowed with the operator-norm topology.
Similarly, we denote by $\LL_s(X,Y)$ the same space as above but equipped with the strong topology. Recall 
that the strong topology is the coarsest topology that makes the maps $X\ni x\mapsto Sx\in Y$ continuous 
when $S$ varies in $\LL(X,Y).$
For simplicity, we write $\LL(X)= \LL(X,X)$ and likewise $\LL_s(X)=\LL_s(X,X)$.

We shall frequently work with the usual Sobolev spaces $W^{\alpha ,p}(\OO,\R^n)$, $p \geq 1$, $\alpha 
\in\R$ and a domain $\OO \subset \R^n,$ equipped with the standard norm $|\cdot |_{W^{\alpha ,p}}$. 
Moreover, we will denote by $W^{\alpha ,p}_0(\OO,\R^n)$ the closure of $\CC^{\infty}_c(\OO,\R^n)$ (the 
smooth, compactly supported functions) in the topology of $W^{\alpha ,p}(\OO,\R^n)$. For notational 
simplicity we write $W^{k,p}(\OO) : = W^{k,p}(\OO,\R)$ and similarly for Bessel potential spaces $H^{k,p}(\OO) = H^{k,p}(\OO,\R)$ (see Example~\ref{ex:Besel} where these spaces are introduced).

We use the symbol $\pi$ to denote a generic partition of $[0,T]$ and we shall sometimes blur the difference 
between thinking of partitions as a set of points and as a set of intervals. 

In the sequel, we will need to introduce various norms and seminorms. The `rule of thumb' is that quantities,
which are only semi-norms, will be denoted by the brackets $[\cdot ]$ while norms will usually be denoted by 
the simple bars $|\cdot |.$ The double bars $\|\cdot \|$ will be used only for the controlled paths spaces 
$\DD_{X,\alpha }^{2\gamma }$ defined in Section \ref{sec:controlled} and for the norms on functions from the Definition~\ref{def:C^k}.

\subsection*{Acknowledgements}

{\small
	The authors would like to thank M.\ Hairer for many useful discussions and comments. The anonymous referees are warmly thanked for their careful reading of the paper, notably for having pointed out several gaps in the first version.
	AG gratefully acknowledges the financial support by the Leverhulme Trust through Hendrik Weber’s Philip Leverhulme Prize. 
	AH and TN gratefully acknowledge the financial support by the DFG via Research Unit FOR 2402.
}

\section{Settings and main results}
\label{sec:results}
\subsection{Functional analytic framework}

We will first define a concept of monotone family of interpolation spaces $(\BB_\alpha, |\cdot|_\alpha)$, which encodes a notion of `spatial regularity' for  \eqref{evolution_equation}.

\begin{defn}\label{def:monotone_family}
A family of separable Banach spaces $(\BB_\alpha,|\cdot|_\alpha)_{\alpha \in \R}$ is a called a \emph{monotone family of interpolation spaces} if for every $\alpha \leq \beta$, $\BB_\beta$ is a continuously embedded, dense subspace of $\BB_\alpha,$ and if the following interpolation inequality holds:
for $\alpha \leq \beta \leq \gamma$ and $x \in \BB_{\alpha}\cap \BB_{\gamma}$:
\begin{equation} \label{interpolation}
|x|^{\gamma-\alpha}_\beta \lesssim |x|^{\gamma-\beta}_\alpha |x|^{\beta-\alpha}_\gamma.
\end{equation}
\end{defn}
The main interest in considering a family as above is the following property, whose proof is evident by interpolation and therefore left to the reader. We write $\Delta_2 = \{ (t,s) : T \geq t \geq s \geq 0\} $.
If $S\colon \Delta_2  \to \LL(\BB_\alpha )\cap \LL(\BB_{\alpha +1})$ is such that for each $x\in\BB_{\alpha +1},$ and any $(t,s)\in\Delta _2,$
$|(S_{t,s} - \id) x |_{\alpha } \lesssim |t-s| |x|_{\alpha +1}$
while
$|S_{t,s}x|_{\alpha +1} \lesssim |t-s|^{-1} |x|_{\alpha },$ then for every $\sigma\in[0,1]$, $S_{t,s}$ belongs to $\LL(\BB_{\alpha + \sigma} )$ and the following estimate holds true:
\begin{equation} \label{semigroup}
| (S_{t,s} - \id) x |_{\alpha} \lesssim |t-s|^{\sigma}  |x|_{\alpha+\sigma}\;,\qquad |S_{t,s}x|_{\alpha+\sigma } \lesssim |t-s|^{-\sigma } |x|_{\alpha}.
\end{equation}

\begin{exa}[Hilbert space setting]
\label{exa:hilbert}
Let $(\HH, |\cdot|)$ be a separable Hilbert space,
		on which we are given a closed densely defined unbounded operator
		$L\colon D(L)\subset \HH\to \HH$ whose resolvent set contains a sector $\Sigma_{\vartheta,\lambda}:=\{\zeta \in \C,\, |\arg(\zeta -\lambda) | < \pi/2+\vartheta\}$ for some $\vartheta>0,$ $\lambda \in \R$, and such that for every $\zeta \in\Sigma_{\vartheta,\lambda} ,$ $|(\zeta - L)^{-1}|\leq \Lambda(1+|\zeta |)^{-1},$ with $\Lambda>0$ independent on $\zeta$. An operator $L$ satisfying this property is called sectorial. 
		
		For $\alpha > 0$ we can define the fractional powers
		$(-L)^{-\alpha}$ through the formula \cite[Eq.\ (6.3)]{pazy1983semigroups}. Next, introduce the space
		\begin{equation}
		\label{interspaces}
			\HH_\alpha :=\mathrm{Im}(-L)^{-\alpha}\subseteq \HH\,,
		\end{equation}
		endowed with the norm $|x|_\alpha = |(-L)^\alpha x|$.
		Additionally, we define $\HH_{-\alpha}$ as the completion of $\HH$ with respect to the norm $|((-L)^{\alpha})^{-1}\cdot|$ (the fractional powers of $L$ are one-to-one thanks to the sector condition). The interpolation inequality \eqref{interpolation} can be proved using spectral decomposition and H\"older inequality (see \cite[Sec~6]{SPDE}).
	\end{exa} 
	
	\begin{exa}
	\label{exa:scale}
		Let $(\BB_k, |\cdot|_k)_{k \in \Z}$ be a monotone family of reflexive Banach spaces, in the sense that for each $k\in\Z,$ $\BB_{k+1}\hookrightarrow\BB_k$ (densely) and \eqref{interpolation} is satisfied for every $\alpha, \beta, \gamma \in\Z\,.$
		Then for $\theta \in [0,1]$ and $k \in \Z$ we can define a space $\BB_{k+\theta}$ by complex interpolation:
		\begin{equation}
		\label{interpolation0}
			\BB_{k+\theta} := \big[\BB_k,\BB_{k+1}\big]_\theta\,.
		\end{equation}
		For the precise definition and properties of complex interpolation spaces see \cite{calderon1964intermediate,lunardi2018interpolation}. 
		With this definition, it can be shown that $(B_\alpha )_{\alpha \in \R}$ is a monotone family of interpolation spaces.
		The reflexivity of $\BB_k$ is necessary in order to guarantee the consistency relation $\big[\BB_k,\BB_{k+1}\big]_0 = \BB_k$ and $\big[\BB_k,\BB_{k+1}\big]_1 = \BB_{k+1}$ implying that for all $\alpha\leq\beta<\gamma$ we have $\BB_\beta = [\BB_\alpha,\BB_\gamma]_{\frac{\beta-\alpha}{\gamma-\alpha}}$ which in turns implies~\eqref{interpolation} by~\cite[Thm~2.6]{lunardi2018interpolation}.
		
		Note that given a Hilbert space $\HH$ and $L$ as in Example \ref{exa:hilbert} which is self-adjoint, one can construct $(\HH_\alpha)_{\alpha \in \R}$ using the so called \textit{Sobolev tower}. First construct $(\HH_k)_{k\in \Z}$ (which do not require fractional powers of $L$) and then define $(\HH_\alpha)_{\alpha\in \R}$ via the complex interpolation as above. If $L$ is self-adjoint then $(\HH_k)_{k\in \Z}$ will be reflexive and \cite[Thm 4.17]{lunardi2018interpolation} guarantees this construction of $(\HH_\alpha)_{\alpha \in \R}$ is equivalent to the one in Example \ref{exa:hilbert}.
		
		Another example, is to take $\BB_k = W^{k,p}(\OO)$ for $k \in \Z$ and $p \in (1,\infty)$ and where $\OO$ is equal to $\R^n$ or $\T^n$, or is a smooth domain inside $\R^n$.
	\end{exa}

\begin{exa}\label{ex:Besel}
Another useful family of Banach spaces consists of the Bessel potential spaces $H^{s,p}(\R^n),$ for $1<p<\infty$ and $\alpha \in \R,$
that can be defined in terms of Fourier transform $\mathcal{F}$ of tempered distributions $\mathcal S'(\R^n)$ as
\[
H^{\alpha,p}(\R^n):= \left \{f\in \mathcal S'(\R^n)\,,\enskip\|f\|_{H^{\alpha,p}}:=\|\mathcal F^{-1}(1+|\xi|^2)^{\alpha/2}\mathcal Ff\|_{L^p}<\infty\right \}\,.
\]
We recall the well known characterization using the complex interpolation $[\cdot,\cdot ]$ (see \cite{lunardi2018interpolation})
\[
H^{k\theta ,p}(\R^n)=[L^p(\R^n),W^{k,p}(\R^n)]_{\theta }\;,
\]
for each $\theta \in[0,1]$ and $k\in \Z.$
Moreover, it holds $W^{\alpha ,p}(\R^n)=H^{\alpha ,p}(\R^n)$ when 
$\alpha $ is an integer, or when $p=2$ for any $\alpha \in \R,$ and for any $\epsilon >0$ we have the continuous 
embedding $W^{\alpha +\epsilon ,p}(\R^n)\hookrightarrow H^{\alpha ,p}(\R^n)\hookrightarrow W^{\alpha -\epsilon ,p}(\R^n)$ for every $\alpha \in \R$ and $p>1$. One can define $H^{\alpha,p}(\T^n)$ similarly through the Fourier transform on the torus and establish analogous properties. 

For a smooth and bounded domain $\mathcal O \subset\mathcal \R^n,$
the space $H^{\alpha,p}(\mathcal O)$ can be defined algebraically as the set of restrictions $u=f|_{\mathcal O}$ of elements $f\in H^{\alpha,p}(\R^n)$. One then defines the norm of $u\in H^{\alpha,p}(\mathcal O)$ to be the infimum of $|f|_{H^{\alpha,p}(\R^n)}$ over all such functions (see \cite[Chap.~3]{triebel1983theory}). We also define $H^{\alpha,p}_0(\OO)$ to be the closure of $\CC^\infty_c(\OO)$ in $H^{\alpha,p}(\OO)$. An important consequence of the above definition is that Bessel spaces respect complex interpolation
\begin{equation}\label{eq:bessel_interpolation}
	H^{\alpha+\theta,p}(\OO) = [H^{\alpha,p}(\OO),H^{\alpha+1,p}(\OO)]_\theta\;,
\end{equation}
for every $\alpha \in \R$, $\theta \in [0,1]$ and same holds for $H^{\alpha,p}_0(\OO)$. This together with reflexivity of Bessel potential spaces guarantees that both $H^{\alpha,p}(\mathcal O)$ and $H^{\alpha,p}_0(\mathcal O)$ generate a monotone family of interpolations spaces. 
\end{exa}	
	
	\subsection{H\"older spaces and controls}
	
	For $n\geq 2$ and a Banach space $V$, we define $\CC_n(0,T;V) $ to be the space of continuous functions from the simplex $\Delta_n = \{T\geq t_n\geq t_{n-1} \geq \dots \geq t_1 \geq 0\}$ to $V$ .
	For $n=1$ we adopt the convention that $\Delta _1:=[0,T]$ while $\CC(0,T;V)\equiv\CC_1(0,T;V)$ is just the usual space of continuous functions taking values in $V.$
	In the sequel, we will be only interested in the cases $n=1,2,3.$
	If $V$ is a Banach space and 
	\[
	S:\Delta _2\to \LL(V)
	\]
	is a two-parameter family of bounded linear maps,
	we define the increment operator $\ddh$ for $f:\Delta _1\to V$ and $g:\Delta _2\to V$ as
	\begin{align}
		&\ddh f_{t,s} := f_{t} - S_{t,s}f_{s},\quad \text{for}\enskip  (t,s)\in\Delta _2;\label{increment}
		\intertext{while}
		&\ddh g_{t,u ,s} := g_{t,s} - g_{t,u} - S_{t,u}g_{u,s},\quad \text{for}\enskip  (t,u,s)\in\Delta _3,\label{increment1}
	\end{align}
	and we recall (see \cite{gubinelli2010}) that $\mathrm{Im}\,\ddh=\mathrm{Ker}\,\ddh$.\footnote{Here $\ddh$ on the left is from \ref{increment} and $\ddh$ on the right in from \ref{increment1}.}
	When $S=\id$, $\ddh$ corresponds to the usual increment operator from controlled paths theory, and we shall use the notation $\dd:=\delta ^{\id}$. 
	
	For $f \in \CC_1(0,T;\BB_\alpha)$ we let
	\begin{align*}
	|f|_{0,\alpha} = \sup_{0\leq t \leq T}|f_t|_\alpha.
	\end{align*}
	If $\gamma >0$, the norm in $\CC_1^\gamma (0,T;\BB_\alpha )$ is defined as the usual H\"older norm, namely
	\begin{equation}
	\label{holder_norm}
	|f|_{\gamma,\alpha} := |f|_{0,\alpha } + [\delta f]_{\gamma ,\alpha }\,.
	\end{equation}
	where for $g=(g_{t,s}):\Delta _2\to \BB_\alpha ,$ we let $[g]_{\gamma ,\alpha }$ be the quantity
	\begin{equation}
	\label{holder_2}
		[g]_{\gamma,\alpha} := \sup_{0\leq s<t\leq T}\frac{| g_{t,s}|_\alpha}{|t-s|^\gamma}\,.
	\end{equation}
	Equipped with $[\cdot ]_{\gamma ,\alpha }$, the space $\CC_2^\gamma (0,T;\BB_\alpha )$ of all families such that the above quantity is finite forms a Banach space.
If $h=(h_{t,u,s}):\Delta _3\to \BB_\alpha ,$ we let
\begin{equation}
\label{norm_3}
[h]_{\gamma _1,\gamma _2,\alpha }:=
\sup_{(t,u,s)\in\Delta _3}\frac{|h _{t,u,s}|_{\alpha }}{|t-u|^{\gamma_1}|u-s|^{\gamma _2}},
\end{equation} 
and we denote by $\CC_3^{\gamma _1,\gamma _2 }(0,T;\BB_\alpha )$ the Banach space formed by all 3-indices elements as above such that $[h]_{\gamma _1,\gamma _2,\alpha }<\infty.$
All the previous norms can be taken with $\BB_\alpha$ replaced by $\BB_\alpha^m$ for $m \in \N,$ though we will still use the notations $|\cdot|_{\gamma,\alpha },[\cdot ]_{\gamma ,\alpha }$ etc.\ throughout the paper. For paths living in a more general Banach space $V$, the notations $|\cdot |_{\gamma ,V}, [\cdot ]_{\gamma ,V}, [\cdot ]_{\gamma _1,\gamma _2,V}$ will be used instead (with obvious changes in the definitions).

We will also work with functions that exhibit a uniform continuity `similar to H\"older' but in a weaker sense.

\begin{defn}\label{control}
	We say that a function $\omega :\Delta_2 \to \R_+$ is a control if it is a continuous map with $\omega(s,s) = 0$ for all $s \in [0,T]$, superadditive in the sense that
	$$\omega(t,r) + \omega(r,s) \leq \omega(t,s)\;,$$
	for all $0 \leq s \leq r \leq t \leq T$.
	
	For $p \geq 1$ and a Banach space $(V,\, |\cdot|_V)$ define a space $\CC_2^{p-var}(0,T; V)$ of all $g : \Delta_2 \to V$ such that there exists a control $\omega$ with $|g_{t,s}|_V \leq \omega^{1/p}(t,s)$.
\end{defn}
	Note that functions in $\CC_2^{p-var}$ are continuous by the definition of a control. Note also that $\CC^\gamma_2 \subset \CC_2^{1/\gamma-var}$ because $\omega(t,s) = C|t-s|$ is a control for all $C>0$. Another example of a control can be obtained using $p-$variation. Let $g:\Delta_2 \to V$ be a continuous function such that:
	$$\omega_g(t,s) = \sup_{\pi(t,s)} \sum_{t_i \in \pi(t,s)} |g_{t_{i+1},t_i}|^p_V < \infty\;,$$
	where the above supremum ranges over all partitions $\pi(t,s)$ of $[s,t]$. Then $\omega_g(t,s)$ defines a control and $|g_{t,s}|_V \leq \omega^{1/p}_g(t,s)$, thus continuous functions of finite $p$-variation lie in the space $\CC_2^{p-var}$.

\subsection{Assumptions and main results}\label{main_results}

We first state our assumptions on the family $(L_t)_{t\in[0,T]}.$
In what follows, we shall fix a number $\vartheta>0$, $\lambda \in \R$ and define a sector $\Sigma_{\vartheta,\lambda}$ of the complex plane as follows:
	\begin{equation}\label{nota:sigma}
	\Sigma_{\vartheta,\lambda}:=\{\zeta \in \C,\, |\arg(\zeta -\lambda) | < \pi/2+\vartheta\}
	\end{equation}
with the convention adopted throughout the paper that $\arg 0 =0$, thus in particular $[\lambda, \infty) \subset \Sigma_{\vartheta,\lambda}$.

In the next statements, we assume that we are given two Banach spaces 
\[
\mathcal (\mathcal X_1,|\cdot|_1) \subset (\mathcal X_0,|\cdot|_0)\,,
\]
where the embedding is continuous and dense.

	\begin{ass}\label{ass:L_t}
		Let $(L_t)_{t\in[0,T]}$ be a family of closed, densely defined linear operators on $\mathcal X_0$ with domains containing $\mathcal X_1$, and such that there exist constants $\Lambda,M>0$ (depending only on $T$) such that:
		\begin{enumerate}[label=(L\arabic*)]
			\item\label{L1} For each $t\in[0,T]$, the resolvent set $\rho (L_t)$ contains $\Sigma_{\vartheta,\lambda}$ and there exists a constant $\Lambda>0$ such that for $i =0,1,$
			\begin{equation}\label{resolvent}
			|(\zeta  - L_t)^{-1}|_{\LL(\mathcal X_i)}\leq \Lambda\big(1+|\zeta |\big)^{-1},\quad \forall\, \zeta \in\Sigma_{\vartheta,\lambda}\;.
			\end{equation}
			
			\item\label{L2}
			For any $t\in[0,T],$ we have
			\[
			|(\zeta -L_t)^{-1}|_{\LL(\mathcal X_0,\mathcal X_1)}\leq M \,,\quad \forall\, \zeta \in\Sigma_{\vartheta,\lambda}\,.
			\]
			
			\item\label{L3} There exists a control $\omega$ and $\varrho \in (0,1]$ such that for all $(t,s) \in \Delta_2$:
			\begin{align*}
			|L_t  - L_s|_{\LL(\mathcal X_1,\mathcal X_0)} \leq \omega^\varrho(t,s)\,.
			\end{align*}
			
			\item\label{L4} The above control satisfies the following integrability property: for all $(t,s)\in\Delta _2$
			\begin{align*}
				\int_s^t \frac{\omega^\varrho(r,s) dr}{r-s} < \infty\;.
			\end{align*}
		\end{enumerate}
	\end{ass}
	
Let $(L_t)_{t\in[0,T]} $ be such that Assumption \ref{ass:L_t} is satisfied.
From classical results obtained independently by Tanabe and Sobolevskii in the 60's \cite{tanabe1960equations,sobolevskii1966equations} (see also the seminal work of Kato \cite{kato1953integration}),
it is known that for $x\in\mathcal X_0,$ the equation
	\begin{equation}\label{non-autonomous}
	\partial _tu=L_tu_t,\quad u_s:=x \in \mathcal X_0,
	\end{equation}
admits a unique solution (in the usual weak, PDE sense) denoted by $S_{t,s}x$ and which depends linearly on $x\in\mathcal X_0.$
The two parameter mapping $S\colon \Delta _2\to \LL(\mathcal X_0)\cap \LL(\mathcal X_1)$ is usually referred to as the \emph{propagator} associated to the family $(L_t)_{t\in[0,T]}$ (it is sometimes called the `parabolic fundamental solution' associated with $L_\cdot $, but for conciseness we adopt the terminology used in \cite{reed1979methods}). 
The family $(S_{t,s})_{(t,s)\in\Delta _2}$ should be heuristically understood as `$\exp(\int_s^tL_r\,dr)$' and its main interest lies in the existence of the Duhamel-type formula \eqref{mild_rep}.

The precise definition of a propagator is as follows.
	\begin{defn}
	\label{def:propagator}
		We say that $S\colon \Delta _2\to \LL(\mathcal X_0)$ is a \emph{propagator} associated to the family $(L_t)_{t\in[0,T]}$ of unbounded, closed operators with dense domains containing $\mathcal X_1$ if and only if the following holds:
		\begin{enumerate}[label=(P\arabic*)]
			\item\label{P1} $S\in \CC_2(0,T;\LL_s(\mathcal X_0))$ and there exist constants $\lambda ,\Lambda>0$ such that for every $(t,s) \in \Delta_2$
			\begin{equation}
			\label{growth_S}
				|S_{t,s}|_{\LL(\mathcal X_0)}, |S_{t,s}|_{\LL(\mathcal X_1)} \leq \Lambda e^{\lambda(t-s)}\;.
			\end{equation}
			
			\item\label{P2} $S_{t,t}=\id$ and $S_{t,s}=S_{t,u }S_{u ,s}$ for $(t,u ,s)\in\Delta _3$. 
			
			\item\label{P3} There exists $C_T>0$ such that for every $(t,s)\in \Delta_2,$ $s\neq t$:
			\[
			|S_{t,s}-\id|_{\LL(\mathcal X_1,\mathcal X_0)} \leq C_T |t-s|\,.
			\]
			\item\label{P4} For all $s,t \in [0,T]$ and $x \in \mathcal X_1$ we have:
			$$
			\frac{d}{dt}S_{t,s}x = L_tS_{t,s}x \quad  \text{and}  \quad \frac{d}{ds}S_{t,s}x=-S_{t,s}L_sx\,,$$
			where the differentiation is taking place in the Banach space $\mathcal X_0$.
			\item\label{P5} For every $(t,s)\in \Delta_2,$ $s\neq t$ we have that $S_{t,s} \mathcal X_0 \subset \mathcal X_1$ and for some constant $N_T>0$
			\begin{equation}
			\label{smoothing_S}
			|L_t S_{t,s}|_{\LL(\mathcal X_0)} \lesssim |S_{t,s}|_{\LL(\mathcal X_0,\mathcal X_1)} \leq  N_{T}|t-s|^{-1}.
			\end{equation} 
		\end{enumerate}
	\end{defn}

\begin{rem}\label{rem:Lbounded}
If $(L_t)$ satisfies Assumption~\ref{ass:L_t} then the first inequality in \eqref{smoothing_S} can be justified as follows. We claim that for each $t\in [0,T]$ and $y \in \mathcal X_1$
\begin{equation}\label{eq:cts_embeding}
	|L_ty|_{0}\leq C_t|y|_{1}\,.
\end{equation}
Using this bound at $t=0$ together with \ref{L3} implies uniformly for $t\in [0,T]$
\[
|L_ty|_{0}\leq (C_0+\omega^\varrho(0,T))|y|_{1}\;.
\]
It remains to show~\eqref{eq:cts_embeding}. Fix $t \in [0,T]$, the assumption on the closure of $L_t$ shows that $\mathcal{Y} = D(L_t)$ is a complete Banach space with respect to the graph norm $x \mapsto |x|_0 + |L_tx|_0$. It is therefore enough to show that the inclusion $\mathcal{X}_1 \subset \mathcal Y$ is continuous i.e. that $|x|_0+|L_tx|_0 \lesssim |x|_1$ for every $x \in \mathcal X_1$. By the Closed Graph Theorem it is equivalent to show that if $x_n\to x$ in $\mathcal X_1$ and $x_n\to x'$ in $\mathcal Y$ then $x = x'$. Then, since $(\mathcal X_1, |\cdot|_1)$ is continuously embedded in $(\mathcal X_0, |\cdot|_0)$, we have for such sequences
\begin{equation*}
|x_n-x|_0\to 0\quad \text{and}\quad |x_n-x'|_0 + |L_t(x_n-x')|_0\to 0\quad\text{as $n\to\infty$.}
\end{equation*}
This trivially implies $x=x'$ by the uniqueness of the limit in $\mathcal X_0$, thus showing~\eqref{eq:cts_embeding}. 
\end{rem}

As is well-known when $\omega (t,s)=C|t-s|$ for some $C > 0$, Assumption \ref{ass:L_t} guarantees that the family $(L_t)_{t\in[0,T]}$ generates a propagator (see \cite{tanabe1960equations,sobolevskii1966equations}).
Though the general case should be known in principle, we are not aware of any reference in the $p$-variation setting (see nevertheless \cite{kato1953integration} for $p=1$).

\begin{rem}
	\label{rem:not_restrictive}
	Note that in general if $L_t$ is the generator of an analytic semigroup then it satisfies a resolvent bound of the form \eqref{resolvent} for some $\Lambda \geq 1$. As will be seen later, in order to be able to recover the existence and uniqueness of the propagators from the Sewing Lemma (Theorem \ref{thm:MSL}), we need to restrict ourselves to the case when $\Lambda = 1$ in \ref{L1}. This is, however, not restrictive as far as Theorems \ref{thm:main} and \ref{thm:weak} are concerned, since for the case $\Lambda>1$ we can simply refer to the classical results of Tanabe and Sobolevskii \cite{tanabe1960equations,sobolevskii1966equations}.  Note that these results also guarantee that $\Lambda$ in~\eqref{resolvent} and $\Lambda$ in~\eqref{growth_S} are the same. We shall see that this is also true in our case $\Lambda = 1$.
\end{rem}

A criterion that guarantees that $\Lambda = 1$ in~\ref{L1}, is the dissipativity of the operators $L_t$ for each $t \in [0,T]$. Recall that an unbounded operator $A$ on a Banach space $V$ is \emph{dissipative} if
\begin{equation}\label{diss}
\langle Au,u^*\rangle\,\leq 0
\quad \forall u\in D(A)\subset V,\quad u^*\in V^*\,\enskip\text{s.t.\ }\langle u,u^*\rangle = |u|^2=|u^*|^2\;,
\end{equation}
where $\langle{u,u^*}\rangle$ denotes the value of $u^*$ at $u$. Then, it is well-known (see~\cite{pazy1983semigroups,goldstein1985semigroups}) that if for a dissipative operator $A$ and $\zeta \in \rho(A)$ one has the image of $\zeta - A$ being equal to $V$, then $|(\zeta-A)^{-1}|_{\LL(V)} \leq (1+|z|)^{-1}$. For an overview on dissipative operators on Banach spaces, we refer for instance to~\cite{lumer1961dissipative}.

\begin{thm}[Tanabe/Sobolevskii]
	\label{thm:tanabe}
	Assume that the Banach spaces $\mathcal X_0$ and $\mathcal X_1$ are reflexive. Let $(L_t)_{t\in[0,T]}$ be a family of operators satisfying the properties \ref{L1}, \ref{L3} of Assumption \ref{ass:L_t}, with the additional hypothesis that $\Lambda=1$ in \eqref{resolvent}.
	Then, there exists a unique map
	\[
	S\colon \Delta_2\to \LL(\mathcal X_0)\;,
	\]
	satisfying properties \ref{P1}, \ref{P2}, \ref{P3}, \ref{P4}, and in addition for all $(t,s) \in \Delta_2$ we have the property\footnote{The existence of the exponential in~\eqref{eq:S_vs_Exp} is implied by the property \ref{L1}, as will be seen in Section \ref{sec:MSL}.}
	\begin{equation}\label{eq:S_vs_Exp}
		|S_{t,s} - e^{(t-s)L_s}|_{\LL(\mathcal X_1, \mathcal X_0)} \lesssim_{T} |t-s|\, \omega^\varrho(t,s)\;.
	\end{equation}
	Moreover, one has in~\eqref{growth_S} that $\Lambda = 1$ and $\lambda$ is the same as in the sector $\Sigma_{\vartheta,\lambda}$.
	Finally, if~\ref{L4} is also satisfied, then $S$ is the propagator associated to the family $(L_t)_{t\in[0,T]}$  (in the sense of Definition \ref{def:propagator}).
\end{thm}
In Section \ref{sec:MSL}, we will provide a whole new proof of the above result which is based on the `multiplicative sewing lemma' Theorem \ref{thm:MSL}. We would like to point out that in order to meet the hypotheses of Theorem \ref{thm:MSL}, in the above theorem we have to assume that both $\mathcal{X}_0$ and $\mathcal{X}_1$ are reflexive (see Lemma~\ref{topology3}). Note that \cite{tanabe1960equations} does not need to assume reflexivity, while \cite{kato1953integration} does assume reflexivity as well.

For our purposes, we will need the propagators to have stronger properties and act not only on two spaces $\mathcal X_0$ and $\mathcal X_1$ but on a continuous range of interpolation spaces $(\BB_\alpha)_{\alpha\in\R}$ for $\alpha \in I$ where $I\subset \R$ is an interval. For this purpose, stronger assumptions on the family $L_t$ will be required.
We have the following definition.

\begin{defn}\label{def:part}
Let $(\BB_\alpha )_{\alpha \in \R}$ be a monotone family of interpolation spaces, and fix an interval $I\subset\R.$
We say that $S$ is a propagator on the full range $(\BB_{\alpha })_{\alpha \in I}$ if for every $\alpha \in I,$ $S$ restricted to $\BB_\alpha $ is itself a propagator, in the sense of Definition \ref{def:propagator} for the pair $(\mathcal X_0,\mathcal X_1)=(\BB_\alpha ,\BB_{\alpha +1})$.
\end{defn}

 We now give a concrete example of a family $(L_t)_{t\in[0,T]}$ where Assumption \ref{ass:L_t} is fulfilled with $\Lambda = 1$.
 \begin{exa}\label{ex:family1}
 	Let $1<p<\infty$ and define
 	\[
 	\mathcal X_0:=L^p(\T^n) \quad \text{and}\quad \mathcal X_1:= W^{2,p}(\T^n).
 	\]
 	Then $\mathcal X_1$ is a constant domain associated to the family $(L_t)_{t\in[0,T]}$ defined for each $t\in[0,T]$ as:
 	\[
 	L_tu:= \nabla \cdot (a_t(x) \nabla u),\quad u\in  \mathcal X_1,
 	\]
 	where $a_t(x) \in \R^{n \times n}$ is a matrix satisfying the following uniform ellipticity condition: there exists a constant $\varkappa >0$ such that
 	\begin{equation}
 		\label{uniform_ellipticity}
 		\sum_{j,k}a^{jk}_t(x)\xi ^j\xi ^k\geq \varkappa |\xi |^2\;,
 	\end{equation}
 	for every $t\in[0,T],$ $x\in\T^n$ and $\xi\in \R^n$. Moreover, we assume that for every $t \in [0,T],$
 	\begin{equation}
 	\label{a_t_C_2}
 	a_t( \cdot)\in \CC^{2}(\T^n; \LL(\R^n)).
 	\end{equation} 
 	The results of~\cite[Sec. 7.3]{pazy1983semigroups} imply that for every $t \in [0,T]$ the operator $L_t$ is sectorial, and since above $\varkappa$ is independent of $t$ one can choose the same $\vartheta$ and $\lambda$ in $\Sigma_{\vartheta,\lambda}$ for every $t \in [0,T]$. Moreover, it is classical (see for instance \cite{cialdea2006criteria}) 
 	that for each $y \in [0,T]$ the operator $L_t$ is dissipative and satisfies property \ref{L1} with $\Lambda = 1$, but
 	since the proof is relatively simple, we now give a short argument why this is true in the case $p \geq 2$ (the case when $p \in (1,2)$ can be done similarly). Fix $t \in [0,T]$. The fact that  $\mathrm{Im}(\zeta - L_t)=\mathcal X_0$ for all $\zeta \in \rho(L_t)$ is well-known for spatially $\CC^2$ coefficients (see again~\cite[Sec. 7.3]{pazy1983semigroups}). We now want to show that $L_t$ is dissipative in the sense of~\eqref{diss}.

	Let $f \in \CC^\infty(\T^n)$ be non zero and define
 	$$g(x) = \frac{|f(x)|^{p-2}}{|f|^{p-2}_{L^p}}  f(x)\,.$$ 
 	Let $q$ be such that $p^{-1} + q^{-1} = 1$ then we have
 	$$|f|^2_{L^p} = |g|^2_{L^q} = \int_{\T^n} f(x)g(x)dx\;.$$
 	We can now use this $g$ and the uniform ellipticity~\eqref{uniform_ellipticity} to deduce
 	\begin{align*}
 		\langle{L_t f, g}\rangle &= - \int_{\T^n} a_t(x) \nabla f(x) \cdot \nabla g(x) dx \\
 		& = - |f|^{2-p}_{L^p} \int_{\T^n} a_t(x) \nabla f(x) \cdot \nabla (|f|^{p-2}(x) f(x)) \\
 		& = - |f|^{2-p}_{L^p} (p-1) \int_{\T^n} |f(x)|^{p-2} a_t(x) \nabla f(x) \cdot \nabla f(x) dx \\
 		& \leq - \varkappa |f|^{2-p}_{L^p} (p-1) \int_{\T^n} |f(x)|^{p-2} |\nabla f(x)|^2 \leq 0\;.
 	\end{align*}
 	By density, this inequality can be extended to an inequality for all $f \in W^{2,p}(\T^n)$. This shows that $L_t$ is dissipative, hence $\Lambda=1.$
 	
 	Moreover, the operators $(L_t)_{t \in [0,T]}$ satisfy \ref{L2} but potentially not uniformly in $t$. If in addition, there exists $p \geq 1$ and a control $\omega$ such that for all $(t,s) \in \Delta_2$
 	\begin{align*}
 		\sup_{x \in \T^n} |a_t(x)-  a_s( x)| \leq \omega^{1/p}(t,s)\;,
 	\end{align*}
 	then~\ref{L2} is satisfied uniformly and moreover~\ref{L3} holds true with $\varrho = 1/p$. (For justification in a case of H\"older continuous $t \mapsto a_t(x)$ see \cite[Section 14]{amann1984existence}, a general control case would follow by using a time change.)
 	The property \ref{L4} is satisfied for example in the case of H\"older continuity i.e when $\omega(t,s) = C|t-s|$.
 
 Finally, if the condition \eqref{a_t_C_2} is replaced by the stronger assumption that $a_t( \cdot)\in \CC^{\infty}(\T^n; \LL(\R^n)),$ then $S$ extends uniquely to a propagator on the full range $(\BB_\alpha )_{\alpha \in\R}$ (in the sense of Definition \ref{def:part}), where the $\BB_\alpha = H^{2\alpha,p}(\T^n)$ are the Bessel potential spaces from Example~\ref{ex:Besel}. See Example~\ref{ex:full_propagator} for more details.
 \end{exa}

Next, in order to present our main results on the evolution equation \eqref{evolution_equation}, we first need to recall the definition of a two-step rough path. 
	\begin{defn}[Rough Path] Let $\gamma >1/3$. We define the space of rough paths $\cC^\gamma(0,T; \R^d)$ to consist of the pairs $(X,\XX) =:\X$ such that $X \in \CC^\gamma(0,T;\R^d)$, $\XX \in \CC_2^{2\gamma}(0,T;\R^{d\times d})$ are satisfying the Chen's relation:
		\begin{equation}\label{chen}
			\delta\XX^{i,j}_{t,u ,s} :=\XX^{i,j}_{t,s}-\XX^{i,j}_{t,u }-\XX^{i,j}_{u ,s}= \delta X^i_{t,u }\delta X^j_{u ,s},
		\end{equation}
		for every $(t,u ,s)\in\Delta _3$ and $1\leq i,j\leq d.$
		The rough paths space is equipped with the pseudometric
	\begin{align*}
		\varrho_\gamma(\X,\tilde{\X}) = [X-\tilde{X}]_\gamma + [\XX-\tilde{\XX}]_{2\gamma}\;,
	\end{align*}
where the quantities $[X]_{\gamma}$ resp.\ $[\XX]_{2\gamma}$ are the H\"older seminorms defined in \eqref{holder_2}. For simplicity, we shall write in the sequel $\varrho_\gamma(\X) := \varrho_\gamma(0,\X)$.
	\end{defn}

Now, we need to restrict our study to a suitable class of non-linearities.
\begin{defn}\label{def:C^k}
	Let $\BB := (\BB_\alpha)_{\alpha \in \R}$ be a monotone family of interpolation spaces. For some fixed $\alpha,\beta \in \R$ and $k \in \N_0$ we define the space
	$\CC^k_{\alpha,\beta}(\BB^m,\BB^n)$ as the space of $k$-differentiable functions $G \colon\BB^m_\theta \to \BB^n_{\theta+\beta}$ for every $\theta\geq\alpha$, and $n,m \in \mathbf{N}_0$ and such that $D^iG$ sends bounded subsets of $\BB^m_\theta$ to bounded sets of $\BB^n_{\theta+\beta}$, for all $i = 0, \dots k$.
	
	Similarly, define $\Lip_{\alpha,\beta}(\BB^m,\BB^n)$ to be a space of Lipschitz functions $\BB^m_\theta \to \BB^n_{\theta+\beta}$ for all $\theta \geq \alpha$ that send bounded sets to bounded sets and such that the Lipschitz constant is uniformly bounded on bounded sets. When $m = n$ we will simply write $C^k_{\alpha,\beta}(\BB^m)$ or $\Lip_{\alpha,\beta}(\BB^m)$.
\end{defn}

We now have all at hand to introduce our main result on the existence and uniqueness of solutions.
\begin{thm}[Solvability of \eqref{evolution_equation}]
\label{thm:main}
Fix a two-step rough path $\X = (X, \XX)$ in $\cC^\gamma(0,T;\R^d)$ with $\gamma >1/3,$ and 
consider a monotone family of interpolation spaces $(\BB_\alpha )_{\alpha \in \R}.$
Assume that $(L_t)_{t\in[0,T]}$ is a given family of linear operators generating the propagator $(S_{t,s})_{(t,s) \in \Delta_2}$ on $(\BB_\alpha )_{\alpha \in (-3\gamma ,0]}$.
 For some $\sigma < \gamma$, assume that we are given a non-linearity
$F \in \CC^3_{-2\gamma,-\sigma}(\BB,\BB^d)$ and $N\in \Lip_{\alpha,-\delta}(\BB)$ for some $0 \leq \delta < 1$.

Then, for every $x \in \BB_0,$ there exists a maximal time $\tau \in (0,T]$ and a unique function $u \in \CC([0,\tau),\BB_0)$ such that $(u,F(u))$ is a controlled rough path in the sense of Definition~\ref{def:controlled} and
 $u$ is a \emph{mild solution} to the Rough PDE \eqref{evolution_equation}, namely 
		\begin{equation} \label{e:rpde1}
			u_t = S_{t,0}x + \int_0^t S_{t,r}N(u_r)dr+ \int_0^t S_{t,r} F(u_r) \cdot d\X_r \, , \quad t<\tau, 
		\end{equation}
where the latter integral is understood in the rough integral sense of Theorem~\ref{integration}.
\end{thm}

\begin{rem}\label{rem:young}
 When the rough path $X$ is $\beta $-H\"older with $\beta >1/2,$ then $u$ is a solution to the mild equation~\eqref{e:rpde1} where all the integrals are well-defined using Young integration.
\end{rem}

\begin{rem}\label{rem:Le}
	It was recently shown by L\^e \cite{le2020stochastic} that the additive sewing Lemma of \cite{gubinelli2004} has a natural extension in a stochastic setting, exploiting a certain martingale decomposition. As already emphasized in \eqref{subcritical_example}, the value $\sigma=\gamma$ is critical, even though It\^o solutions exist and are unique for $\sigma=\frac12$ and $X$ is Brownian so that $\gamma=\frac12-\varepsilon$ for any $\varepsilon>0$ (to avoid issues related to stochastic parabolicity, note however that \eqref{subcritical_example} has to be understood first as a Stratonovich equation, then corrected to an equivalent It\^o form).
	It should be possible to deal simultaneously with the semigroup $\exp(t\Delta)$ and the semi-martingale structure of the right-hand side, in order to recover such solvability results. We chose nevertheless to leave this question for future investigations.
\end{rem}

When talking about the mild solutions it is natural to ask whether these coincide with the weak solutions. The following statement can be considered as an answer to this question.
\begin{thm}
	\label{thm:weak}
	Let $(\BB_\alpha ),(L_t),N,F,x$ be as in Theorem~\ref{thm:main} and let $\nu = \max\{1, \sigma+2\gamma\}$. Then for all $\varphi \in \BB_{-\nu}^*$ the following integral formula holds:
	\begin{equation}
	\label{integral_weak}
	\langle u_t,\varphi\rangle= \langle x,\varphi\rangle + \int_0^t \langle L_su_s,\varphi\rangle ds + \int_0^t \langle N(u_s),\varphi\rangle ds + \int_0^t \langle F(u_s),\varphi\rangle \cdot d\X_s\;,
	\end{equation}
	where the last integral is the usual rough integral in the sense of \cite[Theorem 4.10]{friz} and $L_su_s$ is viewed as an element of $\BB_{-1}$. Conversely, if $(u,F(u))$ is a controlled rough path in the sense of Definition~\ref{def:controlled} and \eqref{integral_weak} holds for all $\varphi \in \BB_{-\nu}^*$, then it is also a mild solution, namely \eqref{e:rpde1} holds. 
\end{thm}

For the precise statement, we refer the reader to the Theorem~\ref{weak}, where we also show that weak solutions are mild solutions.

\subsection{An illustrating example: non-autonomous stochastic reaction-diffusion equations}
\label{subsec:specified}

Consider the non-autonomous evolution problem
	\begin{align}\label{react_diff}
		du_t(x) &= \nabla\cdot \big(a_t(x) \nabla u_t(x)\big) dt + f(u_t(x))dt+\sum_{i=1}^d p_i(u_t(x))d\X^i_t\;,\\
		u_0 &\in H^{k,p}({\T^n})\;,\nonumber
	\end{align}
with $(a^{ij}) \in \CC^{\varrho,\infty}([0,T]\times {\T^n} ; \LL(\R^n)),$ $1\leq i,j\leq n,$ for some $\varrho >0$ as in Example~\ref{ex:family1}. We assume that $f$ and $p_i$ are polynomials with coefficients in $H^{k,p}({\T^n})$ i.e.
\begin{align*}
	f(x,u(x)) = \sum_{j = 0}^m h_j(x) u^j(x)\,\quad x\in {\T^n}\;,
\end{align*}
for some $m \in \N_0$ and $h_j \in H^{k,p}({\T^n})$ for every $j \in \{1,\dots,m\}$, and $p_i$ has a similar form for all $i \in \{1,\dots,d\}$.
Furthermore,
$X_t = (X_t^1,X_t^2,...,X_t^d)$ is assumed to be endowed with a rough path enhancement $\X=(X,\XX)\in \cC^\gamma(\R_+,\R^d)$ with $\gamma >1/3$.

Fix $p\in(1,\infty)$ and for $\alpha \in \R,$ define $\BB_\alpha :=H^{k+2\alpha,p}({\T^n})$ (this is indeed a 
monotone family of interpolation spaces, by definition of the Bessel potential spaces and Example 
\ref{exa:scale}).
By Theorem~\ref{Tanabe2}, the assumptions on $(a^{ij})$ guarantee that $L_t = \nabla\cdot (a_t \nabla)$ is the generator of a propagator $S_{t,s}$ on the full 
range $(\BB_\alpha)_{\alpha \in\R},$ in the sense of Definition \ref{def:part}.

We need to check that the required assumptions on the non-linearities hold. If $k$ is such that $p(k - 4\gamma) 
> n$, then it is easily observed from the Sobolev Embedding Theorem
that for every $\alpha \geq -2\gamma$, multiplication is a smooth operation from $H^{k+2\alpha,p}({\T^n})\times H^{k+2\alpha,p}({\T^n}) \to H^{k+2\alpha,p}({\T^n})$. Therefore, the operator induced by $f$ is smooth from $\BB_\alpha = H^{k+2\alpha,p}({\T^n})$ to itself for every $\alpha \geq 0$. Moreover, since $f$ is a polynomial it is easy to see that it sends bounded sets of 
$H^{k+2\alpha,p}({\T^n})$ to bounded sets for all $\alpha \geq 0$ therefore showing that $f \in \CC^\infty_{0,0}(\BB)$. Similarly, since $H^{k,p}(\T^n) \subset H^{k-4\gamma,p}(\T^n)$ continuously, we see that $p_i$ induce smooth operators from $\BB_\alpha = H^{k+2\alpha,p}({\T^n})$ to itself for every $\alpha \geq -2\gamma$ and that $p_i \in \CC^\infty_{-2\gamma,0}(\BB)$ for each $i=1,\dots ,d$.


We now want to specialize further our results, by introducing a stochastic context for 
\eqref{evolution_equation} (which constitutes an important motivation for introducing a rough paths 
formulation, see the discussion in the introduction).
We have the following.
	
\begin{thm}
	\label{thm:specify}
	 Let $(\Omega ,\mathcal F,\mathbb{P},(\mathcal F_t)_{t\in[0,T]})$ be a filtered probability space satisfying the usual assumptions, let $(B^{i})_{1\leq i\leq d}$ be a multidimensional Brownian motion on $\Omega ,$ for some $d\geq 1$. Let $(k,p) \in \N \times (1,\infty)$ be such that $k>n/p+4/3$. Let $f,p_i :\R \to \R$ be polynomials with coefficients in $H^{k,p}({\T^n})$ for every $i \in \{1,\dots d\}$.  
	 Consider an adapted process $a:\Omega \times[0,T]\to \CC^\infty({\T^n};\LL(\R^n))$ such that for $\mathbb{P}$-a.e. $\omega$, $a(\omega ) \in \CC^{\varrho,\infty}([0,T]\times {\T^n} ; \LL(\R^n))$ satisfies the assumptions of Example \ref{ex:family1} with the coercivity constant $\varkappa(\omega ) > 0$ in \eqref{uniform_ellipticity}. 

Then there exists a stopping time $\tau$ such that the equation
\begin{equation}\label{stochastic_reaction}
\begin{aligned}
du_t(x) &= \nabla\cdot \big(a_t(\omega ,x) \nabla u_t(x)\big) dt + f(u_t(x))dt+\sum_{i=1}^d p_i(u_t(x)) dB^i_t\;,\\
u_0 &\in H_0^{k,p}({\T^n})\;,
\end{aligned}
\end{equation} 
has a weak It\^o solution in the following sense: a stochastic process $u\colon \Omega \times[0,T]\to H^{k,p}({\T^n})$ is adapted and satisfies:
\begin{itemize}
 \item $\mathbb{P}$-almost surely either $\tau = T$, or $\limsup_{t\nearrow \tau }|u_t|_{H^{k,p}}=\infty$;
 \item for every $t\in[0,T]$ and any $\varphi \in \CC^\infty({\T^n})$, we have on $\{t<\tau \}$:
 \begin{multline}
 \label{weak_Cc}
\int_{{\T^n}}(u _t(x)-u _0(x))\varphi (x)dx + \int_0^t\int_{{\T^n}} a_s(x)\nabla u _s(x)\cdot \nabla \varphi(x) dx\,ds 
\\
=\int_0^t\int_{{\T^n}}f(u _s(x))\varphi (x)dx\,ds + \sum_{i=1}^d \int_{0}^t \int_{{\T^n}} p_i(u _s(x))\varphi (x)dx\,dB^i_s\,.
\end{multline}
\end{itemize}

If in addition, the following moment estimate holds true: for every $m \geq 1$ there exists $K_m > 0$ 
\begin{equation}
	\mathbb E\left [\sup_{t\in[0,T]}|a_t(\cdot)|^m_{\CC^{k-1}({\T^n})}\right ] \leq K_m\,,\label{eq:moments}
\end{equation}
then the solution $u\in\CC([0,\tau),H^{k,p}({\T^n}))$ is unique in the class previously defined.\footnote{Note that this does not rule out the possibility of existence of discontinuous in time weak solutions.}
\end{thm}

The existence part is a simple consequence of Theorem~\ref{thm:weak} and the fact that in this context $\varphi \in \BB^*_{-1}$ for $\varphi \in 
\CC^\infty({\T^n})$. The uniqueness statement is a bit more involved and requires some probabilistic tools like Kolmogorov extension theorem. 
A full proof will be presented in Section~\ref{sect:weak}. We want to point out once again that the importance of the above theorem comes 
from the fact that the equation~\eqref{stochastic_reaction} can not be solved as a mild It\^o equation 
because in this case the propagator $S_{t,s}$ is itself random, thus making the stochastic integrand 
$S_{t,s}p_i(u_s)$ non adapted.

Note that the results of Example \ref{ex:family1} and Section~\ref{subsec:specified} can be generalized to general smooth bounded domains $\OO \subset \R^n$ instead of the torus $\T^n$, considering for instance (homogeneous) Dirichlet boundary conditions. In order to replicate the above results one can construct a scale $(\BB_\alpha)_{\alpha \in \R}$ similarly as in Example \ref{exa:hilbert} for the  Dirichlet Laplacian $(L,D(L)) = (\Delta_D, H^{2,p}(\OO)\cap H^{1,p}_0(\OO))$. In order to simplify the presentation, we have decided to postpone talking about general domains until Section~\ref{subsec:CH}, where we look at the Cahn-Hilliard equation with Dirichlet boundary conditions. The case of homogeneous Neumann boundary conditions could be handled in a similar fashion.
	
\section{A multiplicative sewing Lemma and some by-products}
\label{sec:MSL}
Needless to say, the Sewing Lemma is a central result in rough paths theory. It is the core argument that permits to deal with the integration of controlled paths \cite{gubinelli2004} (not even mentioning its implicit use in the original `multiplicative functional'-formalism by Lyons \cite{lyons98}).
There is a vast literature on extensions of this result, e.g.\ in the context of reduced increments associated to a semigroup \cite{gubinelli2010,deya2012rough}, or flows \cite{bailleul2019rough}, see also the recent works \cite{le2020stochastic,brault2019nonlinear,brault2019nonlinear2,brault2020nonlinear3}. We remark that the original statement in \cite{feyel2008non} is fairly general and  applies in particular to non-commutative settings.
In this section, we introduce a new version of this result which fits well in multiplicative, infinite-dimensional contexts.

\subsection{The main result}
	We are going to present a version of a multiplicative sewing lemma which can be considered as a generalization of the non-commutative sewing lemma by Feyel, De La Pradelle, and Mokobodzki \cite{feyel2008non}.
	In the sequel, we denote by $(\MM,\circ)$ a monoid (for convenience we will mostly omit to write $\circ$).
	For a function $\mu \colon \Delta_2 \to \MM$ and an arbitrary partition $\pi = \{s =t_0 < t_1 < \dots < t_k = t\}$ of $[s,t]$ set: 
	\begin{equation*}
		\mu^\pi_{t,s} = \prod_{i=0}^{k-1} \mu_{t_{i+1},t_i}\,.
	\end{equation*}
	
	\begin{defn}\label{def:topologies} Let $(\MM, \circ)$ be a monoid.
			We call a triple $(\MM, |\cdot|, d)$ a submultiplicative monoid if $d$ is a metric on $\MM$ and a function $|\cdot| : \MM \to \R_+$ is such that for all $a,b,c \in \MM$:
			\begin{equation}
			\label{submult}
			d(ac,bc) \leq d(a,b) |c|,\quad\text{and}\quad d(ca,cb) \leq  |c| d(a,b)\,.
			\end{equation}
			Moreover, we assume that for all $C \in \R_+$ the sets $B_C = \{a \in \MM:\, |a| \leq C \}$ are complete with respect to the metric $d$.
	\end{defn}
	\begin{rem}
	\label{topology2}
	Let $\MM$ be a submultiplicative monoid.
		As a consequence of \eqref{submult}, the multiplication is continuous with respect to $d$ on the sets $B_C$.
		Indeed if $|b_n| \leq C$ and $d(a_n,a) \to 0$ and $d(b_n,b) \to 0$ as $n \to \infty$ then:
		\begin{align*}
			d(a_nb_n, ab) \leq d(a_n,a)|b_n| + |a| d(b_n,b) \leq C d(a_n,a) + |a| d(b_n,b) \to 0\, ,
		\end{align*}
		as $n \to \infty$. The same is true if instead of $|b_n| \leq C$ we have $|a_n| \leq C$.
	\end{rem}
	
	\begin{defn}\label{def:monoid}
Let $\MM$ be a submultiplicative monoid.
We say that a function $\mu :\Delta_2\to \MM$ is 
\begin{enumerate}[label=(\roman*)]
 \item \emph{multiplicative} if for every $(t,u ,s)\in\Delta_3$:
			\begin{equation*}
				\mu _{t,s} = \mu _{t,u }\mu _{u, s}.
			\end{equation*}
 \item \emph{almost-multiplicative} if there exists a control $\omega :\Delta _2 \to\R_+$ and $z>1$ such that for each $(t,u ,s)\in\Delta _3:$
	\begin{equation}
	\label{hyp_mu0}
	d(\mu _{t,s} ,\mu _{t,u }\mu _{u, s})\leq \omega ^z(t,s)\,.
	\end{equation}
\end{enumerate}
Next, let $\epsilon : \Delta_2 \to \R_+$ be such that $\epsilon (t,s)$ is continuous, increasing in $t$ and decreasing in $s$.
We say that
\begin{enumerate}[label=(\roman*)]
\setcounter{enumi}{2}
 \item \label{submult:2}
 $\mu $ has \emph{moderate growth} with growth rate $\epsilon $ if
 \begin{equation}\label{partition_growth}
|\mu^\pi_{t,s}| \leq \epsilon(t,s),\quad \text{for each}\enskip (t,s)\in\Delta_2\;,
\end{equation}
for every $(t,s) \in \Delta_2,$ independently of the choice of partition $\pi$ of $[s,t]$. 
\end{enumerate}

We denote by $\mathrm{BG}_2(0,T;\MM)$ the set of all functions $\mu: \Delta_2\to \MM$ with moderate growth for some $\epsilon$ as above.
\end{defn}

	\begin{thm}[Multiplicative sewing Lemma] \label{thm:MSL}
		Let $(\MM, |\cdot|, d)$ be a submultiplicative monoid with unit $\mathds{1}.$ Let $\epsilon : \Delta_2 \to 
		\R_+$ be increasing in the first argument and decreasing in the second, and let $\mu \in 
		\mathrm{BG}_{2}(0,T;\MM)$ with a growth rate $\epsilon$. Assume that there exists a control $\omega$ 
		and a constant $z>1$ so that \eqref{hyp_mu0} holds.
		
		Then, there exists a unique multiplicative $\varphi \in \mathrm{BG}_2(0, T; \MM)$ such that for every $(t,s)\in\Delta_2$:
		\begin{equation} \label{ineq:mu_phi0}
			d(\varphi _{t,s}, \mu_{t,s}) \lesssim_{z,T} \omega^{z}(t,s). 
		\end{equation}
		The function $\varphi$ has the same growth rate $\epsilon$ and for all $(t,s) \in \Delta_2,$ for every sequence of partitions $\pi_n = \{s =t_0 < t_1 < \dots < t_n = t\}$ with mesh-size $|\pi_n| = \max_i |t^n_{i+1} - t^n_i| \to 0$ as $n \to \infty$ we have:
		\begin{equation}\label{partiotion}
			\varphi_{t,s} = d\,\text{-}\lim_{n\to\infty} \prod_{i=0}^{k_n-1} \mu_{t^n_{i+1},t^n_i}\;.
		\end{equation}
		Moreover, if $\mu_{t,t} = \mathds{1}$ for all $t \in [0,T]$ then $\varphi_{t,t} = \mathds{1}$ and if in addition $\mu :\Delta_2 \to (\MM,d)$ is continuous then so is $\varphi :\Delta_2 \to (\MM,d)$.
	\end{thm}

	\begin{proof}
		Our proof is reminiscent to that of the additive Sewing Lemma in \cite{friz}. 
		
		\textit{Existence:}
		Let $(t,s) \in \Delta_2$ and $\pi=\{s =t_0 < t_1 < \dots < t_k = t\}$ be a partition of $[s,t]$. Since $\omega$ is a control there exists $0 < l < k$ such that:
		\begin{align*}
			\omega(t_{l+1},t_{l-1}) \leq \frac{2\, \omega(t,s)}{k}\;.
		\end{align*}
		This is true since assuming the opposite would contradict the superadditivity of $\omega$.
		Denote by $\hat{{\pi}}$ the partition of $[s,t]$ obtained by removing the point $t_l$ from $\pi$. Then using~\eqref{hyp_mu0} and submultiplicativity:
		\begin{align*}
			d(\mu^{\hat{{\pi}}}_{t,s}, \mu^{\pi}_{t,s}) &\leq |\prod_{i = l+1}^{k} \mu_{t_{i+1},t_i}\;|\, d\big(\mu_{t_{l+1},t_{l-1}}, \mu_{t_{l+1},t_{l}}\circ\mu_{t_{l},t_{l-1}}\big)\, |\prod_{i = 0}^{l-2} \mu_{t_{i+1},t_i}| \nonumber\\
			& \leq \epsilon(t, t_{l+1}) \omega^z(t_{l+1},t_{l-1}) \epsilon(t_{l-1},s) \leq \epsilon^2(t,s) 2^z \omega^z(t,s)k^{-z}\,.
		\end{align*}
		Repeating this procedure recursively until we arrive at trivial partition $\pi_0 = \{s,t\}$ we obtain the so called maximal inequality
		\begin{equation}\label{maximal}
			\sup_{\pi} d(\mu_{t,s}, \mu^\pi_{t,s}) \leq 2^z \zeta(z) \epsilon^2(t,s) \omega^z(t,s)\;,
		\end{equation}
		where the supremum is taken over all partitions of $[s,t]$ and $\zeta(z) = \sum_{k = 1}^{\infty} k^{-z}$ is the Riemann zeta function. We claim that to show the existence of the limit~\eqref{partiotion} and its independence of the sequence of the partitions it suffices to show that
		\begin{equation}\label{cauchy}
			\sup_{|\pi|\vee |\pi'| < \varepsilon}\, d(\mu^\pi_{t,s}, \mu^{\pi'}_{t,s}) \to 0\quad\text{as}\quad \varepsilon\to 0\;.
		\end{equation}
		Indeed, this would imply that $\mu^{\pi_n}_{t,s}$ is Cauchy for every sequence of partitions $\pi_n$ with 
		$|\pi_n| \to 0.$ Then since $|\mu^{\pi_n}_{t,s}| \leq \epsilon(t,s)$ the completeness of the sets $\{a \in 
		\MM: |a| \leq C \}$ with respect to the metric $d$ would imply that $\mu^{\pi_n}$ converges to some 
		element $\varphi_{t,s}$ such that $|\varphi_{t,s}| \leq \epsilon(t,s)$. The independence of the limit from 
		the sequence of partitions is obvious once~\eqref{cauchy} holds.
		
		To show~\eqref{cauchy} we assume without loss of generality that $\pi'$ is a refinement of $\pi$ (since 
		otherwise we can simply use the triangle inequality with the term $\mu^{\pi \cup \pi'}_{t,s}$). If $\pi = \{s =t_0 < t_1 < \dots < t_k = t\}$ we define
		$$M^l_{t,s} = \prod_{i = l}^{k-1} \mu^{\pi' \cap [t_i, t_{i+1}]}_{t_{i+1},t_i} \prod_{i = 0}^{l-1} 
		\mu_{t_{i+1},t_i}\;.$$
		Note that $M^k_{t,s} = \mu^\pi_{t,s}$ and $M^0_{t,s} = \mu^{\pi'}_{t,s}$. Then, by the triangle inequality, 
		submultiplicativity and the maximal inequality~\eqref{maximal}:
		\begin{align*}
			d(\mu^\pi_{t,s}, \mu^{\pi'}_{t,s}) &\leq \sum_{i = 0}^{k-1} d(M^i_{t,s}, M^{i+1}_{t,s}) \leq \epsilon^2(t,s) \sum_{i = 0}^{k-1} d(\mu_{t_{i+1},t_i}, \mu^{\pi' \cap [t_i, t_{i+1}]}_{t_{i+1},t_i}) \\ 
			& \leq 2^z \epsilon^4(t,s) \zeta(z) \sum_{i = 0}^{k-1} \omega ^z(t_{i+1},t_i)\\ 
			&\lesssim_{z,T} \omega(t,s) \max_i\{\omega^{z-1}(t_{i+1},t_i)\}\,\to 0\quad\text{as}\quad |\pi| \to 0\;,
		\end{align*}
		thus showing~\eqref{cauchy}.
		
		To show~\eqref{ineq:mu_phi0} it is enough to take the limit in \eqref{maximal} as $|\pi| \to 0$. Note that by Remark~\ref{topology2}, the multiplication of $\mu^{\pi_1}_{t,r}$ with $\mu^{\pi_2}_{r,s}$ is continuous with respect to $d$. This, together with the independence of the limit in~\ref{partiotion} with respect to the sequence of partitions, immediately implies the multiplicativity of $\varphi$ and thus $\varphi \in \mathrm{BG}_2(0,T; \MM)$.
		
		If $\mu_{t,t} = \mathds{1}$ then it is clear from \eqref{ineq:mu_phi0} that $\varphi_{t,t} = \mathds{1}$. One can also use~\eqref{ineq:mu_phi0} to show that this together with continuity of $\mu:\Delta_2 \to (\MM,d)$ will imply continuity of $\varphi:\Delta_2 \to (\MM,d)$. Details are left to the reader.
		
		\bigskip

		\textit{Uniqueness}:
		Let $\psi \in \mathrm{BG}_{2}(0,T;\MM)$ be another multiplicative map satisfying 
		$$d(\psi_{t,s}, \mu_{t,s}) \lesssim_{z,T} \omega^z(t,s)\;,$$
		then by the triangle inequality:
		$$d(\varphi_{t,s}, \psi_{t,s}) \lesssim_{z,T}\omega^z(t,s)\;.$$ 
		Let $\tilde\epsilon$ be a growth rate of $\psi$ and denote $\epsilon_0 = 1 + \tilde\epsilon +\epsilon$.
		Let $\pi = \{s =t_0 < t_1 < \dots < t_k = t\}$ be an arbitrary partition and define now for $1 \leq l \leq k-1$, $D^l_{t,s} = \phi_{t,t_{l}}\psi_{t_ls}$.
		We also denote $D^{0}_{t,s} = \varphi_{t,s}$ and $D^{k}_{t,s} = \psi_{t,s}$. Then 
		\begin{align*}
			d(\varphi_{t,s}, \psi_{t,s}) &\leq \sum_{l = 0}^{k-1} d(D^l_{t,s}, D^{l+1}_{t,s}) \leq \sum_{l = 0}^{k-1} |\varphi_{t,t_{l+1}}| d(\varphi_{t_{l+1},t_l}, \psi_{t_{l+1},t_l}) |\psi_{t_l,s}|\\
			& \lesssim \epsilon^2_0(t,s) \sum_{l = 0}^{k-1} \omega^z(t_{l+1}, t_l)\;,
		\end{align*}
		which converges to zero as $|\pi| \to 0$ similarly as before. Thus, we must have $\varphi = \psi$ which concludes the proof.
	\end{proof}
	
	\begin{rem}
		\label{rem:feyel}
	Though our result is new, it is similar in spirit to \cite[Theorem~10]{feyel2008non}, which itself is a 
	generalization of the construction of the whole signature from the lower order `iterated integrals' by Lyons 
	in \cite[Thm 2.2.1]{lyons98}\footnote{One could also show that \cite[Thm 2.2.1]{lyons98} follows from our 
	version of the multiplicative sewing lemma, Theorem \ref{thm:MSL}, by using the so-called  `neo-classical 
	inequality' to prove the growth condition \eqref{partition_growth}.}.
	
	In \cite{feyel2008non}, the authors make the assumption that the function $|\cdot|$ is Lipschitz 
	continuous with respect to the distance $d$, while in our setting this assumption is replaced by the 
	growth condition~\eqref{partition_growth} and the assumption that the closed balls of $|\cdot|$ are 
	complete with respect to $d$. This becomes a necessary modification if one wants to apply this sewing 
	lemma on infinite dimensional spaces since as we will see later, in most examples $|\cdot|$ will induce a 
	stronger topology than the one generated by $d$. It implies that the assumption of Lipschitz continuity as 
	in \cite{feyel2008non} can no longer be satisfied in general. 
	\end{rem}

	An important example of submultiplicative monoids is provided by some class of (unital) \textit{Banach algebras} $(\mathcal A,|\cdot|)$ (recall that by definition $|ab|\leq |a||b|$ for any $a,b\in\mathcal A$).
	In this case, the property \eqref{partition_growth} can be replaced by a suitable exponential growth assumption, which is easier to verify in practice.

	\begin{cor}[Multiplicative Sewing Lemma in a Banach Algebra]
	\label{cor:mult}
		Let $(\A, |\cdot|)$ be a Banach algebra with a unit $\mathds{1}$ and let $p$ be a norm on $\A$ (possibly different from $|\cdot|$). Let $d_p$ be the metric induced by $p$ i.e. $d_p(a,b) = p(a-b)$, and assume that the triple $(\A, |\cdot|, d_p)$ is a submultiplicative monoid. 
		
		Let $\mu \colon \Delta_2 \to (\A,p)$ be continuous such that $\mu_{t,t} = \mathds{1}$ for all $0 \leq t \leq T$. Assume that there exists a control  $\bar{\omega}$ so that
		\begin{equation}\label{exponential_control1}
			|\mu_{t,s}| \leq \exp\bar{\omega}(t,s),\quad\text{for every $(t,s)\in\Delta_2$.}
		\end{equation}
		Assume also that there exists another control $\omega$ and $z>1$ such that
		\begin{equation} \label{hyp_mu}
		p(\mu _{t,s} - \mu _{t,u }\mu _{u ,s}) \lesssim_{z,T} \omega^{z}(t,s)\,,
		\end{equation} 
		for every $(t,u ,s)\in\Delta_2$. Then, there exists a unique multiplicative and continuous $\varphi \colon \Delta_2 \to (\A, p)$ such that for every $(t,s)\in\Delta_2:$
		\begin{align}
		|\varphi_{t,s}| &\leq \exp\bar{\omega}(t,s)\,,\label{exponential_control2}\\
		p(\varphi _{t,s} -\mu _{t,s})	&\lesssim_{z,T} \omega^{z}(t,s)\,.\label{ineq:mu_phi}
		\end{align}
	\end{cor}
	\begin{proof}
		We can conclude from Theorem~\ref{thm:MSL} once we show $\mu \in \mathrm{BG}_2(0,T; \A)$ and that one can take a growth rate for $\mu$ to be $\epsilon(t,s) := \exp\bar{\omega}(t,s)$. For this we use submultiplicativity of the norm $|\cdot|$ to deduce for every partition $\pi = \{s =t_0 < t_1 < \dots < t_k = t\}$ 
		$$|\mu^\pi_{t,s}| \leq \prod_{i = 0}^{k-1} |\mu_{t_{i+1},t_i}| \leq \prod_{i = 0}^{k-1} \exp\bar{\omega}(t_{i+1},t_i) = \exp\big\{\sum_{i = 0}^{k-1}\bar\omega(t_{i+1},t_i)\big\} \leq \exp\{\bar{\omega}(t,s)\},$$
		which concludes the proof.
	\end{proof}

\subsection{A new proof of Tanabe/Sobolevskii's Theorem (proof of Theorem \ref{thm:tanabe})}

We start with a lemma.
	\begin{lem}\label{topology3}
		Let $(X,|\cdot|_X), (Y,|\cdot|_Y)$ be two reflexive Banach spaces such that $X\subseteq Y$ and $|x|_Y \leq |x|_X$ for all $x \in X$. Define $\A = \LL(X)\cap \LL(Y)$ with the norm 
		$$|\cdot|_\A = \max \{|\cdot |_{\LL(X)},  |\cdot |_{\LL(Y)}\}\,,$$
		and seminorm $p(\cdot) = |\cdot|_{\LL(X,Y)}$, then $(\A, |\cdot|_\A, p)$ is a submultiplicative monoid. 
	\end{lem}
	\begin{proof}
		The fact that $(\A, |\cdot|_\A)$ is an algebra is trivial. For submultiplicativity of $p$:
		\begin{align*}
			p(ab) &= |ab|_{\LL(X,Y)} \leq |a|_{\LL(Y,Y)} |b| _{\LL(X,Y)} \leq |a|_\A p(b)\;,\\
			p(ab) &= |ab|_{\LL(X,Y)} \leq |a|_{\LL(X,Y)} |b| _{\LL(X,X)} \leq p(a)|b|_\A\,.
		\end{align*}
		Now, without loss of generality let us show that the unit ball $B = \{a \in \A: |a|_\A \leq 1 \}$ is complete with respect to $p$. Let $(a_n)$ be a Cauchy sequence for $p$ with $|a_n|_\A \leq 1$ for every $n\geq 0$. By completeness of $\LL(X,Y)$
		we infer the existence of a limit $a\in \LL(X,Y)$ such that $p(a_n-a)\to0$. It remains to show that $a$ belongs to $B$. Since the operator-norm topology is stronger that the weak operator topology, we have that for all $x\in X,y^*\in Y^*$:
		\begin{equation*}
			\langle y^*,a_n x\rangle \to \langle y^*,a x\rangle\,.
		\end{equation*}
		The fact that $Y$ is reflexive implies that the unit ball in $\LL(Y)$ is compact with respect to the weak operator topology (see \cite[Thm 2.19]{choi2008locally}), therefore since $|a_n|_{\LL(Y)} \leq 1$ there exists $b \in \LL(Y)$ such that $|b|_{\LL(Y)} \leq 1$ and for each $y\in Y,y^*\in Y^*$, it holds
		\begin{equation*}
			\langle y^*,a_n y\rangle \to \langle y^*,b y\rangle\,.
		\end{equation*}
		Now, $X \subset Y$, so that convergence in the weak operator topology in $\LL(Y)$ implies convergence in the weak operator topology in $\LL(X,Y)$ and hence $a = b$, thus $|a|_{\LL(Y)} \leq 1$. Similarly, using the fact that $Y^* \subset X^*$ we have that the weak operator topology in $\LL(X,Y)$ is stronger than the weak operator topology in $\LL(X)$, thus $|a|_{\LL(X)} \leq 1$ and therefore $a \in B$.
	\end{proof}

With this lemma, we can now proceed to the proof of one of our main results.
	
\begin{proof}[Proof of Theorem \ref{thm:tanabe}]
We introduce the Banach algebra $\A = \LL(\mathcal{X}_0)\cap \LL(\mathcal{X}_1)$ whose norm is defined as $|\cdot |_\A = \max \{|\cdot |_{\LL(\mathcal{X}_0)},  |\cdot |_{\LL(\mathcal{X}_1)}\}.$
By Lemma~\ref{topology3} if we further define $p = |\cdot|_{\LL(\mathcal{X}_1,\mathcal{X}_0)}$ then the triple $(\A, |\cdot|, p)$ is a submultiplicative monoid.

	Next, from Assumption \ref{L1}, it is classical (see~\cite{pazy1983semigroups,goldstein1985semigroups}) that for each $t \in [0,T],$ one can define an analytic semigroup $e^{sL_t} \in \A.$ It is given by the so-called Dunford-Taylor integral formula (see \cite{pazy1983semigroups,goldstein1985semigroups}):
	\begin{equation}
	\label{formula:exp}
	e^{tL_{s}}=\frac{1}{2\pi i}\int_{\mathscr C }e^{\zeta t}(\zeta  -L_s)^{-1} d \zeta \;,
	\end{equation}
	where $\mathscr C$ is any contour running from $\infty e^{-i\theta }$ to $\infty e^{i\theta }$ in the sector $\Sigma_{\vartheta,\lambda}$ for $\theta \in (0, \pi/2 + \vartheta)$. Moreover, we have 
	\begin{equation}\label{bound:lambda}
	|e^{t L_s}|_{\LL(\mathcal{X}_0)},\,|e^{t L_s}|_{\LL(\mathcal{X}_1)}\leq e^{\lambda t }\,,\quad \forall t \geq 0\,,
	\end{equation}
	for some constant $\lambda$ being the same as in $\Sigma_{\vartheta,\lambda}$. In addition, the following estimates hold
	\begin{equation}\label{estimates_exp}
	|e^{(t-s)L_s}-\id|_{\LL( \mathcal{X}_1, \mathcal{X}_0)}\lesssim_T |t-s|\,,\quad\text{and}\quad |e^{(t-s)L_s}|_{\LL( \mathcal{X}_0, \mathcal{X}_1)}\lesssim_T |t-s|^{-1}\,.
	\end{equation}

		We now show an auxiliary estimate that is going to be useful: for every $u ,s \in [0,T]$ and $\tau >0$ the following estimate holds in $|\cdot|_{\LL(\mathcal{X}_1,\mathcal{X}_0)}:$
		\begin{equation}\label{e:perturb}
			p(e^{\tau L_s}-e^{\tau L_u}) \lesssim \tau\, \omega^\varrho(u ,s)\,.
		\end{equation}  
		Indeed, using the formula \eqref{formula:exp}, we have
		\begin{align*}
			e^{\tau L_s}-e^{\tau L_u}
			=
			\frac{1}{2\pi i}\int_{\mathscr C}e^{\zeta \tau }(\zeta -L_u)^{-1}(L_s-L_u)(\zeta -L_s)^{-1}d\zeta \,.
		\end{align*}
		By~\ref{L1}, $|(\zeta -L_u)^{-1}|_\A \lesssim 1$, and using H\"older continuity~\ref{L3} and submultiplicativity of $p$ we indeed get
		\begin{align*}
			p(e^{\tau L_s}-e^{\tau L_u}) &\leq \frac{1}{2\pi}\int_{\mathscr C}|e^{\zeta \tau }|\, |(\zeta -L_u)^{-1}|_\A\; p(L_s-L_u)|(\zeta -L_s)^{-1}|_\A d|\zeta|\\
			&\lesssim \omega^\varrho(u ,s)\int_{\mathscr C} \frac{|e^{\zeta \tau }| d|\zeta| }{(1+|\zeta|)^2} \lesssim \tau\, \omega^\varrho(u ,s)\;,
		\end{align*}
		which shows \eqref{e:perturb}.
		
		Next, in order to apply the multiplicative sewing lemma we define
		\begin{align*}
			\mu _{t,s}:=e^{(t-s)L_s}\,,\quad (t,s)\in\Delta_2\,.
		\end{align*}
		Clearly $\mu_{t,t} = \id$ for every $t \in [0,T]$. By \eqref{bound:lambda} we have $|\mu_{t,s}|_\A \leq e^{\lambda(t-s)}$ for some $\lambda \in \R$, and since $\bar{\omega}(t,s) = \lambda(t-s)$ is a control then $\mu$ satisfies~\eqref{exponential_control1}. We now show that $\mu :\Delta_2 \to (\A, p)$ is continuous. For pairs $(t,s),(v,u) \in \Delta_2$, assuming without loss of generality that $t-s-v+u > 0$, we have
		\begin{align*}
			p(\mu_{t,s} - \mu_{v,u}) &\leq p\big(e^{(v-u)L_s}(e^{(t-s-v+u)L_s} - \id)\big)+ p(e^{(v-u)L_s} - e^{(v-u)L_u})\\
			&\lesssim |t-s-v+u| + |v-u|\, \omega^\varrho(u ,s)\,,
		\end{align*}
		where for the first term we used \eqref{estimates_exp} and for the second term we use~\eqref{e:perturb}. This clearly implies continuity after taking $|t-v|+|s-u| \to 0$. Now, note that
		\[
		\mu_{t,s} - \mu_{t,u}\mu_{u ,s}=(e^{(t-u)L_s}-e^{(t-u)L_u })e^{(u -s)L_s}\,.
		\]
		Thus,
		\begin{align*}
			p(\mu_{t,s}-\mu_{t,u }\mu_{u,s })
			&\leq p\left(e^{(t-u)L_s}-e^{(t-u)L_u }\right)|e^{(u -s)L_s}|_\A\\
			&\lesssim (t-u)\, \omega^\varrho(u ,s) \leq (t-s)\, \omega^\varrho(t,s)\;,
		\end{align*}
		where going from the second to the third line we used~\eqref{e:perturb}. Since both $|t-s|$ and $\omega(t,s)$ are controls then so is $|t-s|^a\omega^b(t,s)$ for all $a,b > 0$ such that $a+b \geq 1$ (see \cite[p.22]{FV10}). In particular this is true for $a = 1/(1+\varrho)$ and $b = \varrho/(1+\varrho)$. Therefore,
		\begin{align*}
			p(\mu_{t,u ,s}-\mu_{t,u }\mu_{t,u }) \lesssim \Big(|t-s|^{1/(1+\varrho)}\omega^{\varrho/(1+\varrho)}(t,s)\Big)^{1+\varrho}\;,
		\end{align*}
		thus~\eqref{hyp_mu} is satisfied with $z = 1+\varrho$.
		We can thus apply Corollary \ref{cor:mult}, and obtain the existence of the unique continuous multiplicative function
		$S\colon \Delta_2 \to (\A, p)$ 
		such that $S_{t,t} = \id$ for every $t \in [0,T]$ and there exists $C > 0$ such that for all $(t,s) \in \Delta_2$, $|S_{t,s}|_\A \leq e^{\lambda(t-s)}$ and
		\begin{equation}
			|S _{t,s}-\mu_{t,s}|_{\LL(\mathcal{X}_1,\mathcal{X}_0)} \leq C |t-s|\,\omega^\varrho(t,s) \label{phi-mu}\;.
		\end{equation} 
		One can use density of $\mathcal{X}_1$ in $\mathcal{X}_0$ and continuity of $S$ with respect to the topology induced by the $\LL(\mathcal{X}_1,\mathcal{X}_0)$-norm to prove that $S \in \CC_2(0,T;\LL_s(\mathcal{X}_0))$, proving that $S$ satisfies~\ref{P1} and~\ref{P2}. 
		To show~\ref{P3} we simply use~\eqref{phi-mu}: 
		\begin{align*}
			|S_{t,s} - \id|_{\LL(\mathcal{X}_1,\mathcal{X}_0)} &\leq |S_{t,s} - \mu_{t,s}|_{\LL(\mathcal{X}_1,\mathcal{X}_1)} + |\mu_{t,s} - \id|_{\LL(\mathcal{X}_1,\mathcal{X}_0)}\\
			&\lesssim |t-s|\omega^\varrho(t,s) + |t-s| \\
			&\lesssim_T |t-s|\,.
		\end{align*}
		We now show~\ref{P4}. Let $x\in\mathcal{X}_1$ and take any $\epsilon >0$, then by multiplicativity:
		\begin{align}
			\epsilon^{-1}(S _{t+\epsilon ,s}-S _{t,s})x &=
			\epsilon^{-1}(S _{t+\epsilon ,t}-\id)S _{t,s}x \nonumber\\
			&=\epsilon^{-1}(\mu_{t+\epsilon ,t}-\id)S _{t,s}x + \epsilon^{-1}(S _{t+\epsilon ,t}-\mu_{t+\epsilon ,t})S _{t,s}x.\label{eq:partial_S}
		\end{align}
		We have shown that $S _{t,s} \in \LL(\mathcal{X}_0) \cap \LL(\mathcal{X}_1)$ and therefore $S _{t,s}x \in \mathcal{X}_1$ for $x \in \mathcal{X}_1$. Using $\mu_{t+\epsilon,t} = e^{\epsilon L_t}$
		we conclude that the first term in~\eqref{eq:partial_S} converges to $L_tS _{t,s}x$ as $\epsilon \to 0$. For the second term, using~\eqref{phi-mu}:
		\begin{align*}
			\epsilon^{-1}|(S _{t+\epsilon ,t}-\mu_{t+\epsilon ,t})S _{t,s}x|_0 & \leq \epsilon^{-1}|S _{t+\epsilon ,t}-\mu_{t+\epsilon ,t}|_{\LL(\mathcal{X}_1,\mathcal{X}_0)}\, |S _{t,s}x|_1\\ 
			&\lesssim \omega^\varrho(t+\epsilon, t) |S _{t,s}x|_1\;,
		\end{align*}
		which vanishes as $\epsilon$ goes to $0$. Putting it all together we conclude that for every $x\in\mathcal{X}_1$ and $(t,s)\in\Delta _2$ with $s\neq t:$
		\[
		\frac{d}{dt}S _{t,s}x=L_tS _{t,s}x.
		\]
		The proof that $\frac{d}{ds}(S _{t,s}x)= - S _{t,s}L_sx$ is similar, hence \ref{P4} follows. 
		
		The first inequality of \eqref{smoothing_S} follows by Remark~\ref{rem:Lbounded}. Now, assuming in addition that~\ref{L4} holds, we will show that $S$ satisfies~\ref{P5}. Let $x \in \mathcal{X}_1$, since $\mu_{t,s} \in \LL(\mathcal{X}_1)$ 
		we can use~\ref{P4} to differentiate in $\mathcal{X}_0$:
		\begin{align*}
			\frac{d}{dr} (S_{t,r} \mu_{r,s}x)  = S_{t,r}(L_s - L_r)\mu_{r,s}x\, .
		\end{align*}
		Thus, integrating with respect to $r$ over the interval $[s,t]$ we obtain the equation
		\begin{equation}\label{phi:equation1}
			S_{t,s}x = \mu_{t,s}x + \int_s^t S_{t,r}(L_r - L_s)\mu_{r,s}x\, dr\; .
		\end{equation}
		By density of $\mathcal{X}_1$ in $\mathcal{X}_0$ we can extend this integral equation for all $x \in \mathcal{X}_0$. For simplicity denote $|\cdot| = |\cdot|_{\LL(\mathcal{X}_0,\mathcal{X}_1)}$. Since~\eqref{phi:equation1} holds for all $x \in \mathcal{X}_0$ we can use this equation together with $|S_{t,s}|_{\LL(\mathcal{X}_1)} \lesssim_T 1$ to obtain:
		\begin{align}
			|S_{t,s}| &\lesssim |\mu_{t,s}| + \int^t_s |S_{t,r}|\, |L_r-L_s|_{\LL(\mathcal{X}_1,\mathcal{X}_0)} |\mu_{r,s}|\,dr \nonumber\\
			&\lesssim |t-s|^{-1} + \int^t_s |S_{t,r}|\, \omega^\varrho(r,s)|r-s|^{-1}dr\,,\label{phi:equation2}
		\end{align}
		where we used H\"older regularity~\ref{L3} and a semigroup bound $|\mu_{t,s}| \lesssim |t-s|^{-1}$. Now multiply both sides of~\eqref{phi:equation2} by $|t-s|$, set $f(s) = |t-s|\,|S_{t,s}|$ to obtain:
		\begin{align*}
			f(s) &\lesssim 2 + \int^{t}_s f(r) |t-s||t-r|^{-1}\,\omega^\varrho(r,s)|r-s|^{-1}dr\,.
		\end{align*}
		We now use Gronwall's inequality\footnote{A proof for Gronwall's inequality, where the lower limit of integration is a variable, is almost identical to the proof of the classical Gronwall's inequality, modulo changing a sign of the integrating factor.} to derive:
		\begin{align*}
			f(s) &\lesssim \exp\Big(\int^t_s |t-s||t-r|^{-1}\,\omega^\varrho(r,s)|r-s|^{-1}dr\Big)\nonumber\\
			&= \exp\Big(\int^t_s \frac{\omega^\varrho(r,s)dr}{r-s}  +\int^t_s\frac{\;\omega^\varrho(r,s)dr}{t-r}\Big)\;.
		\end{align*}
		The first integral in the exponential is uniformly bounded on $[0,T]$ by assumption~\ref{L4}, and for the second integral we simply observe that there is no blow up of the integrand over the integrated area. Therefore, $f(s) \lesssim_T 1$ and recalling that $f(s) = |t-s|\,|S_{t,s}|$ we conclude $|S_{t,s}| \lesssim_T |t-s|^{-1}$, thus finishing the proof.
\end{proof}

	\begin{rem}
	\label{rem:unify}
		A similar construction of $S$ satisfying \ref{P1}, \ref{P2}, \ref{P3}, \ref{P4} through the limiting infinite 
		product of $e^{(t-s)L_s}$ for a family of operators $L_t$ with $\Lambda=1$ is present in the paper of T.Kato~\cite{kato1953integration}, but only in the case where the family $L_t$ is of bounded 
		variation (as a function of $t$) in $\LL(\mathcal{X}_1,\mathcal{X}_0)$. 
		Tanabe himself proved the existence of a propagator $S$ 
		for all operators satisfying Assumption~\ref{ass:L_t} but the more general $1/\varrho$-variation 
		assumption is replaced by $\varrho$-H\"older regularity. His approach is different and relies on the 
		construction of the approximate solutions to equation~\eqref{non-autonomous}. These approximations 
		do not use a limiting infinite product which allows to relax the dissipativity assumption on $L_t$. The 
		above proof using the multiplicative sewing lemma allows unifying constructions of $S$ when restricted 
		to operators with $\Lambda = 1$.
	\end{rem}

	The following straightforward generalization of Theorem \ref{thm:tanabe} will be extensively used in the sequel. 
	
	\begin{thm}\label{Tanabe2}
	Let $(\BB_\alpha, |\cdot|_\alpha)_{\alpha \in \R}$ be a monotone family of interpolation spaces such that $\BB_\alpha$ is reflexive for each $\alpha \in \R$. Assume that there exists a set of indices $K:=\{k_-,k_-+1,\dots,k_+\}\subset \Z$ such that for every $k \in K$,
	$(L_t)_{t \in [0,T]}$ satisfies Assumption \ref{ass:L_t} where $(\mathcal{X}_0,\mathcal{X}_1)$ is replaced by $(\BB_k,\BB_{k+1})$ (with control $\omega$ possibly depending on $k\in \Z$).

Then the family $S,$ as constructed in Theorem \ref{thm:tanabe}, extends uniquely to a propagator on the full range $(\BB_\alpha )_{\alpha \in [k_-,k_+]}$ (in the sense of Definition \ref{def:part}). More explicitly, we have
	\begin{enumerate}[label = (P\arabic**)]
			\item\label{P1*} $S \in \CC(\Delta_2, \LL_s(\BB_\alpha))$ and there exists $\lambda_\alpha$ such that $\|S_{t,s}\|_{\LL(\BB_\alpha)} \leq e^{\lambda_\alpha(t-s)}$ for every $\alpha \in[k_-,k_++1],$ 
			\item\label{P2*} $S_{t,t} = \id$ and $S_{t,s} = S_{t,u}S_{u,s}$ for all $(s,u,t) \in \Delta_3$.
			\item\label{P3*} For $(s,t) \in \Delta_2$, $\alpha \in[k_-,k_+],$  and $x \in \BB_{\alpha+1}$ the following differential equations hold true in $\BB_{\alpha}$:
			\[
			\frac{d}{dt}S_{t,s}x = L_tS_{t,s}x \;,\qquad \frac{d}{ds}S_{t,s}x=-S_{t,s}L_sx\; .
			\]
			\item\label{P4*} For all $(s,t) \in \Delta_2$, $\alpha, \beta \in[k_-,k_++1]$  and $\beta \geq \alpha$ the propagator $S_{t,s}$ maps $\BB_\alpha$ to $\BB_\beta$, and in addition for $\sigma \in [0,1]$ the following smoothing inequalities are true:
			\begin{equation}\label{P:smoothing}
				|S_{t,s}x|_\beta \lesssim |t-s|^{-(\beta-\alpha)}|x|_\alpha, \qquad |(S_{t,s}-\id)x|_\alpha \lesssim |t-s|^\sigma |x|_{\alpha+\sigma} \,.
			\end{equation}
		\end{enumerate}
	\end{thm} 
\begin{proof}
	If $\alpha $ is an integer, the proof follows by Theorem \ref{thm:tanabe} by taking $(\mathcal{X}_0,\mathcal{X}_1) = (\BB_\alpha,\BB_{\alpha+1})$. The general case follows by interpolation~\eqref{interpolation}, since $(\BB_\alpha )_{\alpha \in\R}$ is a monotone family of interpolation spaces.
	\end{proof}
	
\begin{exa}\label{ex:full_propagator}
Consider the family of operators given in Example \ref{ex:family1} and for each $k\in\Z,$ let $\BB_k = 
H^{2k,p}(\T^n)$ where $p \in (1,\infty),$ and assume that for all $t \in [0,T]$,
$a(t,\cdot) \in \CC^{\infty}(\T^n; \R^{n \times n}).$
Then, one can check similarly as in Example~\ref{ex:family1} that the assumptions of Theorem \ref{Tanabe2} are fulfilled for $K=\Z,$ which means that
the associated propagator satisfies the properties \ref{P1*}---\ref{P4*} for the full scale $(\BB_\alpha )_{\alpha \in\R}.$
\end{exa}

	\subsection{Lie-Trotter product formula}
	In this subsection we present another application of the multiplicative sewing lemma - the proof of a 
	Lie-Trotter-type formula for families $(S_{t,s})$ such that \ref{P1}--\ref{P4} hold. We shall call such a 
	family a quasipropagator. \footnote{Though this result is a non-trivial consequence of Theorem 
	\ref{thm:MSL}, it will not be needed in the rest of the paper.}
	
	Recall that a semigroup $S_{t}$ is called contractive if $|S_{t}| \leq 1$ for all times $t$. The classical 
	Lie-Trotter product formula for the exponential of two (not necessary commuting) operators states that, if 
	$A, B : \mathcal{X}_0 \to \mathcal{X}_0$ are two closed (unbounded) operators with common domain $\mathcal{X}_1$ such that they 
	form (respectively) contractive semigroups $S^{A}_t$ and $S^{B}_t,$ and such that their sum $A+B$ is a 
	closable operator generating a contractive semigroup $S^{A+B}_t,$ then for every $t$ and every $x \in 
	\mathcal{X}_0$:
	\begin{equation}\label{trotter}
		\lim_{n\to\infty}\prod_{i=0}^{n-1}S^A_{\frac{t}{n}}S^{B}_{\frac{t}{n}} x=S^{A+B}_{t} x\,.
	\end{equation}
	A proof of this result which is based on so called Chernoff $\sqrt{n}-$Lemma can be found in \cite[p.\ 
	53]{goldstein1985semigroups}.
	Here we present an alternative proof based on the multiplicative sewing lemma under the form of Corollary 
	\ref{cor:mult} together with simple estimates on semigroups and commutators. In addition, our proof 
	extends to the quasipropagators and does not assume contractivity of the underlying semigroups but 
	only the exponential bounds like in Definition~\ref{def:topologies}, which is weaker. Strictly speaking, every 
	semigroup that satisfies such an exponential bound can be shifted to become contractive, but our proof 
	does not require that and can be applied instantly.
	
	\begin{ass}\label{ass:lie_trotter}
		We are given continuously embedded reflexive Banach spaces $\mathcal X_2 \subseteq \mathcal X_1 \subseteq \mathcal X_0$
		and $(A_t), (B_t)$ for $t \in [0,T]$ are two families of unbounded closed operators satisfying the 
		properties \ref{L1} with $\Lambda = 1$ for the range of indices $K=\{0,1,2\}$ and \ref{L3} with $K= \{0,1\}$.
		Moreover, we assume that for every $t \in [0,T]$, the domains $D(A^i_t)$ and $D(B^i_t)$ both contain 
		$\mathcal{X}_i$ for $i = 1,2$.
	\end{ass}
	Prior to proving a generalized Lie-Trotter formula, we will need a commutator estimate of two semigroups.
	\begin{lem}[Commutator estimate]\label{lem:commutator}
		Let the families $(A_t)_{t \in [0,T]} , (B_t)_{t \in [0,T]}$ satisfy Assumption~\ref{ass:lie_trotter}. For every $v,u \in [0,T]$ define a commutator 
		\begin{align*}
			C(t,s):=\exp\{tA_v\}\exp\{sB_u \}- \exp\{sB_u \}\exp\{tA_v\}\,.
		\end{align*}
		Then the following estimate holds true uniformly in $v,u \in [0,T]$:
		\begin{equation}\label{commutator}
			|C(t,s)|_{\LL(\mathcal{X}_2,\mathcal{X}_0)} \lesssim (t \vee s)^2\,.
		\end{equation}
	\end{lem}
	\begin{proof}
		The fact that these estimates are uniform in $v,u \in [0,T]$ will follow from the uniform bounds in 
		assumption~\ref{ass:L_t}, so without loss of generality we will show this for the constant in time 
		families $A$ and $B$. Define the second order Taylor remainder 
		\begin{align*}
			R^A(t):=e^{tA}-\id -tA \, ,
		\end{align*}
		and similarly define $R^B$. Then such a remainder satisfies
		\begin{equation}
			\label{estim:RA}
			|R^A(t)|_{\LL(\mathcal{X}_k,\mathcal{X}_j)}\lesssim t^{k-j},\quad \text{for}\enskip (k,j) \in \{(2,0),(2,1),(1,0)\}\,.
		\end{equation}
		Indeed, for $(k,j) = (1,0)$ we can see that by the triangle inequality and~\eqref{estimates_exp}:
		\begin{align*}
			|R^A(t)|_{\LL(\mathcal{X}_1,\mathcal{X}_0)} \leq |\exp \{tA\}-\id|_{\LL(\mathcal{X}_1,\mathcal{X}_0)} + t |A|_{\LL(\mathcal{X}_1,\mathcal{X}_0)} \lesssim t + t |A|_{\LL(\mathcal{X}_1,\mathcal{X}_0)} \lesssim t\,.
		\end{align*}
		Likewise, Assumption~\ref{L1} and Theorem \ref{Tanabe2} with $k = 1$ implies that $|\exp\{tA\}-\id|_{\LL(\mathcal{X}_2,\mathcal{X}_1)} \lesssim |t-s|$ and we can use it to show~\eqref{estim:RA} for $(k,j) = (2,1)$.
		For  $(k,j) = (2,0)$ we simply use the formula $\frac{d}{dt} e^{tA} = e^{tA} A$ twice and then the uniform  boundedness of $|e^{tA}|_{\LL(\mathcal{X}_0)}$ in $t$. For $x \in \mathcal{X}_2$
		\begin{equation} \label{ExponentialRemainder}
			\big|R^A(t) x\big|_0 = \Big|\int_0^t\int_0^s e^{rA} A^2 x\,dr ds\Big|_0 \lesssim \int_0^t\int_0^s dr ds\, \big|A^2 x\big|_0 \lesssim t^2 |x|_2 \,.
		\end{equation}
		We use this now to rewrite our commutator $C(t,s)$ as:
		\begin{align*}
			C(t,s) &= \big(\id + tA + R^A(t)\big)\big(\id + sB + R^B(s)\big) \\
			&\quad- \big(\id + sB + R^B(s)\big)\big(\id + tA  + R^A(t)\big)  \\
			& = ts(AB-BA) +R^A(t)+ sR^A(t)B + R^A(t)R^B(s)\\
			&\quad-R^B(s) - tR^B(s)A - R^B(s)R^A(t)\,.
		\end{align*}
		Using $|AB-BA|_{\LL(\mathcal{X}_2, \mathcal{X}_0)} < \infty$ and the bounds \eqref{estim:RA} the result follows.
	\end{proof}
	
	\begin{thm}[Lie-Trotter product formula]\label{thm:lie-trotter}
		Let families $(A_t)_{t \in [0,T]} , (B_t)_{t \in [0,T]}$ satisfy Assumption~\ref{ass:lie_trotter}. Assume that
		$(A_t+B_t)_{t \in [0,T]}$ is a family of closable operators and that this closure also satisfies Assumption~\ref{ass:lie_trotter}.
		Then $(A_t)$, $(B_t)$ and the closure of $(A_t+B_t)$ all generate quasipropagators which 
		we respectively call $S^{A}_{t,s}$, $S^{B}_{t,s}$ and $S^{A+B}_{t,s}$ for $(t,s) \in \Delta_2$. Moreover, for any 
		sequence of partitions $\PP^n$ of $[s,t]$ with $|\PP^n| \to 0$ as 
		$n \to \infty$ the following Lie-Trotter formula holds for every $(t,s) \in \Delta_2$ and every $x \in \mathcal{X}_0$:
		\begin{equation}\label{e:lie-trotter}
			S^{A+B}_{t,s}x = \lim_{n\to\infty} \prod_{[u,v] \in \PP^n} S^{A}_{v,u}S^{B}_{v,u} x\;,
		\end{equation}
		where the limit is taken in $\mathcal{X}_0$ and the product over $[u,v] \in \PP^n$ means that it runs over all two neighbouring points in the partition.
	\end{thm}
	\begin{proof} Take $\A = \LL(\mathcal{X}_2) \cap \LL(\mathcal{X}_0)$, $|\cdot|_\A = \max \{|\cdot|_{\LL(\mathcal{X}_0)}, |\cdot|_{\LL(\mathcal{X}_2)}\}$ and $p(\cdot) = |\cdot|_{\LL(\mathcal{X}_2,\mathcal{X}_0)},$ then using again Lemma~\ref{topology3} we have that $(\A, |\cdot|_\A,p)$ is a submultiplicative monoid. Without loss of generality we assume that the resolvent sets of $A_t$ and $B_t$ both contain $\Sigma_{\vartheta,\lambda}$ for the same $\vartheta>0$ and $\lambda \in \R$. Moreover, without loss of generality we also assume that the control $\omega$ and exponent $\varrho$ in~\ref{L3} is the same for $A$ and $B$. Define $\mu \colon \Delta _2\to \A$ by
		\begin{align*}
			\mu _{t,s}:= S^{A}_{t,s}S^{B}_{t,s}\,, \quad (t,s)\in\Delta_2\,.
		\end{align*}
		It follows from Remark~\ref{topology2} and by Tanabe's Theorem that $|\mu_{t,s}|_\A \leq e^{2\lambda(t-s)}$. 
		We now study the quantity $p(\mu_{t,u ,s}-\mu_{t,u }\mu_{u ,s})$. Denote $\mu^A_{t,s} = 
		e^{(t-s)A_s}$ and $\mu^B_{t,s} = e^{(t-s)B_s}$. First, by the multiplicativity of $S^A$ and $S^B$ we have 
		for any $(t,u ,s)\in\Delta _3$,
		\[\mu_{t,s} - \mu_{t,u}\mu_{u ,s}=S^{A}_{t,u }(S^{A}_{u ,s}S^{B}_{t,u }-S^{B}_{t,u }S^{A}_{u ,s})S^{B}_{u ,s}\,.
		\]
		Second, by the triangle inequality and submultiplicativity of $p$:
		\begin{align}
			p(\mu_{t,s} - \mu_{t,u}\mu_{u ,s}) &\leq |S^{A}_{t,u }|_\A\, p \left(S^{A}_{u ,s}S^{B}_{t,u }-S^{B}_{t,u }S^{A}_{u ,s}\right) |S^{B}_{u ,s}|_\A\nonumber\\
			& \leq e^{\lambda(t-s)} p (S^{A}_{u ,s}S^{B}_{t,u }-S^{B}_{t,u }S^{A}_{u ,s})\nonumber\\
			& \lesssim_{T}\; p(S^{A}_{u ,s} - \mu^A_{u ,s}) |S^{B}_{t,u }|_\A + |\mu^A_{u ,s}|_\A\, p(S^{B}_{t,u } - \mu^B_{t,u }) \nonumber\\
			&\quad + p(\mu^A_{u ,s}\mu^B_{t,u } - \mu^B_{t,u }\mu^A_{u ,s}) \nonumber\\
			&\quad + p(\mu^B_{t,u }-S^{B}_{t,u })|\mu^A_{u ,s}|_\A + |S^{B}_{t,u }|_\A\, p(\mu^A_{u ,s} - S^{A}_{u ,s})	\nonumber\\
			& \lesssim_{T}\; p(S^{A}_{u ,s} - \mu^A_{u ,s}) + p(\mu^B_{t,u }-S^{B}_{t,u }) + p\big(C(u -s, t-u)\big)\label{e:updelta_estimate}\;.
		\end{align}
		Here the commutator $$C(u -s,t-u) = \exp\{(u -s)A_s\}\exp\{(t-u)B_u \} - \exp\{(t-u)B_u \}\exp\{(u -s)A_s\}\,,$$ is exactly of the form considered in Lemma~\ref{lem:commutator}, so we have
		\begin{align*}
			p\big(C(u -s, t-u)\big) \lesssim |t-s|^2\;.
		\end{align*} 
		The construction of $S^A$ and $S^B$ using Theorem~\ref{thm:tanabe} guarantees that 
		\begin{align*}
			p(S^{A}_{u ,s} - \mu^A_{u ,s})\lesssim_T |u -s|\,\omega^\varrho(u ,s)\,,\quad p(S^{B}_{t,u } - \mu^B_{t,u })\lesssim_T |t-u|\,\omega^\varrho(t,u )\;,
		\end{align*} 
		for some $\varrho \in (0,1)$. Applying these to~\eqref{e:updelta_estimate} we get:
		\begin{align*}
			p(\mu_{t,s} - \mu_{t,u} \mu_{u ,s}) \lesssim_{\lambda, T}  |t-s|\omega^\varrho(t,s) + |t-s|^2 \lesssim |t-s|\bar\omega^\varrho(t,s)\,,
		\end{align*}
		where $\bar\omega(t,s) = \omega(t,s) + |t-s|^{\frac 1\varrho}$ is a control. We now have all the necessary ingredients to apply the multiplicative sewing lemma (Corollary \ref{cor:mult}) which implies that the limit $\lim_{|\PP| \to 0} 
		\prod_{[u,v] \in \PP} S^{A}_{v,u}S^{B}_{v,u}$ exists with respect to the semi-norm $p$, is multiplicative  and independent of the partitions. Call this limit $\varphi$. 
		First, we show that $\varphi_{t,s}x = S^{A+B}_{t,s}x$ for $x \in \mathcal{X}_2$ and for that it is enough to show 
		$\frac{d}{dt}\varphi_{t,s}x = \frac{d}{dt}S^{A+B}_{t,s}x$ since $\varphi_{t,t} = S^{A+B}_{t,t} = \id$. Like in the 
		proof of Tanabe's theorem note that there is $R(\epsilon)$ such that $|R(\epsilon)|_0 \lesssim 1$ 
		uniformly in $(t,s)$ and $\epsilon$ such that
		\begin{align*}
			\epsilon^{-1} (\varphi_{t+\epsilon,s}x - \varphi_{t,s}x) &= \epsilon^{-1} (\mu_{t+\epsilon,t} - \id) \varphi_{t,s}x + \omega^\varepsilon(t+\epsilon,t)R(\epsilon) \\
			&=\epsilon^{-1} \big((S^{A}_{t+\epsilon,t} - \id)S^{B}_{t+\epsilon,t} (S^{B}_{t+\epsilon,t} - \id)\big) \varphi_{t,s}x \\
			&\quad + \bar\omega^\varrho(t+\epsilon,t)R(\epsilon) \\
			&\to (A_t + B_t) \varphi_{t,s}x\quad\text{as}\quad \epsilon \to 0\,.
		\end{align*}
		By the assumption that $S^{A+B}$ is a quasipropagator we have $\frac{d}{dt}S^{A+B}_{t,s} = (A_t + B_t) S^{A+B}_{t,s}$ on $\mathcal{X}_1$. Using the inclusion $\mathcal{X}_2 \subset \mathcal{X}_1$ this implies that $\varphi_{t,s}x = S^{A+B}_{t,s}x$ for $x \in \mathcal{X}_2$, since $\varphi$ satisfies the same equation on $\mathcal{X}_2$. It remains to show that the limit in 
		\eqref{e:lie-trotter} can be taken for all $x \in \mathcal{X}_0$. Let $x \in \mathcal{X}_0$ and denote $S^{A+B,n}_{t,s} = 
		\prod_{[u,v] \in \PP^n} S^{A}_{v,u}S^{B}_{v,u}$. Since $\mathcal{X}_2$ is dense in $\mathcal{X}_0$, for every 
		$\epsilon > 0$ we can choose  $y \in \mathcal{X}_2$ such that $|x-y|_0 \leq \epsilon$. But then 
		\begin{align*}
			\big|(S^{A+B}_{t,s} - S^{A+B,n}_{t,s})x\big| &\leq \big|S^{A+B}_{t,s} - S^{A+B,n}_{t,s}\big|_{\LL(\mathcal{X}_2,\mathcal{X}_0)} |y|_2 + \big| S^{A+B,n}_{t,s}\big|_{\LL(\mathcal{X}_0)} |x-y|_0
			\\
			& \lesssim p(S^{A+B}_{t,s} - S^{A+B,n}_{t,s})|y|_2 + e^{\lambda(t-s)} \epsilon\,.
		\end{align*}
		Letting first $n \to \infty$ and then $\epsilon \to 0$ gives the desired result.
	\end{proof}
	Note that in case when the families $(A_t)$ and $(B_t)$ are just constant operators $A$ and $B,$ we 
	simply recover the usual Lie-Trotter formula~\eqref{trotter} for unbounded operators. Moreover, the 
	estimates are a bit better in this case since we have $S^{A}_{t,s} = \mu^A_{t,s},$ and therefore we can 
	obtain the bound
	\begin{align*}
		\big|e^{(t-s)(A+B)} - e^{(t-s)A}e^{(t-s)B}\big|_{\LL(\mathcal{X}_2,\mathcal{X}_0)} \lesssim e^{2\lambda(t-s)} |t-s|^2\;.
	\end{align*}
	
	\begin{rem}[Strang Splitting]
		If  $A,B$ are two infinitesimal generators which are independent of the time-like variable, then one can 
		find an even better approximation of $e^{t(A+B)}$.
		It was noticed by Strang in~\cite{strang1968construction} that for $h\geq 0$ small enough, the 
		operator $\exp\{\frac{h}{2}A\}\exp\{hB\}\exp\{\frac{h}{2}A\}$ yields an approximation of the semigroup 
		$\exp h(A+B)$ which is of higher order than the `na\"ive' choice $\exp hA\exp hB$. Towards using the 
		Sewing Lemma, for $(t,s)\in\Delta _2$ we let 
		\[
		\mu _{t,s}:=\exp\left\{\frac{t-s}{2}A\right\}\exp\big\{(t-s)B\big\}\exp\left\{\frac{t-s}{2}A\right\}\,,
		\]
		so that 
		\begin{align*}
			\mu _{t,s} - \mu _{t,u }\mu _{u ,s}
			&=
			e^{\frac{t-u}{2}A}\left[
			e^{\frac{u -s}{2}A}e^{(t-s)B}e^{\frac{t-u}{2}A}
			-e^{(t-u)B}e^{\frac{t-s}{2}A}e^{(u -s)B}
			\right]e^{\frac{u -s}{2}A}\,.
			\\
			&=:e^{\frac{t-u}{2}A}
			C^{\sharp}(u -s,t-u)e^{\frac{u -s}{2}A}\,.
		\end{align*}
		Expanding the exponentials like in Lemma~\ref{lem:commutator} one can deduce that formally:
		$C^{\sharp}(u -s,t-u)=O(|t-s|^3),$
		for an appropriate norm (the choice of $|\cdot |_{\LL(\mathcal{X}_3,\mathcal{X}_0)}$ would do if one generalizes assumption~\eqref{ass:lie_trotter} to $\mathcal X_3$, and generalizes~\eqref{ExponentialRemainder} accordingly).
		
		The Sewing Lemma implies in turn that the iterated products of $\mu$ converge faster than the former 
		first order approximation.
	\end{rem}

\section{Controlled Path according to a monotone family of interpolation spaces}
\label{sec:controlled}

In this section, we build a framework for studying rough evolution equations of the form 
\eqref{evolution_equation}. Towards this end, we will define the space of paths that locally `look like' a rough 
path $X,$ which will be the natural space where the solution of \eqref{evolution_equation} lives. In order to 
set up a mild formulation of \eqref{evolution_equation} we also need to show that there is a notion of 
integration with respect to $\X$ on such spaces. This will be done using a version of the classical sewing 
lemma for semigroups and propagators \cite{gubinelli2010}. The connection between this result (which we 
will refer to as `affine sewing lemma' in the sequel) and Theorem \ref{thm:MSL} will be established, having observed the role played by the affine group as introduced in \eqref{semi-direct}-\eqref{semi-direct-product}. 

Starting from this section and till the end of the article, we fix a monotone family of interpolations spaces $(\BB_\beta)_{\beta \in \R}$. From now on, the index $\beta$ in $ \BB_\beta$ will refer to the varying parameter of the scale and the index $\alpha$ in $\BB_\alpha$ will denote some fixed level thereof. Moreover, we will fix a family of unbounded operators $(L_t)_{t\in [0,T]}$ that acts on the monotone 
family $(\BB_\beta )_{\beta \in \R}$, such that they generate the propagator $(S_{t,s})_{(t,s)\in \Delta_2}$ on the full range $(\BB_\beta )_{\beta \in \R}$ in the sense of Definition~\ref{def:part}. We do not necessarily assume that $S_{t,s}$ was constructed using the Theorem~\ref{Tanabe2}. This potentially can allow us to apply the forthcoming results to linear operators $L_t$ that do not necessarily have $\Lambda = 1$ in~\ref{L1} or to Banach spaces $\BB_\beta$ that are not necessarily reflexive.

\subsection{Affine Sewing Lemma}

The affine sewing lemma is going to allow us to define integrals of the form $z_t := \int_0^tS_{t,r}y_r\cdot d\X_r $. 
Before we state the affine sewing lemma let us describe the algebraic properties of such integrals. Note that 
the linearity of the integral does not patch together nicely with the usual increment operator $\delta$ from 
\eqref{increment}. This is due to the fact that in this case:
$\dd z_{t,s} \equiv z_t-z_s = \int_s^tS_{t,r}y_r\cdot d\X_r + S_{t,s}z_s - z_s \ne \int_s^tS_{t,r}y_r\cdot d\X_r.$
Instead, we have the relation
\[\ddh z_{t,s} \equiv z_t-S_{t,s}z_s = \int_s^tS_{t,r}y_r\cdot d\X_r\;.\]

As a matter of fact, the integral $\int_s^t S_{t,r}y_r\cdot d\X_r$ has a multiplicative structure. Indeed, letting  $\beta \in \R$ (to be chosen later) and defining
\begin{equation}
\label{def:MM}
\MM(\beta):= \LL(\BB_\beta )\ltimes \BB_{\beta},
\end{equation}
we see that the multiplication of two elements $\mu _j\equiv(S_j ,x_j)\in\MM(\beta),$ $j=1,2,$
can be defined as
$\mu _1\circ \mu _2 := (S_1 S_2 , x_1 + S_ 1 x_2)\,.
$
Defining addition componentwise, we note that $\MM(\beta)$ is a near-ring, namely $(\MM(\beta),\circ)$ 
forms a monoid and the multiplication is right-distributive:
$(\mu _1+\mu _2)\circ \nu  =  \mu _1\circ \nu + \mu _2\circ \nu $
(though it is not left distributive in general).

With this at hand, and assuming that the rough convolution $\int_s^t S_{t,r}y_r\cdot d\X_r$ is meaningful, we should have
\begin{align*}
	\big(S_{t,u}, \int_u^t S_{t,r}y_r\cdot d\X_r\big)\circ \big(&S_{u,s}, \int_s^u S_{u,r}y_r\cdot d\X_r\big) 
	\\
	&= \big(S_{t,u}S_{u,s}, \int_u^t S_{t,r}y_rdX_r +  S_{t,u}\int_s^u S_{u,r}y_r\cdot d\X_r\big) 
	\\
	&= \big(S_{t,s}, \int_s^t S_{t,r}y_r\cdot d\X_r\big)\,,
\end{align*}
meaning that $\varphi_{t,s}:=(S_{t,s},\int_s^t S_{t,r}y_r\cdot d\X_r)$ is multiplicative in $\MM(\beta)$.
This suggests that using an appropriate approximation of the second component, we might be able to use 
Theorem \ref{thm:MSL} in order to construct the rough convolution.\\

	Now, prior to define the integration map we need to specify which type of integrand shall be considered in 
	the sequel. Let $\gamma \in [0,1]$ and $\alpha \in \R$, 
	we introduce the space $\ZZ^{\gamma}_\alpha$ as follows:
	$\ZZ^{\gamma}_\alpha$ consists of each 2-index element $\xi =(\xi_{t,s}) \in \CC_2^{\gamma }(\BB_\alpha) + \CC_2^{2\gamma }(\BB_{\alpha -\gamma})$ with the property that $\delta \xi \in \CC_3^{2\gamma ,\gamma }(\BB_{\alpha -2\gamma })+\CC_3^{\gamma ,2\gamma }(\BB_{\alpha -2\gamma }).$
Namely, there exist $\xi ^1,\xi ^2$ and $h^1,h^2$ with 
\begin{align}
\xi_{t,s} &= \xi ^1_{t,s} + \xi ^2_{t,s},\quad (t,s)\in\Delta _2, \label{nota:ZZ2}\\ 
\delta \xi _{t,u,s} &=h^1_{t,u,s}+ h^2_{t,u,s},\quad (t,u,s)\in\Delta _3, \nonumber
\end{align} 
and such that $[\xi ^1]_{\gamma ,\alpha } + [\xi ^2]_{2\gamma ,\alpha -\gamma }+[h^1]_{2\gamma ,\gamma ,\alpha -2\gamma }+ [h^2]_{\gamma ,2\gamma ,\alpha -2\gamma }<\infty$\,,
see \eqref{norm_3}.
The space $\ZZ^{\gamma}_\alpha$ is then equipped with the natural norm 
$$\|\xi \|_{\ZZ^{\gamma}_\alpha} := \inf_{\xi^1,\xi^2,h^1,h^2}\Big([\xi ^1]_{\gamma ,\alpha } + [\xi ^2]_{2\gamma 
,\alpha -\gamma } + [h^1]_{2\gamma ,\gamma ,\alpha -2\gamma }+ [h^2]_{\gamma ,2\gamma ,\alpha -2\gamma }\Big)\,,$$ 
where infimum is taken over every decomposition of the form \eqref{nota:ZZ2}.
(Note that analogous spaces were introduced in \cite{gubinelli2010}.) 

We also identify the space which is going to be the image of the integration map by $\E^{0,\gamma}_\alpha = \CC(\BB_\alpha) \cap \CC^\gamma(\BB_{ \alpha - \gamma})$ equipped with the norm being the maximum of norms on $\CC(\BB_\alpha)$ and $\CC^\gamma(\BB_{ \alpha - \gamma})$. With this at hand we get:

\begin{thm}[Affine Sewing Lemma] 
	\label{thm:sewing}
	Consider the propagator $(S_{t,s})_{(s,t)\in \Delta_2}$ as in Theorem~\ref{Tanabe2}, let $\alpha \in \R$ and $\gamma\in(1/3,1/2]$.
	
	There exists a unique continuous linear map $\I: \ZZ^{\gamma}_\alpha \to \E^{0,\gamma}_\alpha$ such that $\I_0(\xi) = 0$ for any $\xi \in \ZZ^{\gamma}_\alpha$ and moreover, for every $0\leq \beta < 3\gamma$, it holds
	\begin{equation}\label{e:sewing1_alt}
	| \ddh\I_{t,s}(\xi)-S_{t,s}\xi_{t,s}|_{\alpha-2\gamma + \beta}
	\lesssim \|\xi \|_{\ZZ^{\gamma}_\alpha} |t-s|^{3\gamma - \beta}\;.
	\end{equation}
	Finally, one has
	\begin{equation} \label{e:approximation}
	\I_t(\xi) = \lim_{|\pi|\to 0} \sum_{[u,v]\in \pi}S_{t,u}\xi_{v,u}\;,
	\end{equation}
	where the limit is taken in the sense of topology of $\BB_{\alpha-2\gamma}$, over arbitrary partitions $\pi$ of $[0,t]$ whose mesh-size $|\pi|\equiv\max\{v-u,\,: [u,v]\in \pi\}$ goes to $0$.
\end{thm}

The proof of Theorem \ref{thm:sewing} in the context of semigroups can be found in \cite{gerasimovics2018} 
or \cite{gubinelli2010}. The proof in the case of propagators is carried out \emph{mutatis mutandis}, using 
for instance the smoothing property \eqref{smoothing_S}.
For the sake of completeness, we provide an alternative proof based on the multiplicative sewing lemma. 

\begin{proof}
Define the monoid $\MM=\MM(\alpha-2\gamma).$
We first show that $\MM$ can be endowed with a submultiplicative monoid structure, in the sense of Definition \ref{def:topologies}.
Given $\mu_j \equiv(S_j,x_j)\in \MM,j=1,2,$ one defines a distance (which turns out to be also a norm):
\begin{equation*}
d(\mu _1,\mu _2)
:=|S_1-S_2|_{\LL(\BB_{\alpha-2\gamma})} + |x_1-x_2|_{\alpha-2\gamma}\,,
\end{equation*}
while we let 
\[
|\mu |:= \max\big(1,d(\mu, 0)\big)\,.
\]
Using the right-distributivity of $\circ$ one can show $d(\mu_1\circ\nu, \mu_2\circ \nu) \leq d(\mu_1,\mu_2) 
|\nu|$ and the inequality $d(\nu\circ\mu_1, \nu\circ \mu_2) \leq  |\nu| d(\mu_1,\mu_2)$ can also be shown 
easily. Then, since $|\cdot|$ is continuous with respect to $d$,
the completeness assumption from Definition \ref{def:topologies} is satisfied and one concludes that 
$(\MM,|\cdot |,d)$ is a submultiplicative monoid.

Next, define $\mu :\Delta_2 \to \MM$ as
\[
\mu _{t,s}:=(S_{t,s}, S_{t,s}\xi _{t,s})\,.
\]
For every $(t,u,s)\in\Delta_3,$
observe that 
\begin{equation*}
\mu _{t,s}-\mu _{t,u }\circ\mu _{u ,s}
\equiv
(0,S_{t,s}\delta \xi _{t,u ,s}+ S_{t,u }(S_{u ,s}-\id)\xi _{t,u })\,.
\end{equation*}
Hence, using \eqref{P:smoothing}:
\begin{align*}
d(\mu _{t,s},\mu _{t,u }\circ\mu _{u ,s})
&\lesssim_{S} |\delta \xi_{t,u,s} |_{\alpha -2\gamma } + |(S_{u ,s}-\id)\xi _{t,u }|_{\alpha -2\gamma }
\\
&\leq 
\|\xi \|_{\ZZ^{\gamma}_\alpha}\big[|t-u |^{2\gamma }|u -s|^\gamma + |t-u|^\gamma |u -s|^{2\gamma} \big]
\,,
\end{align*}
showing in particular that $\mu $ is almost-multiplicative.\\

We now note that for every partition $\pi$ of $[s,t]$, 
$\mu^{\pi}_{t,s}=(S_{t,s},\I^\pi_{t,s}(\xi ))$, where $\I^\pi_{t,s}(\xi):=\sum_{(u,v)\in\pi}S_{t,u}\xi_{v,u}$ is the partial sum 
associated with $\pi$. Along the same lines as the proof of \eqref{maximal}, one can show that 
$|\mu^\pi_{t,s}| \lesssim 1+|t-s|^\gamma + |t-s|^{3\gamma}$ uniformly over every partition $\pi$ of $[s,t]$, 
thus showing $\mu \in \mathrm{BG}_2(0,T;\MM)$. We can then either use $\mu \in \mathrm{BG}_2(0,T;\MM)$ and apply Theorem 
\ref{thm:MSL} or use the fact that $|\cdot|$ is Lipschitz with respect to $d$ and apply \cite[Theorem 
10]{feyel2008non} to obtain existence of the unique multiplicative $\varphi_{t,s} = (S_{t,s}, I_{t,s})$ such that 
$|I_{t,s} - S_{t,s}\xi_{t,s}|_{\alpha-2\gamma} = d(\varphi_{t,s}, \mu_{t,s}) \lesssim |t-s|^{3\gamma}$. 
Letting $\I_t(\xi ):=I_{0,t},$ it is seen, thanks to multiplicativity of $\varphi $ that $\ddh\I_{t,s}(\xi )=I_{t,s},$ so that \eqref{e:sewing1_alt} holds with $\beta = 0$.

We now go over the proof of \eqref{e:sewing1_alt} for general $\beta \in (0,3\gamma)$ which also implies the 
continuity of $\I$ as a map $\ZZ^\gamma_\alpha \to \E^{0,\gamma}_\alpha$. To show this we take the dyadic 
partitions, namely $\pi_k:=\{s = t_0 < t_1 < \dots < t_{2^k} = t \}$ where $t_i = s + 2^{-k}i(t-s)$.
Denoting by $m=(u+v)/2,$ and decomposing $\xi $ as in \eqref{nota:ZZ2}, we have
\[\begin{aligned}
\I^{\pi_k}_{t,s}-\I^{\pi_{k+1}}_{t,s}
&=\sum_{[u,v]\in \pi_k}S _{t,u}\delta \xi _{v,m,u}+ S_{t,m}(S _{m,u}-\id)\xi _{v,m}
\\
&=\sum_{[u,v]\in \pi_k}S _{t,u}h^1 _{v,m,u}+\sum_{[u,v]\in \pi_k}S _{t,u}h^2 _{v,m,u}
\\
&\;\,
+\sum_{[u,v]\in \pi_k}S_{t,m}(S _{m,u}-\id)\xi^1 _{v,m}\;+\sum_{[u,v]\in \pi_k}S_{t,m}(S _{m,u}-\id)\xi^2 _{v,m}
\\
&=: \mathrm{I} +\mathrm{II} +\mathrm{III} +\mathrm{IV}\,.
\end{aligned}\]
Denote $\beta' = \alpha -2\gamma +\beta$ for shorthand. From the definition of the space $\ZZ^{\gamma}_\alpha,$ and from \eqref{P:smoothing} we can bound the first two terms as follows:
\[
\begin{aligned}
|\mathrm{I}+\mathrm{II}|_{\beta'}
&\leq  \|\xi \|_{\ZZ^{\gamma}_\alpha}\sum_{[u,v]\in \pi_k}|t-m|^{-\beta}\big[\,|v-m|^{2\gamma }|m-u|^{\gamma } + |v-m|^{\gamma }|m-u|^{2\gamma }\big]\;.
\end{aligned}
\]
For the third term, we have 
\[\begin{aligned}
|\mathrm{III}|_{\beta'}
&\lesssim_S\sum\nolimits_{[u,v]\in \pi_k}|t-m|^{-\beta}|(S_{m,u}-\id)\xi^1_{m,v}|_{\alpha-2\gamma}
\\
&\leq [\xi ^1]_{\gamma ,\alpha }\sum_{[u,v]\in \pi_k}|t-m|^{-\beta}|v-m|^{\gamma }|m-u|^{2\gamma},
\end{aligned}
\]
and similarly for the fourth term:
\[
|\mathrm{IV}|_{\beta'}
\leq [\xi ^2]_{2\gamma ,\alpha -\gamma }\sum_{[u,v]\in \pi_k}|t-m|^{-\beta}|v-m|^{2\gamma }|m-u|^\gamma\,.
\]

Now choose $\delta \geq 0$ such that $3\gamma - 1 > \delta > \beta - 1$. Summing all contributions, and 
observing that $v-m=m-u=\frac{|t-s|}{2^{k+1}}\leq t-m$ we have 
\[\begin{aligned}
|\I^{\pi_k}_{t,s}-\I^{\pi_{k+1}}_{t,s}|_{\beta' }
&\lesssim
\|\xi \|_{\ZZ^{\gamma}_\alpha} \sum_{[u,v]\in \pi_k}|t-m|^{-\beta}|v-m|^{3\gamma -1}|m-u|
\\
&\lesssim
\|\xi \|_{\ZZ^{\gamma}_\alpha} \sum_{[u,v]\in \pi_k}|t-m|^{\delta-\beta}|v-m|^{3\gamma-1-\delta} |m-u|
\\
&\lesssim
\|\xi \|_{\ZZ^{\gamma}_\alpha}2^{-k(3\gamma-1-\delta)}\left|t-s\right|^{3\gamma-1-\delta}\sum_{[u,v]\in \pi_k}|t-m|^{\delta-\beta} |m-u|
\\
&\lesssim  \|\xi \|_{\ZZ^{\gamma}_\alpha}2^{-k(3\gamma-1-\delta)}\left|t-s\right|^{3\gamma-1-\delta}\int_s^t|t-r|^{\delta - \beta}dr
\\
&\lesssim \|\xi \|_{\ZZ^{\gamma}_\alpha}2^{-k(3\gamma-1-\delta)}\left|t-s\right|^{3\gamma-\beta}\,,
\end{aligned}
\]
where we used the convexity of the integrand in the above Riemann sum since $\delta - \beta > -1$. Since 
$\delta$ above  is chosen so that $3\gamma-1-\delta > 0$ we can sum over $k \in \N_0$ to finally obtain 
\eqref{e:sewing1_alt} and finish the proof.
\end{proof}

\begin{rem}
	The monoid defined above can be endowed with a more complex submultiplicative structure which involves two indices of spatial regularity
$\beta'\leq\beta .$ 
Let
\begin{equation*}
\MM:=\MM(\beta',\beta):= \big(\LL(\BB_{\beta' })\cap \LL(\BB_\beta )\big)\ltimes \BB_{\beta' },
\end{equation*}
and given $\mu_j \equiv(S_j,x_j)\in \MM,j=1,2,$ introduce the distance $d(\mu _1,\mu _2)
:=|S_1-S_2|_{\LL(\BB_\beta ,\BB_{\beta'})} + |x_1-x_2|_{\beta'}$.
Defining further
$|\mu |:= \max(1,|S|_{\LL(\BB_\beta )\cap\LL(\BB_{\beta'})} + |x|_{\beta}),$
it is easily checked, using the same compactness argument as that of the proof of Lemma \ref{topology3},
that $(\MM,|\cdot |,d)$ is a submultiplicative monoid.
Making the choice $(\beta',\beta):=(\alpha-2\gamma, \alpha+\gamma-\kappa )$, $\kappa >0$ being arbitrary, 
one can obtain an alternative proof of Theorem \ref{thm:sewing} (the main difference is that the condition 
\eqref{partition_growth} is shown to hold on dyadic partitions only, in which case a version of Theorem 
\ref{thm:MSL} still holds).

The advantage of introducing $\MM(\beta',\beta )$ in comparison with the above (simpler) proof is that it allows obtaining approximation results for rough convolutions. More precisely, given an almost multiplicative 
approximation $\tilde S_{t,s}$ of $S_{s,t}$ (think for instance of a semigroup $e^{(t-s)L_s}$ or of the resolvent approximation of such semi-group), the multiplicative sewing Lemma then tells us that
$\sum_{[u,v]\in\pi} \tilde S_{t,u} \xi_{v,u}$ converges towards the same limit as \eqref{e:approximation} as $|\pi|\to 0$ by considering $\tilde \mu_{t,s} = (\tilde S_{t,s}, \tilde S_{t,s}\xi_{t,s})$.
This fact could certainly be useful in the quasilinear case, or for numerical analysis purposes. 
On the other hand, potential applications of this observation go beyond the scope of this paper, and hence 
we chose to avoid such a level of generality.
\end{rem}

	\subsection{Controlled rough paths}

In this paragraph, we introduce a notion of controlled paths that live in some $\BB_\alpha$.
Our definition differs from the ones given in \cite{gubinelli2010,gerasimovics2018} since it is independent of 
the propagator $S$ and its characterization does not involve the reduced increments $\ddh .$
Prior to introduce the notion of controlled paths in addition to the $\E^{0,\gamma}_\alpha$ we define a space 
$\E^{\gamma,2\gamma}_\alpha = \CC_2^\gamma(\BB_{\alpha-\gamma}) \cap 
\CC_2^{2\gamma}(\BB_{\alpha-2\gamma})$. Both $\E^{0,\gamma}_\alpha$ and $\E^{\gamma,2\gamma}_\alpha$ 
reflect the parabolic nature of~\eqref{evolution_equation} and show an interplay between the time and spatial 
regularity. 

	\begin{defn}[Controlled path according to a monotone family]
		\label{def:controlled}
		Let  $\BB := (\BB_{\beta})_{\beta \in\R}$ be a monotone family of interpolation spaces. Assume that
		$\X \equiv(X,\XX)\in \cC^{\gamma}(0,T;\R^d)$ for some $\gamma>1/3$, and let $\alpha \in \R$.
		We say that a pair $(y,y')$ is \emph{controlled by $\X$} according to $\BB_\alpha$ if the following holds:
		\begin{enumerate}[label=(\roman*)]
			\item\label{D1}
			We have $(y,y')\in \CC(\BB_\alpha) \times (\E^{0,\gamma}_{\alpha-\gamma})^d$ ;
			\item\label{D2}
			The remainder $R^y$ defined as:
			\begin{equation*}
				R^y_{t,s} :=\delta y_{s,t}-y'_s\cdot \delta X_{s,t}\equiv  \delta y_{t,s} - \sum\nolimits_{i=1}^dy_s^{\prime,i}\delta X^i_{t,s}\;,
			\end{equation*}
		belongs to $\E^{\gamma,2\gamma}_\alpha$.
		\end{enumerate}
		We will denote the space of all such controlled rough paths by $\DD^{2\gamma}_{X}([0,T], \BB_\alpha)$ or simply  $\DD^{2\gamma}_{X,\alpha }([0,T])$. When $T>0$ is fixed we will simplify further and make an abuse of notation by writing simply $\DD^{2\gamma}_{X,\alpha }$. Sometimes we will refer to $y'$ as the Gubinelli derivative.

We endow the space $\DD_{X,\alpha }^{2\gamma}$ with the norm:
		\[
		\|y,y'\|_{\DD^{2\gamma}_{X,\alpha}} = |y|_{0, \alpha}+ |y'|_{\E^{0,\gamma}_{\alpha-\gamma}}+ |R^y|_{\E^{\gamma,2\gamma}_\alpha}\,.
		\]
With this definition, it is easy to check that $\DD^{2\gamma}_{X,\alpha}$ is a Banach space (we leave the details to the reader).
	\end{defn}

	One can actually see from the above definition that $y \in \E^{0,\gamma}_\alpha$ and 
	\begin{equation}\label{y:regularity}
		[\dd y]_{\gamma, \alpha-\gamma} \leq |y'|_{0, \alpha-\gamma} [\dd X]_\gamma+ [R^y]_{\gamma, \alpha-\gamma}\;.
	\end{equation}
	For that reason, we do not make H\"older regularity of $y$ as a part of the definition of a controlled rough path. 
	
	Note that one can recover the usual definition of the controlled rough path (see \cite{gubinelli2004,friz}) 
	from the above definition if one takes $\BB_\beta = \BB_0$ for all $\beta \in \R$. In the notation of 
	\cite[Definition 4.6]{friz} we then have $\mathscr D^{2\gamma}_X(0,T;\BB_0) = 
	\DD^{2\gamma}_{X}([0,T],\BB_0)$. Therefore, all our later analysis applies to the finite-dimensional case by 
	choosing such constant family $\BB_0$ and the propagator to be the identity: $S_{t,s} = \id$.
	
		\begin{rem}
		The definition of controlled rough paths can be reformulated using the `reduced increment' $\ddh$ 
		instead of $\dd=\dd^{\id}$, which might look more natural when talking about the integrals 
		$\int_0^tS_{t,s}y_s\cdot d\X_s$.
		Following \cite{gubinelli2010, deya2012rough,gerasimovics2018}, it is indeed possible to introduce the space
		$\DD^{2\gamma,}_{X,S,\alpha }$ consisting of pairs $(y,y') \in \CC(0,T;\BB_\alpha)\times \CC(0,T;\BB_{\alpha-\gamma})$, 
		such that
		$\tilde{R}^y_{t,s} := \ddh y_{t,s} - S_{t,s}y'_s \dd X_{t,s}$ belongs to $\CC^{2\gamma}_2(0,T;\BB_{\alpha-2\gamma}) \cap \CC^{\gamma}_2(0,T;\BB_{\alpha-\gamma})$
		while 
		$\ddh y'$  belongs to $\CC_2^\gamma(0,T;\BB_{\alpha-2\gamma})\,$. 
		
		However, this has the inconvenience that the spaces considered depend on the propagator $S$ while, as seen in the present paper, it is possible to get rid of this dependency (the spaces $\DD_{X,\alpha }^{2\gamma }$ do however, depend on the scale $(\BB_\beta )_{\beta \in\R}$).
		This would in addition make some proofs (like the proof of Lemma~\ref{lem:composition}) more tedious. Finally, the 
		space $\DD^{2\gamma}_{X,\alpha}$ is much closer to the usual notion of controlled rough path; for 
		instance its elements are controlled rough paths in the sense of \cite[Definition 4.6]{friz} at the level of 
		$\BB_{\alpha-2\gamma}$. 
	\end{rem}

	In the sequel we will extensively use the following interpolation inequality, which is an immediate 
	consequence of \eqref{interpolation}; for every $\beta \in [\gamma,2\gamma],$ we have:
	\begin{equation}\label{interpolation2}
		|R^y_{t,s}|_{\alpha-\beta} \leq |R^y|_{\E^{\gamma,2\gamma}_\alpha} |t-s|^\beta\,.
	\end{equation}

	We now state the fundamental result for the above controlled rough paths, namely that for such paths 
	the `rough convolution' with respect to $\X$ is well-defined.
	\begin{thm}[Integration] \label{integration}
		Let $\X = (X,\XX) \in \cC^{\gamma}(0,T;\R^d)$ for some $\gamma>1/3$ and let $\alpha \in \R$. Let $(y^i,y^{i,\prime}) \in \DD^{2\gamma}_{X,\alpha }$ for  $i=1,\dots ,d$. Then the integral 
		\begin{equation} \label{e:integration1}
				\int_0^tS_{t,s}y_s\cdot d\X_s  
				:= \lim_{|\pi|\to 0} \sum_{[u,v] \in \pi}S_{t,u}(y_{u}\cdot \delta X_{v,u} + y'_{u}:\XX_{v,u})\,,
		\end{equation}
		exists as an element of $\BB_\alpha,$ where we denote $y'_u:\XX_{v,u}:=\sum_{1\leq i,j\leq d}y^{\prime,ij}\XX_{v,u}^{ij}.$
		
		Moreover, for every $0 \leq \beta < 3\gamma$ the above integral satisfies the estimate 
		\begin{equation} \label{e:integration2}
		\Big| \int_s^tS_{t,u}y_u\cdot d\X_u - S_{t,s}(y_s \cdot \delta X_{t,s}+ y'_s:\XX_{t,s})\Big|_{\alpha-2\gamma+\beta} \lesssim \varrho_\gamma(\X) \|y,y'\|_{\DD^{2\gamma}_{X,\alpha}} |t-s|^{3\gamma-\beta}\,,
		\end{equation} 
		for all $(s,t) \in \Delta_2$.
	\end{thm}
	\begin{proof}
	It suffices to apply the affine sewing lemma (Theorem \ref{thm:sewing}) to 
	\[
	\xi_{t,s} = y_s\cdot \delta X_{t,s} + y'_s:\XX_{t,s},\quad (t,s)\in\Delta _2\,,
	\]
	for which we need to show that $\xi \in \ZZ^{\gamma}_\alpha$. 
	Indeed, the existence of the integral \eqref{e:integration1} will follow immediately by \eqref{e:approximation} while \eqref{e:integration2} is a consequence of \eqref{e:sewing1_alt}.
	
	First, note that $\xi $ is indeed an element of $\CC_2^{\gamma ,\alpha } + \CC_{2}^{2\gamma ,\alpha -\gamma }$ and that moreover
	\[
	\|\xi \|_{\CC_2^{\gamma ,\alpha } + \CC_{2}^{2\gamma ,\alpha -\gamma }}\leq |y\cdot \delta X|_{\gamma ,\alpha } + |y':\XX|_{2\gamma ,\alpha -\gamma }
	\leq \varrho_\gamma (\X) \|y,y'\|_{\DD_{X,\alpha }^{2\gamma }}\,,
	\]
	by definition of $\DD_{X,\alpha }^{2\gamma }.$
	Next, thanks to Chen's relation we have the algebraic identity
	\begin{align*}
		\dd \xi _{t,u ,s} = \delta X_{t,u }\cdot R^y_{u ,s} + \XX_{t,u }:\delta y'_{u ,s}\,,
	\end{align*}
from which we infer
\begin{equation*}
|\delta \xi_{t,u, s}|_{\alpha -2\gamma }\leq [X]_\gamma |t-u |^\gamma |u -s|^{2\gamma }[R^{y}]_{2\gamma ,\alpha -2\gamma } + [\XX]_{2\gamma }|t-u |^{2\gamma }|u -s|^\gamma [\delta y']_{\gamma,\alpha -2\gamma }\,.
\end{equation*} 
Hence, it follows that $\delta \xi \in \CC_2^{2\gamma ,\gamma }(\BB_{\alpha -2\gamma })+ \CC_2^{\gamma ,2\gamma }(\BB_{\alpha -2\gamma })$ with
\begin{equation*}
\|\delta \xi\|_{ \CC_2^{2\gamma ,\gamma }(\BB_{\alpha -2\gamma }) +  \CC_2^{\gamma ,2\gamma }(\BB_{\alpha -2\gamma })}
\leq \varrho_\gamma (\X)\|y,y'\|_{\DD_{X,\alpha }^{2\gamma}}\,.
\end{equation*} 
Summing the above contributions, we find that $\xi \in \ZZ^{\gamma}_\alpha$,
and the conclusion of Theorem \ref{thm:sewing} yields the claimed estimate.
	\end{proof}

The following result not only describes the stability of integration but also tells us that the `rough 
convolution' improves the spatial regularity of the controlled rough path.

	\begin{cor}
	\label{cor:continuous}
	The integration map defined in Theorem \ref{integration} is continuous from $\mathcal D_{X,\alpha }^{2\gamma }$ into itself.
	In addition, for $T \leq 1$ and for every $\sigma ,\gamma'$ such that $0<\sigma <\gamma '\leq \gamma ,$
	the linear map
		\[
\DD^{2\gamma'}_{X,\alpha}([0,T]) \to \DD^{2\gamma'}_{X,\alpha+\sigma}([0,T]),\enskip \enskip 
		(y,y') \mapsto (z,z') := \Big(\int_0^\cdot S_{\cdot,u}y_u\cdot d\X_u,\, y\Big),
		\]
		is well-defined, bounded, and it satisfies the following estimate: 
		\begin{equation*}
			\| z,z'\|_{\DD^{2\gamma'}_{X,\alpha+\sigma}} \leq |y_0|_{\alpha } + C_{\gamma,\sigma} T^{\varepsilon} \big(1+\varrho_\gamma(\X)\big)\|y,y'\|_{\DD^{2\gamma'}_{X,\alpha}}\; ,
		\end{equation*}
		where $\varepsilon:=\min\{\gamma -\gamma',\gamma' -\sigma\}.$
	\end{cor}

	\begin{proof}
The first step is to show that $z$ is indeed controlled by $X$ if one lets $z'=y$. 
For this we need to evaluate the remainder $R^z_{t,s}:=\delta z_{t,s}-y_s\cdot \delta X_{t,s}$ and show that it 
has the correct regularity.
	Denote by
		\begin{equation*}
		\mathscr{R}_{t,s} =\int_s^tS_{t,u}y_u\cdot d\X_u - S_{t,s}(y_s \cdot \delta X_{t,s}+ y'_s:\XX_{t,s})\,,\quad (t,s)\in\Delta_2 ,
		\end{equation*}
		where the first integral is understood in the sense of Theorem \ref{integration}.
		Using the fact that $\varrho_{\gamma'} (\X) \leq \varrho_{\gamma} (\X)$ for $T \leq 1$ and using the estimate \eqref{e:integration2} with $\beta = \sigma +(2-i)\gamma' $ for $i\in\{1,2\},$ we get
		\begin{align*}
		|\mathscr R_{t,s}|_{\alpha +\sigma -i\gamma'}
		= |\mathscr R_{t,s}|_{\alpha-2\gamma'  +\sigma +(2-i)\gamma' }
		&\lesssim
		\varrho_{\gamma'} (\X)\|y,y'\|_{\mathcal D_{X,\alpha }^{2\gamma' }}|t-s|^{\gamma' - \sigma + i\gamma' }
		\\
		&\leq \varrho_{\gamma} (\X)\, \|y,y'\|_{\mathcal D_{X,\alpha }^{2\gamma' }} |t-s|^{i\gamma' } T^{\gamma'-\sigma } 
		\,,
		\end{align*}
uniformly over $(t,s)\in\Delta _2.$
		Next, observe that
		\[
		\begin{aligned}
		R_{t,s}^z
		&\equiv \int_s^t S_{t,u}y_u\cdot d\X_u - y_s\cdot \delta X_{s,t} +(S_{t,s}-\id)\int_0^s S_{s,u}y_u \cdot d\X_u 
		\\
		&=  \left(\int_s^t S_{t,u}y_u\cdot d\X_u - S_{t,s}(y_s \cdot \delta X_{t,s} + y'_s:\XX_{t,s})\right) 
		\\
		&\quad \quad 
		+ (S_{t,s}-\id)y_s\cdot \delta X_{t,s} +(S_{t,s}-\id)\int_0^sS_{s,u}y_u\cdot d\X_u + S_{t,s}y'_s:\XX_{t,s}\,,
		\\
		&=: \mathscr R_{t,s} + \mathrm{I}_{t,s} + \mathrm{II}_{t,s} + \mathrm{III}_{t,s}\,.
		\end{aligned}
		\]
Using the smoothing property \eqref{smoothing_S} for $S$, we see that for $i=1,2$:
\[
|\mathrm{I}_{t,s}|_{\alpha +\sigma -i\gamma '}
\leq [X]_\gamma |t-s|^{\gamma }|S_{t,s}-\id|_{\mathcal L(\BB_{\alpha },\BB_{\alpha -(i\gamma '-\sigma )})}|y_s|_{\alpha }
\lesssim |t-s|^{i\gamma '}T^{\gamma -\sigma }|y|_{0,\alpha }\,.
\]
For the second term, we have, thanks to Theorem \ref{integration};
\[\begin{aligned}
|\mathrm{II}_{t,s}|_{\alpha +\sigma -i\gamma '} 
&\lesssim |S_{t,s}-\id|_{\mathcal L(\BB_{\alpha+\sigma} ,\BB_{\alpha+\sigma -i\gamma '})}|S_{s,0}y_0|_{\alpha+\sigma}[X]_\gamma s^{\gamma } 
\\
&\quad \quad 
+|S_{t,s}-\id|_{\mathcal L(\BB_{\alpha +\sigma} ,\BB_{\alpha+\sigma -i\gamma '})}|S_{s,0}y'_0|_{\alpha +\sigma}[\XX]_{2\gamma} s^{2\gamma } 
\\
&\quad \quad \quad 
+|S_{t,s}-\id|_{\LL(\BB_{\alpha+\sigma},\BB_{\alpha+\sigma -i\gamma '})}|\mathscr R_{s,0}|_{\alpha+\sigma}
\\
&\lesssim \varrho_\gamma(\X)
|t-s|^{i\gamma '} \Big\{s^{\gamma-\sigma} |y|_{0,\alpha }
+ |y'|_{0,\alpha -\gamma'} s^{2\gamma -\gamma ' - \sigma}
+ s^{\gamma'-\sigma}\|y,y'\|_{\mathcal D_{X,\alpha }^{2\gamma '}}
\Big\}\\
&\lesssim \varrho_\gamma(\X) \|y,y'\|_{\mathcal D_{X,\alpha }^{2\gamma '}}
|t-s|^{i\gamma '} T^{\gamma'-\sigma}\;.
\end{aligned}
\]
Similarly, we have
\begin{align*}
|\mathrm{III}_{t,s}|_{\alpha +\sigma -\gamma '}\;
&\leq [\XX]_{2\gamma }|t-s|^{2\gamma }|S_{t,s}|_{\mathcal L(\BB_{\alpha -\gamma '},\BB_{\alpha -\gamma ' + \sigma})}|y'_s|_{\alpha -\gamma '}
\\
&\lesssim \varrho_\gamma (\X)|t-s|^{\gamma'}|y'|_{0,\alpha -\gamma '} T^{2\gamma-\gamma'-\sigma}
\\
|\mathrm{III}_{t,s}|_{\alpha + \sigma - 2\gamma' }
&\leq [\XX]_{2\gamma }|t-s|^{2\gamma }|S_{t,s}y'_s|_{\alpha - \gamma '}
\\
&\lesssim \varrho_\gamma (\X)|t-s|^{2\gamma'}|y'|_{0,\alpha -\gamma '} T^{2\gamma-2\gamma'}\,.
\end{align*}
Combining the above estimates, we obtain:
\begin{equation*}
		\max_{i=1,2}[R^z]_{i\gamma ',\alpha+\sigma-i\gamma'} 
		\lesssim \varrho_\gamma (\X) T^{\varepsilon} \|y,y'\|_{\mathcal D_{X,\alpha }^{2\gamma '}}\,.
\end{equation*}
In particular, we see that $z$ is controlled by $X$ according to $\BB_\alpha$ and that $z'=y$.

Next, to estimate the H\"older norm of the Gubinelli derivative, we first observe that
\[
[R^y]_{\gamma ',\alpha +\sigma -2\gamma '}\leq C |R^y|_{\E^{\gamma',2\gamma'}_\alpha} T^{\gamma '-\sigma }\,,
\]
as can be easily seen by the interpolation inequality \eqref{interpolation2}.
Then, using that $\gamma '>\sigma $ we have
\[\begin{aligned}
|\delta z'_{t,s}|_{\alpha+\sigma-2\gamma'} 
&= |\delta y_{t,s}|_{\alpha+\sigma-2\gamma'} \leq |y'_s|_{\alpha+\sigma-2\gamma'}|\delta X_{t,s}|+|R^y_{t,s}|_{\alpha+\sigma-2\gamma'}\,.
\\
&\lesssim [X]_{\gamma }|y'|_{0,\alpha -\gamma '} |t-s|^{\gamma '} T^{\gamma -\gamma '} +|t-s|^{\gamma '}|R^y|_{\E^{\gamma',2\gamma'}_\alpha}T^{\gamma '-\sigma }\,,
\end{aligned}
\]
which, by writing $z'_t=z'_0 + \delta z'_{t,0}$, yields the estimate
\begin{align*}
|z'|_{\E^{0,\gamma '}_{\alpha+\sigma -\gamma'}}\equiv 
|z'|_{0,\alpha +\sigma -\gamma '}+
[\delta z']_{\gamma ',\alpha+\sigma-2\gamma'} 
\lesssim 
|y_0|_{\alpha } + \varrho_\gamma (\X)T^{\varepsilon}\|y,y'\|_{\mathcal D_{X,\alpha }^{2\gamma '}}\,.
\end{align*}

\noindent To conclude that $(z,z')$ is controlled by $X$ according to $\BB_{\alpha+\sigma}$ it remains to estimate $|z|_{0,\alpha+\sigma}$, for which we use:
\begin{align*}
	z_t = \mathscr R_{t,0} + S_{t,0} y_0 \cdot \delta X_{t,0} + S_{t,0} y'_0 : \XX_{t,0}\,.
\end{align*}
Therefore, using~\eqref{e:integration2}, smoothing properties of the propagator and $\varrho_{\gamma'} (\X) \leq \varrho_{\gamma}(\X)$ we get
\begin{align*}
	|z_t|_{\alpha+\sigma} \lesssim \varrho_{\gamma}(\X) \|y,y'\|_{\mathcal D_{X,\alpha }^{2\gamma '}} t^{3\gamma' - 2\gamma' - \sigma} + |y|_{0,\alpha} [X]_\gamma t^{\gamma - \sigma} + |y'|_{0,\alpha-\gamma'} [\XX]_{2\gamma} t^{2\gamma -\gamma' - \sigma},
\end{align*}
which implies that $|z|_{0,\alpha+\sigma} \lesssim \varrho_{\gamma}(\X) \|y,y'\|_{\mathcal D_{X,\alpha }^{2\gamma '}} T^\varepsilon$, thus finishing the proof.
	\end{proof}

We are now going to see that a controlled path composed with some sufficiently regular function is again a 
controlled path. 
	\begin{lem} \label{lem:composition}
		Let $\gamma \in (1/3,1/2]$, $\alpha \in \R$ and fix $\sigma \geq 0$. Let $F\in \CC^2_{\alpha -2\gamma ,-\sigma }(\BB)$ be some non-linearity with bounded derivatives up to second order.
		For $(y,y') \in \DD^{2\gamma}_{X,\alpha }$, define
		\begin{align*}
			(z_t,z'_t) := (F(y_t),DF(y_t)\circ y'_t)\,,\quad 
		\text{for every}\enskip t\in [0,T]\,.
		\end{align*}
		Then the following assertions are true.
		\begin{enumerate}[label=(\roman*)]
		 \item \label{comp:i}
		One has $(z,z') \in \DD^{2\gamma}_{X,\alpha-\sigma}$ and moreover:
		\begin{equation} \label{e:compos1}
			\|z,z'\|_{\DD^{2\gamma}_{X,\alpha-\sigma}} \lesssim \|F\|_{\CC^2}\big(1+\varrho_\gamma(\X)\big)^2\|y,y'\|_{\DD^{2\gamma}_{X,\alpha}}(1+\|y,y'\|_{\DD^{2\gamma }_{X,\alpha}})\,.
		\end{equation}
		
		\item \label{comp:ii}
		If we assume further that $F\in \CC^3_{\alpha -2\gamma ,-\sigma }(\BB)$ with bounded third derivative, and if $(\tilde z,\tilde z'):=(F(\tilde y),DF(\tilde y)\circ \tilde y')$ for another such pair $(\tilde y,\tilde y')\in\mathcal D_{X,\alpha }^{2\gamma '},$ then the following estimate holds
		\begin{align*}
		\|z-\tilde z,z'-\tilde z'\|_{\DD^{2\gamma}_{X,\alpha-\sigma}} 
			 &\lesssim
			 \|F\|_{\CC^3}\big(1+\varrho_\gamma(\X)\big)^2\|y-\tilde y,y'-\tilde y'\|_{\DD^{2\gamma }_{X,\alpha}}
			 \\
			&\quad \quad \quad \quad \quad 
			\times(1+\|y,y'\|_{\DD^{2\gamma}_{X,\alpha}}+\|\tilde y,\tilde y'\|_{\DD^{2\gamma }_{X,\alpha}})^2\,.
		\end{align*}
		\end{enumerate}
	Where in both~\ref{comp:i} and~\ref{comp:ii} we set $\|F\|_{\CC^k} = \max \{\|F\|_{\CC^k(\BB_{\alpha-\gamma}, \BB_{\alpha-\gamma-\sigma})}, \|F\|_{\CC^k(\BB_{\alpha-2\gamma}, \BB_{\alpha-2\gamma-\sigma})}\}$.
	\end{lem} 

	\begin{proof}
First, observe that because of the continuity of $F$ and the inclusion $\BB_\alpha\subset 
\BB_{\alpha-\gamma}$,  we have $(z,z') \in \CC(0,T;\BB_{\alpha-\sigma}) \times 
\CC(0,T;\BB^{d}_{\alpha-\sigma-\gamma})$. We can view $D^kF(y_t)$ as an element of 
$\LL(\BB_{\alpha-i\gamma}^{\otimes k}, \BB_{\alpha-i\gamma -\sigma})$ for $k=1,2,3$ and $i=1,2$. 
With this at hand, we write $\delta z'_{t,s}= DF(y_s)\circ \delta y'_{t,s} +(DF(y_t)-DF(y_s))\circ y'_t,$ and since 
$|\cdot|_{\alpha-2\gamma} \leq |\cdot|_{\alpha-\gamma}$ we obtain:
\begin{align*}
	[\delta z']_{\gamma ,\alpha-\sigma-2\gamma}
	&\lesssim |DF|_{\mathcal L(\BB_{\alpha -2\gamma },\BB_{\alpha -2\gamma -\sigma })}[\delta y']_{\gamma,\alpha-2\gamma}
	\\
	&\qquad
	+ |D^2F|_{\LL(\BB_{\alpha-2\gamma}^{\otimes 2},\BB_{\alpha-2\gamma -\sigma})} [\delta y]_{\gamma, \alpha-2\gamma} |y'|_{0,\alpha -\gamma }
	\\
	&\lesssim \|F\|_{\CC^2} \big(1+\varrho_\gamma(\X)\big) \|y,y'\|_{\DD^{2\gamma}_{X,\alpha}} (1 + \|y,y'\|_{\DD^{2\gamma}_{X,\alpha}})\;,
\end{align*}
where we used~\eqref{y:regularity} to estimate $[\delta y]_{\gamma, \alpha-2\gamma}$.
Next, we estimate the remainder term 
\[
R^z_{t,s} := F(y_t) - F(y_s) - DF(y_s)\circ y'_s\cdot \dd X_{t,s}\;,
\]
and show that it belongs to $\CC_2^{2\gamma }(0,T;\BB_{\alpha-\sigma-2\gamma }) \cap \CC_2^{\gamma }(0,T;\BB_{\alpha-\sigma-\gamma })$.
We rewrite $R^z$ as
\begin{align*}
	R^z_{t,s} & = F(y_t)-F(y_s)-DF(y_s)\circ\delta y_{t,s} +DF(y_s)\circ R^y_{t,s} \\
		& = T_{t,s} +DF(y_s)\circ R^y_{t,s}  \;,
\end{align*}
where, by Taylor's formula
		\[
T_{t,s}:= \left(\int_{0}^1\int_{0}^1D^2F(y_s+\theta\theta ' \delta y_{t,s})d\theta ' \theta d\theta  \right)\circ(\delta y_{t,s} \otimes \delta y_{t,s})\,.
\]
 		By definition of the spaces $\CC_{\alpha - 2\gamma, -\sigma }^2$, we have for $i=1,2$:
		\begin{align*}
		|R^z|_{i \gamma,\alpha-\sigma-i\gamma} 
		&\leq 
		|D^2F|_{\mathcal L(\BB_{\alpha -i\gamma }^{\otimes 2},\BB_{\alpha -i\gamma -\sigma })}[\delta y]^2_{\gamma,\alpha-i\gamma}
		\\
		&\quad \quad \quad \quad \quad 
		+|DF|_{\mathcal L(\BB_{\alpha -i\gamma },\BB_{\alpha -i\gamma -\sigma })}[R^y]_{i \gamma,\alpha-i\gamma} 
		\\
		&\lesssim \|F\|_{\CC^2} \big(1+\varrho_\gamma(\X)\big)^2\|y,y'\|_{\DD^{2\gamma}_{X,\alpha}}(1+\|y,y'\|_{\DD^{2\gamma}_{X,\alpha}})\;,
		\end{align*}
which implies \ref{comp:i}.\\

For \ref{comp:ii}, we write (with obvious notations)
\[
R^{z}_{t,s}-R^{\tilde z}_{t,s}
=T_{t,s} -\tilde T_{t,s} + (DF(y_s)-DF(\tilde y_s))\circ R^y_{t,s} + DF(\tilde y_s)\circ (R^y_{t,s}-R^{\tilde y}_{t,s})\,.
\]
For $i=1,2$ we have
\[\begin{aligned}
|T_{t,s} -\tilde T_{t,s}|_{\alpha -\sigma -i\gamma }
&\leq |D^3F|_{\LL(\BB^{\otimes 3}_{\alpha -i\gamma },\BB_{\alpha -i\gamma -\sigma })}| y- \tilde y|_{0,\alpha - i \gamma } |\delta y_{t,s}|^2_{\alpha -i\gamma }
\\
& \quad + |D^2F|_{\LL(\BB^{\otimes 2}_{\alpha -i\gamma },\BB_{\alpha -i\gamma -\sigma })}|\delta y_{t,s} - \delta \tilde y_{t,s}|_{\alpha - i \gamma} |\delta \tilde y_{t,s}|_{\alpha - i \gamma} \\
&\lesssim \|F\|_{\CC^3}| y-\tilde y|_{\E^{0,\gamma}_\alpha }(1  + [ y]_{\gamma ,\alpha - \gamma } +  [\tilde y]_{\gamma ,\alpha - \gamma })^2
|t-s|^{2\gamma }\,.
\end{aligned}
\]
Similarly,
\[
|(DF(y_s)-DF(\tilde y_s))\circ R^y_{t,s}|_{\alpha -i\gamma }\leq |D^2F|_{\mathcal L(\BB_{\alpha -i\gamma }^{\otimes 2},\BB_{\alpha -i\gamma -\sigma })} |y-\tilde y|_{0,\alpha } |R^y|_{\E^{\gamma,2\gamma}_\alpha} |t-s|^{i\gamma }\;,
\]
while
\[
|DF(\tilde y_s)\circ (R^y_{t,s}-R^{\tilde y}_{t,s})|_{\alpha -i\gamma -\sigma }
\leq |DF|_{\LL(\BB_{\alpha -i\gamma },\BB_{\alpha -i\gamma -\sigma })}|R^y-R^{\tilde y}|_{\E^{\gamma,2\gamma}_\alpha}\;.
\]
Summing the above three estimates yields the correct bound for $|R^{z}-R^{\tilde z}|_{\E^{\gamma,2\gamma}_{\alpha-\sigma}}$.

\noindent For the Gubinelli derivatives, we have
\[\begin{aligned}
&\delta ( D F(y)y'- D F(\tilde y)\tilde y^\prime)_{t,s}
\\
&= 
\int _0^1\left(D^2 F(y_s+\theta \delta y_{t,s}) - D ^2F(\tilde y_s+\theta \delta \tilde y_{t,s})\right)d\theta 
\circ (\delta y_{t,s} \otimes y'_t)
\\
&\quad \quad 
+\int_0^1 D ^2F(\tilde y_s+\theta \delta \tilde y_{t,s})d\theta \circ \big((\dd y_{t,s} - \dd \tilde{y}_{t,s}) \otimes y'_t\big)
\\
&\quad \quad 
+\int_0^1 D ^2F(\tilde y_s+\theta \delta \tilde y_{t,s})d\theta \circ \big(\delta \tilde y_{t,s} \otimes (y'_t-\tilde y'_t)\big)
\\
&\quad \quad 
+ ( D F(y_s)- D F(\tilde y_s))\circ \delta y'_{t,s} 
+ D F(\tilde y_s)\circ(\delta y_{t,s}-\delta \tilde y'_{t,s})\,.
\end{aligned}
\]
This gives
\[
|\delta z'_{t,s}-\delta \tilde z'_{t,s}|_{\alpha -2\gamma }
\lesssim \|F\|_{\CC^3} \|y-\tilde y, y'-\tilde y'\|_{\DD_{X,\alpha }^{2\gamma }}\Big(1+\|y,y' \|_{\DD_{X,\alpha }^{2\gamma }} + \|\tilde{y},\tilde{y}' \|_{\DD_{X,\alpha }^{2\gamma }}\Big)^2\,.
\]
This finishes the proof of Lemma \ref{lem:composition}.
\end{proof}

\section{Equations with subcritical multiplicative noise: proof of Theorem \ref{thm:main}}
\label{sec:evolution}
	In this section we fix $\gamma \in(1/3,1/2]$, $\alpha \in \R$, a rough path $\X = (X, \XX) \in \cC^\gamma(0,T;\R^d)$, and we let $\sigma \in[0,\gamma)$. 
	We will address the proof of local existence and uniqueness for the rough PDE
	\begin{equation} \label{e:rpde2}
	du_t = L_tu_tdt +N(u_t)dt + \sum\nolimits_{i=1}^dF_i(u_t)d\X^i_t\quad \text{and}\quad  u_0 = x \in \BB_\alpha,
	\end{equation}
	under suitable conditions on the non-linearities.
	The proof of Theorem \ref{thm:main} is a simple consequence of Theorem \ref{RPDE} below. Further 
	properties of the solution map will be also given in Theorems \ref{stab_sol} and \ref{smoothing}.

	In the sequel, an equation of the form \eqref{e:rpde2} will be referred to as `subcritical' provided that
	the function $F$ sends $\BB_\alpha$ to $\BB_{\alpha-\sigma}$ with some $\sigma \in[0, \gamma).$
	Concrete examples of subcritical equations will be given in Section \ref{sec:examples}.

	\subsection{Solutions to subcritical RPDEs}

	For $i=1,\dots d,$ let $F_i \in \CC^2_{\alpha -2\gamma,-\sigma},$ and for each $(y,y') \in 
	\DD^{2\gamma}_{X,\alpha }$ let
	$$(z_t, z'_t) := \Big(\int_0^{t}S_{t,s}F(y_s)\cdot d\X_s, F(y_t )\Big)\,,\quad \text{for every}\enskip 
	t\in[0,T]\,.$$
	Then, Lemma \ref{lem:composition} together with Corollary \ref{cor:continuous} gives us that 
	$(z,z')$ is again an element of the controlled paths space $\DD^{2\gamma}_{X,\alpha }$.
	This is due to the fact that, though $F$ reduces the spatial regularity by $\sigma$, the lost regularity is 
	recovered from the smoothing properties of the integration map associated with $S$.
	This observation suggests that we might be successful in applying a Banach fixed point argument in order 
	to solve \eqref{e:rpde2} locally.

	\begin{thm}[Local solution of subcritical RPDEs] \label{RPDE}
	Fix $\alpha \in\R,$ $\gamma \in (1/3,1/2]$ and $\sigma \in[0,\gamma ).$
	Assume that we are given a non-linearity $F = (F_1,\dots,F_d)$ such that
	$F_i \in \CC^3_{\alpha -2\gamma,-\sigma}(\BB)$ for $i=1,\dots ,d$ and $N \in \Lip_{\alpha,-\delta}(\BB)$ for some $n\geq 1$ and $1 > \delta \geq 0$.

For every $x \in \BB_\alpha ,$
 there exists $0 < \tau \leq T$ and a unique $(u,u') \in \DD^{2\gamma}_{X,\alpha }([0,\tau))$ such that $u' = F(u)$ and
		\begin{equation} \label{e:rpde3}
		u_t = S_{t,0}x + \int_0^t S_{t,r}N(u_r)dr+ \int_0^t S_{t,r}F(u_r)\cdot d\X_r\, , \quad t<\tau. 
		\end{equation}
	\end{thm}
	\begin{proof}
Assume first that $T\leq 1$ and let us define the map
		\begin{equation}\label{soln_map}
			\MM_T(y, y')_t :=\Big(S_{t,0}x+\int_0^tS_{t,s}N(y_s)ds+\int_0^tS_{t,s}F(y_s)\cdot d\X_s, F(y_t)\Big)\;.
		\end{equation}
		Instead of solving the equation directly in the space $\DD^{2\gamma}_{X,\alpha }$ we will show that the map $\MM_T$ is invariant and contractive inside a ball of a larger space. We now fix a parameter $\gamma '\in(\sigma ,\gamma )$
	and let $\varepsilon = \min\{\gamma-\gamma', \gamma'-\sigma\}$. We further define two continuous paths $\xi :[0,T]\to \BB_\alpha $ and $\xi ':[0,T]\to \BB_{\alpha -\gamma '}$ as
	\begin{equation*}
	\xi _t:= S_{t,0}x + \int_0^tS_{t,r}F(x)\cdot d\X_r\,,\qquad 
	\xi '_t:= F(x),\qquad t\in[0,T]\,,
	\end{equation*}
	and observe that $(\xi ,\xi ')\in\mathcal D_{X,\alpha }^{2\gamma '}$. This is indeed a consequence of 
	Corollary~\ref{cor:continuous} applied to the constant path $(F(x),0) \in \DD_{X,\alpha-\sigma}^{2\gamma 
	'}$, and of the fact that $(S_{t,0}x, 0)$ belongs to $\DD_{X,\alpha }^{2\gamma '}$ (using the smoothing 
	properties of the propagator). 

	The first step is to show the existence of a positive $T_*(\varepsilon, \gamma, |x|_\alpha 
	,\varrho_\gamma(\X)) \leq 1$ such that for every $T\in[0,T_*]$ the map $\MM_T$ leaves the ball $B_T(x)$ 
	invariant, where
	\begin{align*}
		B_T(x) = \Big\{(y,y') \in \DD^{2\gamma '}_{X,\alpha }([0,T])\,:\, (y_0,y_0') = \big(x,F(x)\big)\;
			\text{and\ }\|y-\xi, y' - \xi'\|_{\DD^{2\gamma'}_{X,\alpha }}  \leq 1\Big\}\,.
	\end{align*}
		
	Recall that from the definition of $\CC^3_{\alpha -2\gamma,-\sigma}$ it follows that $F_i$ together with its 
	derivatives sends bounded sets of $\BB_\beta$ to $\BB_{\beta-\sigma}$ for $\beta \geq \alpha-2\gamma$. 
	Therefore, without loss of generality one can view $F_i$ and its derivatives to be bounded since we restrict 
	ourselves to the ball $B_T(x)$. Similarly, one can assume that $N$ is globally Lipschitz.
		
		\bigskip
		
		\item[\em \indent Step 1: Stability of $B_T(x)$.]
		Let $(y,y') \in B_T(x)$ and let $(z ,z ') = \MM_T(y,y')$.
		For simplicity denote the drift term $\mathscr N_t:=\int_0^t S_{t,r} N(y_r)dr$ and define:
		\[
		(\zeta ,\zeta '):=(z -\xi -\mathscr N , z'-\xi ')\,.
		\]
		Note that $\zeta_0 = \zeta '_0 = 0$ and $
		\zeta _t= \int_0^{t}S_{t,r} (F(y_r)-F(x))\cdot d\X_r$.
		Furthermore, it is readily checked by definition of $\|\cdot \|_{\mathcal D_{X,\alpha }^{2\gamma '}}$ and the triangle inequality, that
		\begin{equation}
		\label{estim:D_z}
		\begin{aligned}
		\| z -\xi  ,z'-\xi '\|_{\mathcal D_{X,\alpha }^{2\gamma '}}
		&\leq \|\mathscr N,0\|_{\mathcal D_{X,\alpha }^{2\gamma '}} + \|\zeta ,\zeta '\|_{\mathcal D_{X,\alpha }^{2\gamma '}}
		\\
		&= |\mathscr N|_{0,\alpha} + |\mathscr N|_{\E^{\gamma',2\gamma'}_\alpha}+\|\zeta ,\zeta '\|_{\mathcal D_{X,\alpha }^{2\gamma '}}\;.
		\end{aligned}
		\end{equation} 
		
		\noindent In order to estimate the drift term, we note that
		\[
		\delta \mathscr N_{t,s}= (S_{t,s}-\id)\int_0^sS_{s,r}N(y_r)dr + \int_s^t S_{t,r}N(y_r)dr \,.
		\]
		The first term is easily estimated thanks to \eqref{smoothing_S}. We have indeed for $i=0,1,2$ using $N \in \Lip_{\alpha,-\delta}(\BB)$ and denoting by $\|N\|_1$ a Lipschitz constant for $N$ as a function $\BB_\alpha \to \BB_{\alpha-\delta}$
		\[\begin{aligned}
		\Big|(S_{t,s}-\id)\int_0^sS_{s,r}N(y_r)dr\Big|_{\alpha -i\gamma '}
		&\lesssim \|N\|_1 |t-s|^{i\gamma '}\int_0^s(s-r)^{-\delta }(N(0) + |y_r|_\alpha)dr
		\\
		&\lesssim_N |t-s|^{i\gamma '} T^{1-\delta }(1+ |y|_{0,\alpha })\,.
		\end{aligned}
		\]
		For the second term, we have similarly
		\begin{align*}
			\Big|\int_s^t S_{t,r} N(y_r)dr \Big|_{\alpha -i\gamma '} 
			&\lesssim \int_s^t |t-r|^{-\max\{0,\delta -i\gamma '\}}|N(y_r)|_{\alpha -\delta}dr
			\\ 
			&\lesssim_N |t-s|^{\min\{1, 1 +i\gamma '-\delta \}} (1+| y|_{0,\alpha})\,.
		\end{align*} 
		Note that thanks to our hypothesis that $1-\delta >0$ and $\gamma '<1/2$ it follows that $\kappa :=\min\{1-2\gamma ', 1 -\delta \}$ is positive and 
		\begin{equation}
		\label{drift_term}
		\max_{i=1,2,3} |\mathscr N|_{i\gamma', \alpha -i\gamma ' } 
		\lesssim_N T^\kappa (1+|y|_{0,\alpha}) \leq T^\kappa (2 + \|\xi, \xi'\|_{\mathcal D_{X,\alpha }^{2\gamma '}}) \,.
		\end{equation} 
		
		\noindent Next, note that Corollary \ref{cor:continuous} together with Lemma \ref{lem:composition}-\ref{comp:i} imply
		\begin{align}\label{estim:z}
		\|\zeta ,\zeta '\|_{\mathcal D_{X,\alpha }^{2\gamma '}} 
		&\lesssim \varrho_\gamma (\X)T^{\varepsilon} \|F(y_\cdot )-F(x), DF(y_\cdot )\circ y'_\cdot  \|_{\mathcal D_{X,\alpha -\sigma }^{2\gamma '}} \nonumber
		\\
		&\lesssim\|F\|_{\CC^2}\big(1+\varrho_\gamma (\X)\big)^{3} T^{\varepsilon} \|y_\cdot - x,y'\|_{\mathcal D_{X,\alpha }^{2\gamma '}}\nonumber \\ 
		&\leq \|F\|_{\CC^2}\big(1+\varrho_\gamma (\X)\big)^{3} T^{\varepsilon} \Big(1 + \|\xi - x, \xi'\|_{\mathcal D_{X,\alpha }^{2\gamma '}} \Big) \,.
		\end{align}

		Putting together \eqref{estim:z}, \eqref{estim:D_z} and \eqref{drift_term}, and noting that the bound on $\|\xi, \xi'\|_{\mathcal D_{X,\alpha }^{2\gamma '}}$ only depends on $F, x, \varrho_{\gamma}(\X)$ we see that there is some constant $C > 0$ which only depends on $N, F, x, \varrho_{\gamma}(\X)$ and the indices $\gamma, \gamma', \sigma, \delta$ such that:
		\begin{align*}
		\|z-\xi ,z'-\xi '\|_{\mathcal D_{X,\alpha }^{2\gamma '}} \leq C T^{\varepsilon \wedge \kappa }\,.
		\end{align*}
		Taking $T$ small enough, we see that $\MM_T(B_T(x))\subset B_T(x),$ which shows the claimed stability.\\

		\item[\em \indent Step 2: contraction property.]
Consider now $(y^j,y^{j\prime})\in B_T(x)$ and let $(z^j,z^{j\prime}):=\mathcal M_T(y^j,y^{j\prime}),$ $j=1,2.$
For every $(t,s)\in\Delta _2$, we have
\[
\begin{aligned}
z^1_t-z^2_t
&=\int_0^tS_{t,r}[N(y^1_r)-N(y^2_r)]dr + \int_0^tS_{t,r}[F(y^1_r)-F(y^2_r)]\cdot d\X_r
\\
&=\mathscr {\bar N}_t + \bar\zeta _t\,.
\end{aligned}
\]
Similarly, as above, using the Lipschitz property of $N$, we have for the drift term:
\[\begin{aligned}
|\delta \mathscr {\bar N}_{t,s}|_{\alpha -i\gamma '}
&\lesssim
\Big| (S_{t,s}-\id)\int_0^sS_{s,r}[N(y_r^1)-N(y_r^2)]dr + \int_s^t S_{t,r}[N(y^1_r)-N(y^2_r)]dr \Big|_{\alpha-i\gamma'}
\\
&\lesssim \|N\|_1 \Big(|t-s|^{i\gamma '}\int_0^s(s-r)^{-\delta }dr
+\int_s^t |t-r|^{-\max\{0,\delta -i\gamma '\}}dr\Big)\;|y^1-y^2|_{0,\alpha }
\\
&\lesssim_{N, \alpha ,|x|_\alpha, \X }T^{\kappa}\|y^1-y^2,y^{1\prime}-y^{2\prime}\|_{\mathcal D_{X,\alpha }^{2\gamma '}}\,.
\end{aligned}
\]
Next, applying again Corollary \ref{cor:continuous} and then using Lemma \ref{lem:composition}-\ref{comp:ii}, we get
		\begin{align*}
		\|\bar\zeta ,\bar\zeta '\|_{\mathcal D_{X,\alpha }^{2\gamma '}}
		&\lesssim \varrho_\gamma (\X)T^{\varepsilon} \|y^1-y^2, DF(y^1)\circ y^{\prime 1} -DF(y^2)\circ y^{\prime 2} \|_{\mathcal D_{X,\alpha -\sigma }^{2\gamma '}}
		\\
		&\lesssim\|F\|_{\CC^3}\big(1+\varrho_\gamma (\X)\big)^{3} T^{\varepsilon} \|y^1-y^1,y^{\prime 1}-y^{\prime 2}\|_{\mathcal D_{X,\alpha }^{2\gamma '}}\,.
		\end{align*}
This shows the claimed contraction property for $T\leq T_*$ small enough.\\

		Applying Banach Fixed point Theorem, we see that there exists a unique fixed point $(u,u')\in 
		\DD_{X,\alpha }^{2\gamma '}$ for $\mathcal M_T,$ and it is clearly a solution of \eqref{e:rpde3} with $u' = F(u)$. 
		Repeating the argument with $(T^1,x^1):=(T_*, u_{T_*})$ in place of $(0,x),$ we construct a sequence 
		$(T^n,x^n)$ and it is easily seen that $\tau :=\sup_{n\in\N}T^n< T$ implies 
		$\limsup_{n\to\infty}|x^n|_{\alpha }=\infty$. We then construct a solution on $[0,\tau)$ by `gluing 
		together' each solution on $[T^n,T^{n+1}].$ To show that this solution actually lives in the space 
		$\DD^{2\gamma}_{X,\alpha }$ it suffices to use the equation, together with the 
		bounds~\eqref{e:integration2} and~\eqref{e:compos1}.
		Details are left to the reader.
	\end{proof}
	\begin{rem}
		Note that the techniques of the proof do not restrict us to the range of the H\"older regularity of the 
		rough path to be $\gamma \in (1/3,1/2]$ and one could generalize the statement and the whole theory 
		above to the range $\gamma \in (1/n,1/2]$ for $n \geq 3$. However, this would require a change in the 
		definition of a controlled rough path in order to add a control of the  higher order (up to $n-2$) 
		iterated integrals of $X$ as explored in \cite{Ramification}. In turn, one would also need better smoothness properties on $F$ ($\CC^n$ 
		would do). We chose to avoid these considerations for computational convenience.
	\end{rem}

	\subsection{Continuity of the solution map and smoothing away from zero}
	Here we will briefly mention properties of the mild solution to the equation~\eqref{e:rpde2} that we 
	constructed above. We will not present any proofs here since they are simply technical modifications of 
	the analogous statements from \cite{gerasimovics2018} in the case of a constant differential operator $L$ 
	and a nonlinearity $F$ which sends $\BB_\alpha$ to itself (i.e there is no loss of space regularity from $F$).
	Our first result is going to be on the stability of the solution map~\eqref{soln_map} with respect to noise 
	and the initial condition.
	In the case of a Brownian rough path, the continuity of the solution map is one of the main advantages of 
	our pathwise formulation compared with It\^o calculus (where only measurability holds in general). We now 
	define a notion of distance between two controlled rough paths:
	
	\begin{defn}
		Let $\gamma \in (1/3,1/2]$, $0 < \sigma < \gamma' \leq \gamma$ and let $\X, \tilde{\X} \in \cC^\gamma(0,T;\R^d)$. For $(y,y') \in \DD^{2\gamma}_{X,\alpha}([0,T])$ and $(z,z') \in \DD^{2\gamma}_{\tilde X,\alpha}([0,T])$ define a distance $d_{2\gamma',2\gamma,\alpha}$ as follows
		\begin{multline*} 	
			d_{2\gamma',2\gamma, \alpha}(y,z) = |y-z|_{0,\alpha} + |y'-z'|_{0,\alpha-\gamma} 
			\\
			+| y'-z'|_{\gamma',\alpha-2\gamma} + [R^y-R^z]_{\gamma',\alpha-\gamma} + |R^y-R^z|_{2\gamma',\alpha-2\gamma},
		\end{multline*}
		where we make an abuse of notation by not writing the dependence of $d_{2\gamma',2\gamma, \alpha}(y,z)$ on $\X, \mathbf{\tilde X}$ and $y', z'$.
	\end{defn}

	\begin{thm}[Stability of the solution to subcritical RPDE] \label{stab_sol}
		Let $\gamma \in (1/3,1/2]$ and $\X,\tilde{\X} \in \cC^\gamma(0,T;\R^d)$. Let $\alpha \in \R$ and $x,\tilde{x} \in \BB_\alpha$. Assume that $F \in \CC^3_{\alpha-2\gamma,-\sigma}(\BB,\BB^d)$ for $0 \leq \sigma < \gamma$, and $N$ be as in Theorem \ref{RPDE}. Define $(u,F(u))\in\DD^{2\gamma}_{X, \alpha}([0,\tau_1))$ to be the solution to the RPDE:
		$$du_t = L_tu_tdt +N(u_t)dt + F(u_t)\cdot d\X_t\, ,\quad u_0 = x \in \BB\,;$$
		and $(z,F(z))\in \DD^{2\gamma}_{\tilde{X},\alpha}([0,\tau_2))$ to be a solution of the same RPDE but driven by the rough path $\tilde{\X}$ and initial condition $\tilde{x}$. Assume that  $\varrho_\gamma(\X), \varrho_\gamma(\tilde{\X}), |x|_{\alpha}, |\tilde{x}|_{\alpha} \leq M$. Then for every $\gamma'$ such that $0 < \sigma < \gamma' \leq \gamma$ there exists time $\tau< 1\wedge \tau_1 \wedge \tau_2$ such that for the following seminorm taken with respect to this time $\tau$ we have:
		\begin{equation} \label{e:cty}
			d_{2\gamma',2\gamma, \alpha}(u,z) \lesssim_M \varrho_\gamma(\X,\tilde{\X}) + | x-\tilde{x}|_\alpha. 
		\end{equation}
		Moreover, if both solutions do not blow up before time $T$ i.e. $(u,F(u))\in\DD^{2\gamma}_{X,\alpha}([0,T])$ and $(z,F(z))\in \DD^{2\gamma}_{\tilde{X},\alpha}([0,T])$, then \eqref{e:cty} holds on $[0,T]$.
	\end{thm}
	The proof relies on the continuity properties of the integration map and composition with the smooth 
	functions similar to Lemma~\ref{lem:composition} and Corollary~\ref{cor:continuous} modified for the 
	above metrics. The small sacrifice of time-regularity, which is represented by taking $\gamma' < \gamma$, 
	and the fact that $u$ is a fixed point of the solution map~\eqref{soln_map} then allows us to use 
	continuity properties of integration and composition to deduce for $\varepsilon = \min(\gamma-\gamma', 
	\gamma'-\sigma, 1 - \delta, 1 - 2\gamma')$:
	\[
	d_{2\gamma',2\gamma, \alpha}(u,v) \lesssim d_{2\gamma',2\gamma, \alpha}(u,v)\tau^{\varepsilon} +\varrho_\gamma(\X,\tilde{\X}) + | x-\tilde{x}|_\alpha\,.
	\]
	Therefore, taking $\tau$ small enough, we obtain the desired result. The global Lipschitz estimate on the whole $[0,T]$ in~\eqref{e:cty} can be obtained simply by iteration of the result.
	
	From classical PDE theory, it is expected that away from zero the solution is going to be infinitely smooth  
	in space because of the smoothing property of the propagator. The following proposition tells us that it is 
	indeed the case for our subcritical equations.
	
	\begin{prop} \label{smoothing}
		Let $\gamma \in (1/3,1/2]$, $0 \leq \sigma < \gamma$, and $\X \in \cC^\gamma(0,T;\R^d)$ . Let $\alpha \in \R$, $x \in \BB_\alpha$ and $F,\,N$ be as in Theorem~\ref{stab_sol} and $(u,F(u))\in\DD^{2\gamma}_{X, \alpha}([0,\tau))$ be the solution to the equation~(\ref{e:rpde2}). Denote $M_t := |u|_{0,\alpha,[0,t]}$ then for every $0<s<t<\tau$ and $\beta> 0$ we have that
		$(u,F(u)) \in \DD^{2\gamma}_{X,\alpha+\beta}([s,t])$ and there exist $\nu = \nu(\delta,\gamma,\sigma,\alpha)$, and a finite constant $C(M_t) = C(M_t,\tau,F,N,X)$ such that:
		\begin{align*} 
			|u|_{0,\alpha + \beta, [s,t]} \lesssim s^{-\beta}  |u|_{0,\alpha,[0,t]} + C(M_t)\, t^\nu\,. 
		\end{align*}
	\end{prop}
	
	\begin{rem}
		In \cite{gerasimovics2018} in the case of $F_i \in \CC^3_{\alpha-2\gamma,0}$ and the operators $L$ 
		independent of time, the authors also show that solution to RPDE~\eqref{e:rpde1} when driven by the 
		Brownian rough path coincides almost surely with the mild solution to the corresponding SPDE where 
		the integration is given now in It\^o sense. The authors also use Lipschitz continuity of the solution 
		map Theorem~\ref{stab_sol} in order to show Malliavin differentiability of the solution to rough partial 
		differential equations driven by a (regular enough) Gaussian rough path. We believe that the analogous 
		results should hold true in the case of time dependent families $L_t$ and $F_i \in 
		\CC^3_{\alpha-2\gamma,-\sigma}$ with $\sigma < \gamma$.
	\end{rem}

\subsection{Weak solutions. Proofs of Theorems \ref{thm:weak} and \ref{thm:specify} }\label{sect:weak}
	
	In this subsection, we are going to show the equivalence between mild solutions and the weak solutions 
	that were already mentioned in Section \ref{main_results}. For notational convenience 
	we will now restrict ourselves to the case when the initial condition belongs to the space $\BB_0$. With 
	this at hand, we give a notion of the weak solution:
	\begin{defn}\label{def:weak}
		Let $(\BB_\beta)_{\beta \in \R}$ be a monotone family of interpolation spaces and let $x \in \BB_0$. Let $\X,L,N,F$ be as in Theorem \ref{RPDE} and let $\nu = \max\{1, \sigma+2\gamma\}$. We say that $(u,F(u)) \in \DD^{2\gamma}_{X}([0,T], \BB_0)$ is a weak solution of the rough PDE~\eqref{evolution_equation} if for all $\varphi \in \BB^*_{-\nu}$ the following integral formula holds for all $0 \leq t \leq T$:
		\begin{equation}\label{eq:weak}
	\langle u_t,\varphi\rangle = \langle u_0,\varphi\rangle + \int_0^t \langle L_su_s,\varphi\rangle ds + \int_0^t \langle N(u_s),\varphi\rangle ds + \int_0^t \langle F(u_s),\varphi\rangle\cdot d\X_s,
		\end{equation}
		where as in Theorem~\ref{integration}, $\langle F(u_s),\varphi\rangle$ denotes the vector $(\langle F_1(u_s),\varphi\rangle, \dots, \langle F_d(u_s),\varphi\rangle)$.
	\end{defn}
	Note that all the above integrals make sense. Since $-\nu \leq -1$ and because of the dense inclusions 
	$\BB_0 \subset \BB_{-\delta} \subset \BB_{-1}$ we also get the reverse inclusions $\BB^*_{-1} \subset 
	\BB^*_{-\delta} \subset \BB^*_0$ and thus all the terms $\langle u_0,\varphi\rangle, \langle L_su_s,\varphi\rangle, \langle N(u_s),\varphi\rangle$ are 
	well-defined. Moreover, denoting by $\mathscr{D}^{2\gamma}_X(0,T;V)$ the usual controlled rough path 
	spaces in the Banach space $V$ (in the sense of \cite[Definition 4.6]{friz}) we note that $\big(F(u), 
	DF(u)F(u)\big) \in \DD^{2\gamma}_{X}([0,T], \BB_{-\sigma})$ implies that $\big(F(u), DF(u)F(u)\big) \in 
	\mathscr{D}^{2\gamma}_X(0,T; \BB_{-\sigma-2\gamma})$. Using the fact that $\varphi \in \BB^*_{-\nu} 
	\subseteq \BB^*_{-\sigma-2\gamma}$ we can easily show that 
	$$\Big(\langle F(u),\varphi\rangle, \langle DF(u)F(u),\varphi\rangle\Big) \in \mathscr{D}^{2\gamma}_X(0,T;\R^d)\;,$$
	which implies that the rough integral in \eqref{eq:weak} is well-defined.
	
	Before showing the equivalence of mild and weak solutions, we need the following technical lemma:
	
	\begin{lem} \label{lem:weak}
		Let $\X \in \cC^{\gamma}(0,T;\R^d)$ for $\gamma \in (1/3,1/2]$ . Then for every $\varphi \in \CC(0,T; \BB^*_0)$ and $(y,y') \in \DD^{2\gamma}_{X}([0,T], \BB_0)$ we have for each $0\leq t\leq T$:
		\begin{align*} 
		\int_{0}^{t} \Big\langle\int_{0}^{s}S_{r,s}y_r\cdot d\X_r,\varphi_s \Big\rangle ds = \int_{0}^{t} \int_{r}^{t} \langle S_{r,s}y_r,\varphi_s\rangle ds \,\cdot d\X_r\;.
		\end{align*}
	\end{lem}
	The proof can be carried out \textit{mutatis mutandis} as in \cite[Lemma 6.2]{gerasimovics2018}. The 
	only difference is that instead of the inner product of Hilbert space $\HH$ we use testing with the 
	functions from $\BB^*_0$, thus one should simply replace the Cauchy-Schwarz inequality with the 
	boundedness of the functional $\varphi$. With this at hand:
	
	\begin{thm} \label{weak}
		Let $(\BB_\beta)_{\beta \in \R}$ be a monotone family of interpolation spaces and let $u_0 \in \BB_0$. Let $\X,L,N,F$ be as in Theorem \ref{RPDE}. Assume that for all $\beta \in \R$ one has $L_\cdot \in \CC(0,T; \LL(\BB_{\beta},\BB_{\beta-1}))$, Then $(u,F(u)) \in \DD^{2\gamma}_{X}([0,T], \BB_0)$ is a mild solution of the rough PDE~\eqref{evolution_equation} namely it satisfies
		\begin{align*}
		u_t = S_{t,0}\,x + \int_0^tS_{t,r}N(u_r)dr + \int_0^tS_{t,r}F(u_r)\cdot d\X_r\;,
		\end{align*}
		if and only if it is a weak solution in the sense of Definition~\ref{def:weak}.
	\end{thm}
	\begin{proof}
		Without loss of generality we can assume in both cases that $x = 0$ by replacing $(u_t,u'_t)$ by $(u_t+S_{t,0}x,u'_t)$ (using $\ddh S_{\cdot,0} \,x = 0$).
		
		\noindent\textbf{Mild $\Rightarrow$ Weak}. Assume also for simplicity that $N = 0$ since dealing with the drift term is easier than with 
		the diffusion term $F$. Now let $(u,F(u)) \in \DD^{2\gamma}_{X}([0,T], \BB_0)$ satisfy for $0\leq t \leq T$ 
		$$u_t = \int_0^t S_{t,s}F(u_s)\cdot d\X_s\,.$$
		Let $\varphi \in \BB^*_{-\nu}$ be arbitrary for $\nu = \max\{1, \sigma+2\gamma\}$. Note that since $\BB_{-\nu}$ is continuously embedded in $\BB_{-1}$ then $L_\cdot \in \CC(0,T; \LL(\BB_0,\BB_{-\nu}))$ implying $L^*_\cdot\varphi \in \CC(0,T; \BB^*_0)$. Thus, testing with $L_s^*\varphi $ and integrating from $0$ to $t$ gives
		\begin{equation*}
		\begin{aligned}
		\int_0^t \langle L_su_s ,\varphi\rangle ds 
		&= \int_0^t \big\langle L_s \int_0^s S_{s,r}F(u_r)\cdot d\X_r,\varphi\big\rangle ds  \\
		& = \int_0^t  \int_r^t \langle L_s S_{s,r}F(u_r),\varphi\rangle ds \,\cdot  d\X_r \\
		& = \int_0^t \Big\langle \int_r^t L_s S_{s,r}F(u_r) ds ,\varphi\Big\rangle  \cdot d\X_r \\
		& = \int_0^t \Big\langle  \int_r^t  \frac{d}{ds}  S_{s,r}F(u_r) ds ,\varphi \Big\rangle\cdot  d\X_r \\
		&= \int_{0}^{t} \big\langle S_{t,r} F(u_r),\varphi\big\rangle d\X_r - \int_{0}^{t} \big\langle F(u_r),\varphi\big\rangle\cdot  d\X_r \\
		&= \langle u_t,\varphi\rangle  - \int_{0}^{t} \big\langle F(u_r),\varphi\big\rangle\cdot  d\X_r  \;,
		\end{aligned}
		\end{equation*}
		where we used Lemma~\ref{lem:weak} in the second equality together with the fact that $L_s^*\varphi$ is a continuous function in time with values in $\BB^*_0$, and the fourth equality is justified by \ref{P3*}.
		
		\noindent\textbf{Weak $\Rightarrow$ Mild}. The proof is similar to the standard proof for SPDEs found either in \cite{DPZ} or \cite{SPDE}. All the additional difficulties are similar to the ones in the proof of Lemma \ref{lem:transportMildFormulation} in the case of supercritical noise. 
	\end{proof}
	
	We now come back to the applications of this Theorem to the stochastic setting and address the proof of Theorem~\ref{thm:specify}.
	
	\begin{proof}[Proof of existence in Theorem~\ref{thm:specify}.]
		First, note that the multidimensional Brownian motion can be lifted almost surely to an $\omega $-dependent rough path $(B(\omega ),\B(\omega ))$ by setting for $(t,s)\in\Delta _2$:
		\[
		\B^{i,j}_{t,s}:=\int_s^t(B^i_r-B^i_{s})dB^j_r\,,
		\]
		in the sense of It\^o integration. 
		The corresponding random variable is supported in the space $\mathscr C^\gamma (0,T;\R^d)$ for any 
		$\gamma <\frac{1}{2}$.
		For details, we refer to \cite[Section 10]{friz}.
		
		Fix $\gamma \in (1/3,1/2)$ such that $1/3 < \gamma < \frac14(k-n/p)$, which is possible by assumption $k > n/p + 4/3$. With all the above requirements satisfied, we see that pathwise, the equation \eqref{react_diff} falls into 
		the framework of Theorem~\ref{RPDE}.
		Indeed, apply Theorem~\ref{RPDE} $\mathbb{P}$-almost surely, with the family $(L_t(\omega 
		))_{t\in[0,T]}$ defined as $L_t(\omega):=\nabla \cdot (a_t(\omega,\cdot)\nabla (\cdot ))$ to solve 
		equation~\eqref{react_diff} on $\DD^{2\gamma}_{B,0}$ for $\BB_0= H^{k,p}({\T^n})$. Note that we have chosen $\gamma$ so that $k > n/p + 4\gamma$. Therefore, the spaces $\BB_{\alpha} = H^{k+\alpha,p}({\T^n})$ are Banach algebras for every $\alpha > -2\gamma$. It is now easy to see that $f \in \CC^\infty_{0,0}(\BB)$ and $p_i \in \CC^\infty_{-2\gamma,0}(\BB)$ (see Section~\ref{subsec:specified}).
		We conclude that there exists a unique local solution $u(\omega)$ to \eqref{stochastic_reaction} with 
		a potential blow up time $\tau(\omega)$, such that $u \in \CC([0,\tau(\omega)),H^{k,p}({\T^n}))$.
		
		Now one can easily see that for $\varphi \in \CC^\infty({\T^n})$, the functional $\<\cdot, 
		\varphi\>_{L^2({\T^n})}$ belongs to $\BB^*_{-1}$ (here we need $\BB^*_{-1}$ since in the case of reaction 
		diffusion equations $\sigma = 0$ and therefore $\nu = \max\{1, \sigma+2\gamma\} = 1$). This implies 
		that $u$ satisfies almost surely on $\{t < \tau\}$ the integral formula:
		\begin{align*}
			\big\< u_t, \varphi \big\>_{L^2({\T^n})} = \big\< u_0, &\varphi \big\>_{L^2({\T^n})} + \int_0^t \big\< \nabla \cdot (a_t(\cdot)\nabla u_s), \varphi \big\>_{L^2({\T^n})} ds \\ 
			&+ \int_0^t \big\< f(u_s), \varphi \big\>_{L^2({\T^n})} ds + \sum_{i = 1}^d \int_0^t \big\< p_i(u_s), \varphi \big\>_{L^2({\T^n})} \cdot d\textbf{B}^i_s\,.
		\end{align*}
		From \cite[Section 9]{friz} we see that the above rough integrals coincide with the usual It\^o integral 
		almost surely since the processes $\<p_i(u_s), \varphi\>_{L^2({\T^n})}$ are adapted. This implies that the constructed solution is indeed a weak It\^o solution.
	\end{proof}

		\begin{proof}[Proof of uniqueness in Theorem~\ref{thm:specify}.] In this proof we assume that $f=0$ since working with the diffusion term only adds notational difficulty. 
			For notational simplicity, in the sequel we write $\CC(\BB_0)$ instead of $\CC([0,T],\BB_0)$, and similarly for other spaces. 
			The proof requires a localization argument: fix an arbitrary $\varrho>0$, define the stopping time
			$\tau_\varrho:=\inf\{t\in(0,\tau):|u_t|_{0}\geq \varrho\}$ and denote by $u^\varrho$ the stopped process $u^\varrho_t:=u_{t\wedge\tau_\varrho}$, $t\in[0,T]$.
			
			We will prove that for any\footnote{In fact it is enough to show this for some $\gamma \in (1/3,1/2)$, but we will prove a slightly stronger statement.} $\gamma \in (1/3,1/2)$ and every $\gamma > \varepsilon > 0$, the remainder\pagebreak
			\begin{equation*}
			\begin{aligned}
			R^{u^\varrho}_{t,s}(x) 
			&= u_t^\varrho(x) - u_s^\varrho(x) - \sum_{i = 1}^d p_i(u_s^\varrho(x))\delta B^i_{t\wedge \tau_\varrho,s\wedge \tau_\varrho}
			\\
			&=
			\int_{s\wedge \tau_\varrho}^{t\wedge\tau_\varrho}L_ru_r dr +\sum_{i=1}^d\int_{s\wedge \tau_\varrho}^{t\wedge\tau_\varrho}(p_i(u^\varrho_r)-p_i(u^\varrho_s))dB^i_r
			\,,
			\end{aligned}
			\end{equation*}
			belongs to $\CC^{2\gamma-\varepsilon}_2(\BB_{-2\gamma}) \cap \CC^{\gamma-\varepsilon}_2(\BB_{-\gamma})$, $\mathbb P$-almost surely. The uniqueness then follows from a slight modification of Theorem~\ref{thm:weak}. Indeed, since we have that the Brownian rough path $(B,\B)$ lives in $\cC^{\frac{1}{2} - \delta}$ for every $\delta>0$, $\mathbb{P}$-almost surely, then one can use some extra regularity of the Brownian rough path to show that the solution map~\eqref{soln_map} for the equation~\eqref{stochastic_reaction} has a unique fixed point in the space
			\begin{equation}\label{ex:twisted_space}
				\big\{(v,v',R^v) \in \CC(\BB_0) \times \CC^{\gamma-\varepsilon}(\BB_0) \times \big(\CC^{2\gamma-\varepsilon}_2(\BB_{-2\gamma}) \cap \CC^{\gamma-\varepsilon}_2(\BB_{-\gamma})\big)\big\}\;.
			\end{equation}
			
			The rest of the proof is split in two parts. First, we show that for every $\gamma \in (1/3,1/2)$ and $\varepsilon \in (0,\gamma)$, $u^\varrho$ is $(\gamma-\varepsilon)$-H\"older with values in $\BB_{-\gamma}$ with full probability. Then, we use this information to deduce that the two-parameter quantity $R^{u^\varrho}$ defined above lies in $\CC_2^{2\gamma-\varepsilon}(\BB_{-2\gamma})$. This, together with the assumption that the week solution $u$ belongs to $\CC(\BB_0)$, easily implies that $(u^\varrho, p(u^\varrho), R^{u^\varrho})$ belongs to the space defined in \eqref{ex:twisted_space} and therefore must be unique solution on the interval $[0,\tau_\varrho]$. Therefore, since such uniqueness is true for every $\varrho>0$ and since $u^\varrho = u$ on $[0,\tau_\varrho]$ one must have that $u$ is a unique weak It\^o solution to~\eqref{stochastic_reaction} on $[0,\tau)$.\smallskip
			\item[\indent\textit{Step 1: proof that $\mathbb P\left (u^\varrho\in \CC^{\gamma-\varepsilon}(\BB_{-\gamma})\right )=1$}.]
			Recall the notations of Theorem \ref{thm:weak}.
			Since here $\sigma=0$, we have that $\nu=\max(1,\sigma+2\gamma)=1$.
			Let $M_t:=\sum_{i=1}^d \int_0^{t\wedge\tau_\varrho}p_i(u_s)dB^i_s$, we will first show that for every $\gamma \in [0,1/2)$, $M\in \CC^\gamma(\BB_0)$, $\mathbb P$-a.s. 
			Observe by a standard density argument that \eqref{weak_Cc} holds for any $\varphi\in \BB_{-1}^*\simeq H^{-k+2,q}({\T^n})$, where here $q$ is the conjugate exponent i.e.\ $p^{-1} + q^{-1} = 1$. Fix $\varphi\in \BB_{-1}^*\subset \BB^*_0$, take $(t,s)\in\Delta_2$ and use Burkholder–Davis–Gundy inequality to obtain
			\begin{equation}
			\mathbb E|\langle\delta M_{t,s},\varphi\rangle|^m
			\leq C(p,\varrho)^m |t-s|^{m/2}|\varphi|_{\BB_{0}^*}^m\,,
			\end{equation}
			where $C$ is a constant depending on the polynomials $p_i$ and $\varrho$.
			Kolmogorov's continuity criterion implies that for any $\gamma\in [0,\frac12-\frac1m)$, there is an event $\Omega^\varphi$ of full probability such that $\langle M_\cdot(\omega),\varphi\rangle$ belongs to $\CC^{\gamma}(\R)$, for any $\omega\in \Omega^\varphi$.
			Let $(\varphi_n)_{n\in\N}$ be a sequence of elements in $\BB_{-1}^*$ such that $\{\varphi_n\}_{n\in\N}$ is dense in $\BB_{0}^*$, define $\Omega_\infty:=\cap _{n\in\N}\Omega^{\varphi_n}$ and observe that $\mathbb P(\Omega_{\infty})=1.$
			By the previous discussion, we further remark that for each $n\in\N,$ $(t,s)\in\Delta_2,$ and $\omega\in \Omega$:
			\begin{equation}
			\label{Kolmog_each_phi}
			\mathbf1_{\Omega_{\infty}}(\omega)|t-s|^{-\gamma}|\langle \delta M_{t,s}(\omega),\varphi_n\rangle| \leq \tilde C_{\varrho}(\omega)|\varphi_n|_{\BB_{0}^*}\;,
			\end{equation}
			where $\tilde C_{\varrho}$ is a random constant with $\mathbb{E} \tilde C_{\varrho}^m \lesssim_m C(p,\varrho)^m$.
			Since both sides of \eqref{Kolmog_each_phi} are continuous with respect to the $\BB_{0}^*$-norm, we conclude that the same relation holds true for each $\varphi\in \BB_{0}^*$. Taking the supremum with respect to $\varphi\in\BB_{0}^*$ and $(t,s)\in\Delta_2$ shows the desired conclusion. Taking $m$ big enough, we can guarantee in particular that $\gamma$ can be taken arbitrary on the interval $(1/3,1/2)$. 
			
			Now set $g_t:=L_tM_t = L_t \sum_{i=1}^d \int_0^{t\wedge\tau_\varrho}p_i(u_s)dB^i_s$. By the above computation and since $L_{\cdot} \in \CC(0,T; \LL(\BB_0,\BB_{-1}))$ we have $g \in \CC(\BB_{-1})$, $\mathbb P$-almost surely. To be more precise note the following estimate for every $h \in \BB_0 = H^{k,p}({\T^n})$
			\begin{equation}\label{eq:random_L}
				|L_{t} h|_{-1}  = |(1-\Delta)^{\frac{k}{2} - 1}\nabla\cdot (a_{t} \nabla h)|_{L^{p}({\T^n})} \lesssim |a_{t}|_{\CC^{k-1}({\T^n})} |h|_{H^{k,p}({\T^n})} \leq A |h|_{0}\;,
			\end{equation}
		 where we use notation $A = \sup_{t \in [0,T]} |a_t|_{\CC^{k-1}({\T^n})}$. Therefore, $|g|_{0,-1} \leq A |M|_{0,0}$. Next, it follows that $\mathbb P$-almost surely, $u^\varrho_t-M_t=v_{t\wedge \tau_\varrho}$ where $v$ is defined as a weak (hence a mild) solution of the following parabolic equation with random coefficients
			\begin{equation}\label{classical_parabolic}
			\begin{aligned}
			&\partial_t v_t - L_tv_t= g_t\quad \text{on}\enskip[0,T]
			\\
			&v_0=u_0 \in \BB_0\,.
			\end{aligned}
			\end{equation}
			Now, for each $t\in[0,T]$ we have
			\begin{align*}
			\delta v_{t,s} = (S_{t,s} - \id)S_{s,0}u_0 + \int_s^t S_{t,r} g_r dr + (S_{t,s} - \id) \int_0^s S_{s,r} g_r dr\;.
			\end{align*}
			It is easy to see that the first two terms can be bounded by $(|u_0|_0 + |g|_{0,-1}) |t-s|^{\gamma}$ in $\BB_{-\gamma}$. For the third term, for any $0 < \varepsilon < \gamma$ we have
			\begin{align*}
			|(S_{t,s} - \id)\int_0^s S_{s,r} g_r dr|_{-\gamma} &\lesssim |t-s|^{\gamma-\varepsilon} \int_0^s |S_{s,r} g_r|_{-\varepsilon} dr\\ 
			&\lesssim |t-s|^{\gamma-\varepsilon} \int_0^s |s-r|^{-1+\varepsilon} |g_r|_{-1} dr \lesssim |t-s|^{\gamma-\varepsilon} T^\varepsilon |g|_{0,-1}\;,
			\end{align*}
			where we used that $s \leq T$.
			Therefore, $v \in \CC^{\gamma-\varepsilon}(\BB_{-\gamma})$ and thus $u^{\varrho} = v_{\cdot \wedge \tau_\varrho} + M \in \CC^{\gamma-\varepsilon}(\BB_{-\gamma})$ with probability one. Moreover, using~\eqref{eq:random_L} we obtain
			\begin{equation}\label{eq:u_estimate}
				[u^\varrho]_{\gamma-\varepsilon,-\gamma} \lesssim_{T} |u_0|_0 + |g|_{0,-1} + [M]_{\gamma,-\gamma} \lesssim |u_0|_0 + (A+1)\tilde C_\varrho\;.
			\end{equation}

			\item[\indent\textit{Step 2: proof that $\mathbb P\big(R^{u^\varrho}\in \CC_2^{2\gamma-\varepsilon}(\BB_{-2\gamma})\big)=1$}.] 
			First, note two elementary bounds
			\begin{align*}
				\big|\int_{s\wedge\tau_\varrho}^{t\wedge\tau_\varrho}L_ru_r\big|_{-1} &\lesssim |t-s|A|u^\varrho|_{0,0}\;,\\  
				\big|\int_{s\wedge\tau_\varrho}^{t\wedge\tau_\varrho}L_ru_r\big|_{0}\;\; &=|\delta u^\varrho_{t,s}-\delta M_{t,s}|_{0} \lesssim |u^\varrho|_{0,0}+|M|_{0,0}\,.
			\end{align*}
			The interpolation inequality~\eqref{interpolation} implies then for any $\theta\in[0,1]$ and every $(t,s)\in \Delta_2$
			\begin{equation}\label{interp_L_u}
				\big|\int_{s\wedge\tau_\varrho}^{t\wedge\tau_\varrho}L_ru_r\big|_{-\theta}\leq C_\varrho |t-s|^{\theta} A^\theta(1 + |M|_{0,0})^{1-\theta}\,,
			\end{equation}
			with potentially different constant $C_\varrho$. 
			Using the equation for $u$ and applying~\eqref{interp_L_u} with $\theta=2\gamma$ yields for any $\varphi\in \BB_{-1}^*\subset\BB_{-2\gamma}^*$
			\begin{align*}
			\mathbb E|\langle R^{u^\varrho}_{t,s},\varphi\rangle|^m
			&\lesssim
			\mathbb E\big|\int_{s\wedge\tau_\varrho}^{t\wedge\tau_\varrho} \langle L_ru_r,\varphi\rangle dr\big|^m
			+\sum_{i=1}^d \mathbb E\big|\int_{s\wedge\tau_\varrho}^{t\wedge\tau_\varrho} \langle \delta p_i(u^\varrho)_{r,s},\varphi\rangle dB^i_r\big|^{m}
			\\
			&\lesssim_{\varrho,m,p}
			|\varphi|^m_{\BB_{-2\gamma}^*}\big[|t-s|^{2\gamma m}\mathbb E A^{2\gamma m} (1 + |M|_{0,0})^{(1-2\gamma)m}
			\\
			& \qquad\qquad\qquad+ |t-s|^{(\gamma-\varepsilon+\frac12)m}\mathbb{E}[u^\varrho]^m_{\gamma-\varepsilon,-\gamma}\big]
			\\
			&\leq |\varphi|^m_{\BB_{-2\gamma}^*} C\left(\varrho,K_m,p,T\right )|t-s|^{\big( (\gamma-\varepsilon+\frac12)\wedge 2\gamma \big)m}\,,
			\end{align*}
			where we used Burkholder–Davis–Gundy in the second inequality, the moments estimate~\eqref{eq:moments} on $A$ together with moments bounds on $|M|_{0,0}$ and the estimate~\eqref{eq:u_estimate} on $[u^\varrho]_{\gamma-\varepsilon,-\gamma}$ in the last inequality.
			
			Using a two-parameter version of Kolmogorov's criterion (see \cite{FHL20+}),
			we obtain by a similar argument as in the previous step, that
			$R^{u^\varrho}$ belongs to $\CC_2^{\beta}(\BB_{-2\gamma})$ with probability one, for any $\beta < (\gamma-\varepsilon+\frac12)\wedge 2\gamma - \frac1m$. Taking $m$ large enough guarantees $R^{u^\varrho} \in \CC_2^{2\gamma - \varepsilon}(\BB_{-2\gamma})$.
			This concludes our second step, thus finishing the proof of uniqueness.
		\end{proof}

	\section{Further examples}
	\label{sec:examples}
	
	In this section, we present two equations that can be covered by Theorem~\ref{thm:main}, and we quickly 
	explain why the required assumptions are satisfied. For simplicity only, the examples given here are 
	autonomous, namely $L_t$ is constant in time (non-autonomous versions of the following are easily seen 
	to be covered as well, we refer for instance to section \ref{subsec:specified}).
	
	For completeness, and because it was quickly discussed in the introduction, 
	we will also provide an example of an equation where the diffusion term is not subcritical, and thus our 
	framework does not apply (though existence and uniqueness were shown in previous papers, based on a 
	priori estimates).
	Nevertheless, we will show that the variational solution satisfies also a suitable mild formulation, which 
	could be useful to prove regularity results.	

We recall that, given a filtered probability space $(\Omega ,\mathcal F,\mathbb{P},(\mathcal F_t)_{t\in[0,T]})$, a fractional Brownian motion (fBM) with Hurst parameter $H \in (1/3, 1/2]$ is defined as a Gaussian process $B^H$ with covariance such that
\[
\mathbb{E}[B^H_tB^H_s] = \frac{1}{2}\big(s^{2H}+t^{2H} - |t-s|^{2H}\big)\;,
\]
which covers the standard Brownian motion if one lets $H = 1/2$.

	\subsection{Stochastic Navier-Stokes equation}
	The two-dimensional Navier-Stokes equation describes the time evolution of the velocity $u_t(x)$ of an 
	incompressible fluid on a surface (here represented by the flat torus $\T^2$ for simplicity). A 
	perturbation of the latter by a singular, finite-dimensional noise term can be written, for parameters 
	$1/3<H$ and $\sigma <H$, as
	\begin{align}\label{NS}
		du_t(x) = \Delta u_t(x) dt + \mathscr B\left(\mathcal{K}u_t(x), u_t(x)\right)dt + f(x) (-\Delta)^\sigma u_t(x) dB^H_t\;,\\
		u_0 \in L^p(\T^2,\R^2)\quad \text{with} \quad \int_{\T^2}u_0(x)dx=0\,.\nonumber
	\end{align}
	where $f$ is smooth and bounded,
	while $B^H_t$ denotes a fractional Brownian motion enhanced to a rough path $\mathbf B^H:=( 
	B^H(\omega ),\mathbb{B}^H(\omega ))\in \mathscr \CC^\gamma (0,T;\R^d)$ for every $\gamma <H$ (see 
	Theorem \ref{thm:specify}).
	The bilinear operator $\mathscr B$ is defined formally as $\mathscr B(u,w) = - (u \cdot \nabla)w$,
	while $\mathcal{K}$ is the continuous linear mapping defined in the Fourier space as
	$$\mathcal{F}(\mathcal{K}w)(k) = \frac{-\mathcal{F}(w)(k_2,-k_1)}{k^2_1+k^2_2}\,, \quad 
	k\equiv(k_1,k_2)\in\Z^2\,,$$
	$\mathcal{F}$ being the discrete Fourier transform.
	Fix $p\in(1,\infty)$ and introduce the scale of Banach spaces $\BB_\beta = H^{2\beta,p}_0(\T^2, \R^2)$ 
	which consists of the Bessel potential spaces on $\T^2$ intersected with the space of distributions $u$ such that 
	$\mathcal{F}(u)(0)=0$. It is well known (see~\cite{pazy1983semigroups}) that the Laplacian $\Delta$ generates a semigroup $(S_t)_{t \in [0,\infty)}$ on the full range $(\BB_\beta)_{\beta \in \R}$ in the sense of Definition~\ref{def:part}.
	Moreover, the non-linearity $N(u) = \mathscr B(\mathcal K u, u)$ lies in $\Lip_{0,-\delta}(\BB)$ and any $1 > \delta  > 1/2$ (we refer for instance to Section 8 in \cite{hypoelipticity} for details).
	 
	The operator $(-\Delta)^\sigma$ sends $\BB_\beta$ to $\BB_{\beta-\sigma}$ for all $\beta \in \R$, while multiplication with the smooth function $f$ is a smooth operation from $\BB_{\beta-\sigma}$ to itself. Hence, we obtain that for every initial condition $u_0 \in L^{p}(\T^2,\R^2)$ with $\int_{\T^2}u_0(x)dx=0$ (meaning $u_0 \in \BB_0$) and for almost every realization $\mathbf{B}^H(\omega )$ of the fractional Brownian Motion, there is a unique maximal $\tau(\omega ) $ and $u(\omega )$, solution of \eqref{NS} on $[0,\tau (\omega )).$
	
	For related investigations on rough Navier-Stokes equations, we also refer to \cite{hofmanova2019navier,hofmanova2020vorticity}, where a slightly different model is investigated based on a priori estimates.
	
	\subsection{Cahn-Hilliard equation}\label{subsec:CH}
	Let again $H>1/3,$ fix a dimension $n\geq 1,$ a let $\OO \subset \R^n$ be a bounded domain with smooth boundary.
	We consider the semilinear equation
	\begin{align}\label{CH}
		& du_t(x) = -\Delta\big(\Delta u_t(x)  - u^3_t(x)+u_t(x)\big) dt + \sum_{i = 1}^d f_i(x) \cdot \nabla u_t(x)dB^{H,i}_t \;,
		\\
		& u_t(x)=\Delta u_t(x)=0\quad \text{on}\quad [0,T]\times \partial\OO\,,
	\end{align}
	where $f_i \in \CC^\infty_c(\OO;\R^n)$ and the pair $(u_0,\Delta u_0)$ lies in $H^{k,p}_0(\OO)\times H^{k-2,p}_0(\OO)$ 
	and with $(k,p) \in [0,\infty)\times(1,\infty)$ chosen so that $k>\frac np \vee (2+\frac1p)$ (the second condition guarantees that the trace at $\partial\OO$ of $\Delta u_0$ is meaningful).
Because of the mixed boundary conditions, it is more convenient in this case to work within a scale constructed from the fractional powers of the bi-Laplacian.
For that purpose, we introduce the following operator on $L^p(\OO)$
\begin{equation}\label{bilaplacian_def}
(L,D(L)) = (-\Delta^2,\{u\in H^{4,p}(\OO)\colon u=\Delta u=0\text{ on }\partial\OO\})\,.
\end{equation}
Next, we introduce the scale of Banach spaces
\begin{equation}\label{scale_bilaplacian}
(\BB_\beta, |\cdot|_\beta):=\big(D(-L)^{\frac k4+\beta},|(-L)^{\frac k4+\beta}\cdot|\big)\,,\quad \text{if}\enskip \beta\geq 0\,,
\end{equation}
and if $\beta<0$ one proceeds by completion as in Example \ref{exa:hilbert}.
With the definition \eqref{bilaplacian_def}, the bi-Laplacian coincides with the square $(i\Delta_D)^2$ where $\Delta_D$ is the Dirichlet Laplacian, namely
\begin{equation*}
	(\Delta_D,D(\Delta_D))=
	\left (\Delta,H^{2,p}(\OO)\cap H^{1,p}_0(\OO)\right )\,.
\end{equation*}
This implies that, $(L,D(L))$ satisfies the hypotheses of \cite[Corollary 3.7.15]{arendt2011vector}, and we infer that for any $\beta\in\R$ it satisfies Assumption \ref{ass:L_t} with $(\mathcal X_0,\mathcal X_1)=(\BB_{\beta},\BB_{\beta-1})$.
Hence, the operator $L$ generates a semigroup on the full range $(\BB_{\beta})_{\beta\in \R}$, and moreover it is easily checked that the non-linearity $N(u) := - \Delta(-u^3+u)$ sends $\BB_0$ to $\BB_{-1/2}$ (this can be seen as a consequence of the fact that $\BB_{0}$ forms an algebra for $kp > n$, which is inherited from the same property for the Bessel potential space $H^{k,p}(\OO)$). 
	
Concerning the noise term, $B^H(\omega )$ is again lifted to a rough path as in the previous paragraph, and similarly as before $F_i(u):=f_i(x) \cdot \nabla u$ sends $\BB_\beta$ to $\BB_{\beta-1/4}$, for any $\beta \in \R.$ This means in particular that the assumptions of Theorem \ref{thm:main} are satisfied and hence we conclude existence and uniqueness of local solutions, for each $u_0\in \BB_0$.
	
	\subsection{Parabolic equations with transport noise}
	
	As in Example \ref{ex:family1}, assume that we are given a second order operator $L_t = \nabla\cdot (a(t) \nabla )$ satisfying the uniform ellipticity condition \eqref{uniform_ellipticity} and assume in addition that $a(t)$ is symmetric for all $t$. Recall that $L_t$ satisfies Assumption \ref{ass:L_t} with the scale of spaces defined as $\BB_\beta :=H^{2\beta}(\R^n).$ Consider the parabolic equation with transport noise
	\begin{equation} \label{transportEq}
		du_t(x) = L_t u_t(x) dt + \sum_{i=1}^d f_i(x) \cdot \nabla u_t(x) d\X_t^i\;,  \quad u_0 \in L^2(\R^n)\;,
	\end{equation}
	where $X = (X^1, \dots, X^d) \in \mathscr{C}^{\gamma}([0,T];\R^d)$ is a \emph{geometric} rough path with H\"{o}lder exponent $\gamma \in (\frac13, \frac12]$, and $f = (f_1, \dots, f_d)$ is a collection of smooth vector fields on $\R^d$. 
	
	Notice that this equation is `supercritical' and does not fall into the framework of Theorem \ref{thm:main}. 
	Indeed, note that  $F(u) = ( f_1 \cdot \nabla u, \dots, f_d \cdot  \nabla u)$ sends $\BB_{\beta} $ into 
	$\BB^d_{\beta  - \frac12}$, thus violating the subcriticality requirement since $\gamma < \frac12$. 
	Nevertheless, this type of rough partial differential equation with transport rough input has been investigated in the literature \cite{BG,DGHT,HH,HN} (see also \cite{H} for a similar quasilinear ansatz), relying on a priori 
	estimates and techniques from the theory of transport equations as introduced by DiPerna and Lions in 
	\cite{DiPernaLions}. The approach uses the variational formulation instead of the mild formulation of the 
	equation, i.e.\
	$$
	\langle\delta u_{t,s} , \phi\rangle = \int_s^t \langle u_r,L_r \phi\rangle dr - \sum_{i=1}^d \int_s^t \langle u_r, \Div( f_i \phi)\rangle  d\X^i_r\;,
	$$
	for $\phi \in H^3(\R^n)$, meaning that the rough integral $\int_0^{\cdot} F(u_r) \cdot d\X_r$ has to be understood as a $H^{-3}(\R^n)-$valued path. Above and below $\langle \cdot, \cdot\rangle$ denotes dual pairing of $H^{-k}(\R^n)$ and $H^{k}(\R^n)$
	
	We now show that, provided we have a variational solution to \eqref{transportEq}, this equation may be equivalently formulated using the propagators.

	\begin{lem} \label{lem:transportMildFormulation}
	Let $u$ be the solution to \eqref{transportEq} as described in \cite[Theorem 1]{HH}. Then $u$ allows for the following representation in $H^{-3}(\R^n)$
		\begin{equation} \label{transportMildFormulation}
			\ddh u_{t,s} = \sum_{i=1}^d \int_s^t S_{t,r} (f_i \cdot \nabla u_r) d\X_r^i\;,
		\end{equation}
		where $S$ denotes the propagator associated to $(L_t)_{t \in [0,T]}$.
	\end{lem}
	
	\begin{proof}
		As in \cite{BG}, we introduce the \emph{unbounded rough drivers}:
		$$
		A_{t,s}^1 \phi = \sum_{i=1}^d \, f_i \cdot \nabla \phi  X_{t,s}^i \qquad A_{t,s}^2 \phi = \sum_{i,j=1}^d  \, f_i \cdot \nabla ( f_j \cdot \nabla \phi) \XX_{t,s}^{j,i},
		$$
		for every $(s,t)\in\Delta _2.$
		Recall that the solution to \eqref{transportEq} is defined as the path $u \colon [0,T] \rightarrow L^2(\R^n)$ such that the linear functional $u^{\natural}_{t,s}$ defined for every $\phi \in H^3(\R^n)$ and $(s,t)\in\Delta _2$ by
		$$
		\langle u_{t,s}^{\natural},\phi\rangle := \langle \delta u_{t,s}, \phi\rangle - \int_s^t\langle u_r, L_r \phi\rangle dr -\langle u_s , [A_{t,s}^{1,*} + A_{t,s}^{2,*}] \phi\rangle\,,
		$$
		satisfies the estimate $|u_{t,s}^{\natural}|_{-3/2} \lesssim |t-s|^{\lambda }$ for some $\lambda > 1$ (we recall that $|\cdot|_{\alpha}:=|\cdot |_{H^{2 \alpha}(\R^n)}$).
	 
		Notice that since $a(t)$ is symmetric, then $L_t$ is self-adjoint and so is the propagator $S_{t,s}$. 
		Fix $t > 0$, $\phi \in H^3(\R^n)$ and define $f_s = S_{t,s}\phi$ so that $\delta f_{t,s} = - \int_s^t L_r f_r dr$.
		Straightforward computations give 
		\begin{align}
			\dd  \langle u,f\rangle_{t,s} & = \langle\dd u_{t,s}, f_s\rangle  + \langle u_t, \dd f_{t,s}\rangle \notag \\
			& = \int_s^t\langle u_r, L_r f_s\rangle dr  +\langle u_s, A_{t,s}^{1,*} f_s + A_{t,s}^{2,*} f_s\rangle  + \langle u_{t,s}^{\natural}, f_s\rangle - \int_s^t \langle u_t, L_r f_r\rangle dr \notag \\
			& = \langle u_s, A_{t,s}^{1,*} f_s+ A_{t,s}^{2,*} f_s\rangle  + u(f)^{\sharp}_{t,s}\,, \label{URD-Mild}
		\end{align}
		where we have defined 
		$$
		u(f)^{\sharp}_{t,s} = \int_s^t \big( \langle u_r, L_r f_s\rangle  - \langle u_t, L_r f_r\rangle \big)dr + \langle u_{t,s}^{\natural}, f_s\rangle.
		$$
		For the second term above we use the definition of $u^{\natural}$ and the continuity of $S_{t,s}$ on $H^3(\R^n)$ to get
		$$
		|\langle u_{t,s}^{\natural}, f_s\rangle| \lesssim |t-s|^{\lambda} |S_{t,s} \phi|_{3/2} \lesssim |t-s|^{\lambda} |\phi|_{3/2}. 
		$$
		For the first term, we write
		\begin{align*}
			\Big|\int_s^t \big(\langle u_r, L_r f_s\rangle  -& \langle u_t, L_r f_r\rangle \big)dr  \Big|  \leq  \int_s^t  \left| \langle\dd u_{t,r}, L_r f_r\rangle\right| dr   +  \int_s^t \left|\langle u_r, L_r \dd f_{r,s}\rangle \right| dr \\
			&  \leq |t-s|^{\lambda -1} |u|_{\lambda -1, -1/2}  \int_s^t  |  L_r f_r |_{1/2} dr + |u|_{0,0} \int_s^t \left|  L_r \dd f_{r,s} \right|_0 dr \\
			&  \lesssim |t-s|^{\lambda } |u|_{\lambda -1, -1/2} |\phi|_{3/2}  + |u|_{0,0} |t-s|^{3/2} |\phi|_{3/2}, 
		\end{align*}
		where we have used 
		$$
		|  L_r f_r |_{1/2} \lesssim |\phi|_{3/2} \quad \textrm{ and } \quad |\dd f_{r,s}|_1 = | S_{t,r} (\id - S_{r,s}) \phi |_{1} \lesssim |r-s|^{1/2} |\phi|_{3/2} .
		$$
		Now note that $f_t = \phi$, thus $\dd  \langle u,f\rangle_{t,s}  = \langle u_t, \phi \rangle - \langle u_s, S_{t,s} \phi\rangle = \langle\ddh u_{t,s}, \phi\rangle$.
		Using the fact that $S$ is self-adjoint and rewriting \eqref{URD-Mild} we get the following equality on $H^{-3}(\R^n)$ 
		\begin{align*}
			\ddh u_{t,s}  & = S_{t,s} A_{t,s}^1 u_s + S_{t,s} A_{t,s}^2 u_s  + u^{S,\natural}_{t,s}\,,
		\end{align*}
		where $u^{S,\natural}_{t,s}$ is defined on $H^3(\R^n)$ as
		$$
		u^{S,\natural}_{t,s}(\phi) = \int_s^t \big(\langle u_r, L_r S_{t,s} \phi\rangle  -\langle u_t, L_r S_{t,r} \phi \rangle\big)dr + \langle u_{t,s}^{\natural}, S_{t,s} \phi\rangle .
		$$
		By the uniqueness part of Theorem \ref{thm:sewing} this is exactly \eqref{transportMildFormulation} as constructed in Theorem~\ref{integration}.
	\end{proof}

	\bibliographystyle{plain}
	

\begin{thebibliography}{10}
		
		\bibitem{amann1984existence}
		H.~Amann.
		\newblock Existence and regularity for semilinear parabolic evolution equations.
		\newblock {\em Annali della Scuola Normale Superiore di Pisa-Classe di Scienze}, \textbf{11}(4): 593--676, 1984.
		
		\bibitem{amann1986quasilinear}
		H.~Amann.
		\newblock Quasilinear evolution equations and parabolic systems.
		\newblock {Transactions of the American Mathematical Society}, \textbf{293}(1): 191--227, 1986.
		
		\bibitem{arendt2011vector}
		W.~Arendt, C.J.~Batty, M.~Hieber and F.~Neubrander.
		\textit{Vector-valued Laplace Transforms and Cauchy Problems} (2nd edition).
		Monographs in Mathematics Vol.\ \textbf{96}, Birkh\"auser, 2011.
		
		
		\bibitem{BG}
		I.~Bailleul and M.~Gubinelli.
		\newblock Unbounded rough drivers.
		\newblock {\em Annales Math{\'e}matiques de la Facult{\'e} des Sciences de Toulouse}, \textbf{26}(4): 795--830, 2017.
		
		\bibitem{bailleul2019rough}
		I.~Bailleul and S.~Riedel. 
		\newblock Rough flows.
		\newblock{\em Journal of the Mathematical Society of Japan}, 2019.
		
		
		\bibitem{brault2019nonlinear}
		A.~Brault and A.~Lejay. 
		The non-linear sewing lemma I: weak formulation.
		{\em Electronic Journal of Probability}, \textbf{24}, 2019.
		
		\bibitem{brault2019nonlinear2}
		A.~Brault and A.~Lejay. 
		The non-linear sewing lemma II: Lipschitz continuous formulation. Preprint arXiv:1810.11988, 2018.
		
		\bibitem{brault2020nonlinear3}
		A.~Brault and A.~Lejay. 
		\newblock The non-linear sewing lemma III:  Stability and generic properties.
		\textit{Forum Mathematicum}, \textbf{1}. No. ahead-of-print. De Gruyter, 2020.

		\bibitem{brehier2018kolmogorov}
		C.E.~Br\'ehier, and A.~Debussche.
		\newblock Kolmogorov equations and weak order analysis for SPDEs with nonlinear diffusion coefficient.
		\newblock {\em Journal de Math\'ematiques Pures et Appliqu\'ees}, \textbf{119}: 193--254, 2018.
		
		\bibitem{brzezniak2012stochastic}
		Z.~Brzezniak and M.~Veraar. 
		\newblock Is the stochastic parabolicity condition dependent on $ p $ and $ q $?.
		\newblock {\em Electronic Journal of Probability},\textbf{17}, 2012.
		
		\bibitem{calderon1964intermediate}
		A.~Calder{\'o}n.
		\newblock Intermediate spaces and interpolation, the complex method.
		\newblock {\em Studia Mathematica}, \textbf{24}(2): 113--190, 1964.
		
		\bibitem{caruana2011rough}
		M.~Caruana, P.K.~Friz, and H.~Oberhauser.
		\newblock A (rough) pathwise approach to a class of non-linear stochastic partial differential equations.
		\newblock {\em Annales de l'Institut Henri Poincar\'e (C) Non Linear Analysis}, \textbf{28}: 27--46, 2011.
		
		\bibitem{choi2008locally}
		C.~Choi and J.M.~Kim.
		\newblock Locally convex vector topologies on b(x,y).
		\newblock {\em Journal of the Korean Mathematical Society}, \textbf{45}(6): 1677--1703,
		2008.
		
		\bibitem{coutin2014perturbed}
		L.~Coutin and A.~Lejay.
		\newblock Perturbed linear rough differential equations [{\'E}quations diff{\'e}rentielles lin{\'e}aires rugueuses perturb{\'e}es].
		\newblock In {\em Annales math{\'e}matiques Blaise Pascal}, \textbf{21}:103--150, 2014.
		
		\bibitem{cialdea2006criteria}
		A.~Cialdea and V.~Mazya.
		\newblock Criteria for the $L^p$-dissipativity of systems of second order differential equations.
		\newblock {\em Ricerche di matematica}, \textbf{55}(2): 73--105, 2006.
		
		\bibitem{DPZ}
		G.~Da Prato and J.~Zabczyk.
		\newblock {\em Stochastic equations in infinite dimensions}.
		\newblock Cambridge University Press, 2008.
		
		\bibitem{deya2011rough}
		A.~Deya.
		\newblock A discrete approach to Rough Parabolic Equations.
		\newblock{\em Electronic Journal of Probability,} \textbf{16}: 1489--1518, 2011.
		
		\bibitem{deya2012numerical}
		A.~Deya.
		\newblock Numerical schemes for rough parabolic equations.
		\newblock{\em Applied Mathematics \& Optimization,} \textbf{65}(2): 253--292, 2012.
		
		\bibitem{deya2012rough}
		A.~Deya, M.~Gubinelli and S.~Tindel.
		\newblock Non-linear rough heat equations.
		\newblock {\em Probability Theory and Related Fields}, \textbf{153}(1-2): 97--147, 2012.
		
		\bibitem{DGHT}
		A.~Deya, M.~Gubinelli, M.~Hofmanov{\'a}, and S.~Tindel.
		\newblock A priori estimates for rough PDEs with application to rough conservation laws.
		\newblock To appear in {\em Journal of Functional Analysis} (Arxiv preprint available at arXiv:1604.00437).
		
		\bibitem{DiPernaLions}
		R.J.~DiPerna and P.L.~Lions.
		\newblock Ordinary differential equations, transport theory and Sobolev spaces.
		\newblock {\em Inventiones mathematicae}, \textbf{98}(3): 511--547, 1989.
		
		
		\bibitem{feyel2008non}
		D.~Feyel, A.~de~La~Pradelle and G.~Mokobodzki.
		\newblock {A non-commutative sewing lemma}.
		\newblock {\em Electronic Communications in Probability}, \textbf{13}: 24--34, 2008.
		
		\bibitem{FHL20+}
		P.K.~Friz, A.~Hocquet and K.~L\^e.
		Rough Markov diffusions and stochastic differential equations.
		In progress.
		
		\bibitem{friz2014rough}
		P.K.~Friz and H.~Oberhauser.
		\newblock {Rough path stability of (semi-) linear SPDEs}.
		\newblock {\em Probability Theory and Related Fields},
		\textbf{158}(1-2):401--434, 2014.
		
		\bibitem{friz}
		P.K.~Friz and M.~Hairer.
		\newblock {\em A course on rough paths: with an introduction to regularity structures}.
		\newblock Springer, 2014.
		
		\bibitem{FV10}
		P.K.~Friz and N.B.~Victoir.
		\newblock {\em Multidimensional stochastic processes as rough paths: theory and applications}, volume \textbf{120}.
		\newblock Cambridge University Press, 2010.
		
		\bibitem{gerasimovics2018}
			A.~Gerasimovi{\v{c}}s and M.~Hairer.
			\newblock {H{\"o}rmander's theorem for semilinear SPDEs.}
			\newblock {Electronic Journal of Probability, volume \textbf{24}, 2019.}
		
		\bibitem{goldstein1985semigroups}
		J.A.~Goldstein.
		\newblock {\em Semigroups of linear operators and applications}.
		\newblock Oxford University Press, USA, 1985.
		
		\bibitem{gubinelli2004}
		M.~Gubinelli.
		\newblock {Controlling rough paths}.
		\newblock {\em Journal of Functional Analysis}, \textbf{216}: 86--140, 2004.
	
		\bibitem{Ramification}
		M.~Gubinelli.
		\newblock {Ramification of rough paths}.
		\newblock {\em Journal of Functional Analysis}, \textbf{248}, 4: 693--721, 2010.
		
		\bibitem{gubinelli2010}
		M.~Gubinelli and S.~Tindel.
		\newblock {Rough evolution equations}.
		\newblock {\em The Annals of Probability}, \textbf{38}(1): 1--75, 2010.

		\bibitem{GLT2006}
		M.~Gubinelli, A.~Lejay and S.~Tindel.
		Young Integrals and SPDEs. 
		{\em Potential Analysis}, \textbf{25}, 307–326 (2006). 
		
		\bibitem{SPDE}
		M.~Hairer.
		\newblock An introduction to stochastic PDEs.
		\newblock {\em arXiv preprint}, 2009.
		
		\bibitem{hypoelipticity}
		M.~Hairer and J.C.~Mattingly.
		\newblock A theory of hypoellipticity and unique ergodicity for semilinear stochastic PDEs.
		\newblock {\em Electronic Journal of Probability}, \textbf{16}(23): 658--738, 2011.
		
		\bibitem{H}
		A.~Hocquet.
		Quasilinear rough partial differential equations with transport noise.
		Preprint arXiv:1808.09867 (2018).
		
		\bibitem{HH}
		A.~Hocquet, M.~Hofmanov{\'a}.
		\newblock An energy method for rough partial differential equations.
		\newblock {\em Journal of Differential Equations}, \textbf{265}(4): 1407--1466, 2018.
		
		\bibitem{HN}
		A.~Hocquet and T.~Nilssen. 
		An It\^o Formula for rough partial differential equations and some applications.
		\textit{Potential Analysis}: 1--56, 2020.
		
		\bibitem{hofmanova2019navier}
		M.~Hofmanov\'a, J-M.~Leahy and T.~Nilssen. 
		On the Navier–Stokes equation perturbed by rough transport noise.
		\textit{Journal of Evolution Equations} \textbf{19}(1): 203--247, 2019.
		
		\bibitem{hofmanova2020vorticity}
		M.~Hofmanov\'a, J-M.~Leahy and T.~Nilssen. 
		On a rough perturbation of the Navier-Stokes system and its vorticity formulation.
		\textit{Annals of Applied Probabilty}. To appear.
		
		\bibitem{kato1953integration}
		T.~Kato.
		\newblock Integration of the equation of evolution in a Banach space.
		\newblock {\em Journal of the Mathematical Society of Japan}, \textbf{5}(2): 208--234, 1953.
		
		\bibitem{krylov1981stochastic}
		N.V.~Krylov, B.L.~Rozovskii.
		\newblock{Stochastic evolution equations}.
		\newblock{\em Journal of Soviet Mathematics} \textbf{16}(4): 1233--1277, 1981.
		
		\bibitem{kuehn2018pathwise}
		C.~Kuehn and A.~Neam\c{t}u.
		Pathwise mild solutions for quasilinear stochastic partial differential equations.
		\textit{Journal of Differential Equations}, 2020.
		
		\bibitem{le2020stochastic}
		K.~L\^e.
		A stochastic sewing lemma and applications.
		\textit{Electronic Journal of Probability} \textbf{25}, 2020.
		
		
		\bibitem{leon1990equivalence}
		J.A.~Le\'on.
		\newblock On equivalence of solution to stochastic differential equation with anticipating evolution system.
		\newblock {\em Stochastic analysis and applications}, \textbf{8}(3): 363--387, 1990.
		
		\bibitem{leon1998stochastic}
		J.A.~Le\'on and D.~Nualart.
		\newblock {Stochastic evolution equations with random generators}.
		\newblock {\em The Annals of Probability}, \textbf{26}(1): 149--186, 1998.
		
		\bibitem{lions1998fully}
		P.L.~Lions and P.E.~Souganidis.
		\newblock {Fully nonlinear stochastic partial differential equations}.
		\newblock {\em Comptes Rendus de l'Acad\'emie des Sciences-Series I-Mathematics},\textbf{326}(9): 1085--1092, 1998.
		
		\bibitem{lions1998fully4}
		P.L.~Lions and P.E.~Souganidis.
		\newblock Uniqueness of weak solutions of fully nonlinear stochastic partial differential equations.
		\newblock {\em Comptes Rendus de l'Acad\'emie des Sciences-Series I-Mathematics},\textbf{331}(10): 783--790, 2000.
		
		
		\bibitem{lumer1961dissipative}
		G.~Lumer and R.S.~Phillips.
		\newblock Dissipative operators in a Banach space.
		\newblock {\em Pacific Journal of Mathematics}, \textbf{11}(2): 679--698, 1961.
		
		\bibitem{lunardi2018interpolation}
		A.~Lunardi.
		\newblock {\em Interpolation theory}, volume~16.
		\newblock Springer, 2018.
		
		\bibitem{lyons98}
		T.J.~Lyons.
		\newblock {Differential equations driven by rough signals}.
		\newblock {\em Revista Matematica Iberoamericana}, \textbf{14}(2): 215--310, 1998.
		
		\bibitem{pardoux1980stochastic}
		E.~Pardoux.
		\newblock{Stochastic partial differential equations and filtering of diffusion processes.}
		\newblock{\em Stochastics.} \textbf{3}(1-4):127--167, 1980.
		
		\bibitem{pazy1983semigroups}
		A.~Pazy.
		\newblock {\em Semigroups of linear operators and applications to partial differential equations.}
		\newblock Springer-Verlag New York, 1983.
		
		\bibitem{reed1979methods}
		M.~Reed and B.~Simon. 
		\newblock {\em Methods of modern mathematical physics, Vol. III: Scattering theory.}
		\newblock New York, San Francisco, London, 1979.
		
		\bibitem{russo1993forward}
		F.~Russo and P.~Vallois.
		\newblock Forward, backward and symmetric stochastic integration.
		\newblock {\em Probability theory and related fields}, \textbf{97}(3):403--421, 1993.
		
		\bibitem{sobolevskii1966equations}
		P.E.~Sobolevskii.
		\newblock Equations of parabolic type in a Banach space.
		\newblock {\em American Mathematical Society Translations,} \textbf{49}:1--62, 1966.
		
		\bibitem{strang1968construction}
		G.~Strang.
		\newblock On the construction and comparison of difference schemes.
		\newblock {\em SIAM Journal on Numerical Analysis}, \textbf{5}(3):506--517, 1968.
		
		\bibitem{tanabe1960equations}
		H.~Tanabe.
		\newblock On the equations of evolution in a Banach space.
		\newblock {\em Osaka Mathematical Journal}, \textbf{12}(2):363--376, 1960.
		
		\bibitem{triebel1983theory}
		H.~Triebel.
		\textit{Theory of Function Spaces}.
		Springer Science \& Business Media, 1983.
		
		
		
		\bibitem{hesse2019local}
		R.~Hesse,  and A.~Neam\c{t}u.
		\newblock Local mild solutions for rough stochastic partial differential equations
		\newblock {\em Journal of Differential Equations}, 2019
		
		
	\end{thebibliography}

\end{document}